\documentclass[oneside,reqno,11pt]{amsart}

\usepackage{MyPackages}
\usepackage{MyNewTheorems} 
\usepackage{MyMathSymbols}
\usepackage{MyMathFonts}

\usetikzlibrary{shapes,snakes}

 \numberwithin{equation}{section}

\theoremstyle{plain} 
\newcommand{\thistheoremname}{}
\newtheorem*{genericthm*}{\thistheoremname}
\newenvironment{namedthm*}[1]
  {\renewcommand{\thistheoremname}{#1}%
   \begin{genericthm*}}
  {\end{genericthm*}}

\begin{document}

\tikzset{cross/.style={cross out, draw=black, minimum size=2*(#1-\pgflinewidth), inner sep=0pt, outer sep=0pt},
cross/.default={3.5pt}}
\tikzset{crossb/.style={cross out, draw=blue, minimum size=2*(#1-\pgflinewidth), inner sep=0pt, outer sep=0pt},
crossb/.default={3.5pt}}


\author[Hansol Hong]{Hansol Hong}
\address{Department of Mathematics \\ Yonsei University \\ 50 Yonsei-Ro \\ Seodaemun-Gu \\ Seoul 03722 \\ Korea} 
\email{hansolhong@yonsei.ac.kr}

\author[Hyunbin Kim]{Hyunbin Kim}
\address{Department of Mathematics \\ Yonsei University \\ 50 Yonsei-Ro \\ Seodaemun-Gu \\ Seoul 03722 \\ Korea} 
\email{hyunbinkim@yonsei.ac.kr}

\title[Morse superpotentials and blow-ups of surfaces]{Morse superpotentials and blowups of surfaces}

\begin{abstract}
We study the Landau-Ginzburg mirror of toric/non-toric blowups of (possibly non-Fano) toric surfaces arising from SYZ mirror symmetry. Through the framework of tropical geometry, we provide an effective method for identifying the precise locations of critical points of the superpotential, and further show their non-degeneracy for generic parameters. Moreover, we prove that the number of geometric critical points equals the rank of cohomology of the surface, which leads to its closed-string mirror symmetry due to Bayer's earlier result.  
 \end{abstract}
\maketitle

\setcounter{tocdepth}{1}
\tableofcontents



\section{Introduction}

There have been extensive studies on the effects of birational changes of the space on various quantum invariants. In particular, it appears to be a general phenomenon that a blowup at a point for a generic parameter introduces an additional 'independent' factor to the existing invariants prior to the blowup, such as a semi-simple field factor to the quantum cohomology \cite{Bayer} or a (semi-)orthogonal summand consisting of a single object in a semi-orthogonal decomposition of the derived category \cite{Orlov}, \cite{VXW}, etc.

In this paper, we investigate how blowups affect Lagrangian Floer theory of SYZ fibers in a class of complex surfaces, while establishing related mirror symmetry statements. Our focus centers on complex surfaces obtained via blowups on toric surfaces at distinct points, allowing for blowup centers in generic positions of the toric divisor (not fixed by the torus action). Such surfaces are, of course, no longer toric, and often referred to as the non-toric blowups of toric surfaces. They are important building blocks of log Calabi-Yau surfaces as they serve as toric models \cite{GHKK}, and recently their mirror symmetry questions were intensively studied, for e.g., by \cite{HK}. Amongst them, the approach of \cite{BCHL} pursues the Strominger-Yau-Zaslow (SYZ in short) aspects of the mirror construction of log Calabi-Yau surfaces in the realm of Lagrangian Floer theory, which will be taken as the basic geometric setup of this paper.

More specifically, \cite{BCHL} demonstrates that the complement of an anticanonical divisor in a non-toric blowup $X$ admits a special Lagrangian fibration away from a small neighborhood of the nodal fibers. The presence of singular fibers gives rise to complicated wall-crossing phenomena, due to the existence of Maslov $0$ holomorphic disks emanating from these fibers. The wall structure of the fibration can be explicitly computed using what is so called scattering diagrams, drawn on the base of SYZ fibration. Scattering diagrams are comprised of combinatorial data formed by affine rays (called \emph{walls}) of rational slopes satisfying a certain consistency. Each wall carries a cluster transformation, recording how the count of holomorphic disks changes when crossing over rays. We shall see in Lemma \ref{lemma:centralchamber} that in the case of our interest, the scattering diagram has an open chamber $R_0$ on which the count of Maslov $2$ holomorphic disks (of small energies, in non-Fano situation) remains constant. Through holomorphic/tropical correspondence, tropical disks can be used as substitutes for holomorphic disks. These are piecewise linear objects within the scattering diagram, which helps us perform our analysis on the mirror side efficiently in a combinatorial manner.

Provided that the torus fibers are weakly unobstructed\footnote{This is automatic by the degree reason when $X$ is semi-Fano, and will be shown in Lemma \ref{lem:torusweakunobs} for general non-Fano cases.}, the count of Malsov $2$ disks with boundaries on such fibers gives rise to a Landau-Ginzburg (LG in short) model, a two-variable Laurent series $W$ with $\Lambda$-coefficients called the \emph{potential}. We will mostly stick to the potential computed over a preferred chamber $R_0$ in the scattering diagram. Although the expression of $W$ differs from one chamber to another, we can pretend that the same expression continues throughout other chambers by analytic continuation as in \ref{subsubsec:globalizecoord} (with some interesting subtlety, see \ref{subsec:geomnon-toric}).

We remark that while the Laurent series $W$ may even be extended to a function on $(\Lambda \setminus \{0\})^2$ in this way, it is natural to exclude the region beyond the base of the fibration. Indeed, outer points are not geometrically meaningful in the sense that they do not admit geometric representatives on the symplectic side. To the contrary, recall that each point in the actual domain of $W$ corresponds to a torus fiber equipped with a local system along the original spirit of SYZ \cite{SYZ96}, or more recently family Floer theory \cite{Ab-famFl1, Yuan}. We will also consider its bulk-deformation $W^\mathfrak{b}$ by $\mathfrak{b} \in H^{even} (X;\Lambda_+)$ subject to a certain technical condition.\\

One of the primary goals of this paper is to establish that the potential $W$ acquired through this method is the genuine mirror of the non-toric blowup $X$ (possibly non-Fano) by proving close-string mirror symmetry in this context. Namely, we compare the (big) quantum cohomology $(QH^\ast_\mathfrak{b}, \star_\mathfrak{b})$ of $X$ and the Jacobian ideal ring of the bulk-deformed potential $W^\mathfrak{b}$. Due to Bayer \cite{Bayer}, it is known that $QH^\ast_\mathfrak{b} (X)$ is semi-simple for a generic K\"{a}hler form and $\mathfrak{b}$, i.e. completely decomposed into field factors. Thus our objective becomes verifying that the count of non-degenerate critical points of $W^\mathfrak{b}$ matches the rank of $QH^\ast_\mathfrak{b} (X)$. However, when naively treated as a function on $(\Lambda \setminus \{0\})^2$, the Laurent series $W^\mathfrak{b}$ typically shows an excessive number of critical points, hence it is crucial that we restrict our analysis to a more confined 'geometric' region, as mentioned above. See Definition \ref{def:geomcrit} for more details. This point has already appeared in earlier works such as \cite{GW}, \cite{FOOO10b}. 

Interestingly, ruling out non-geometric critical points is necessary only after higher Maslov index disks start nontrivially contributing to the potential (which can happen due to sphere bubbles with negative Chern numbers or $\mathfrak{b}$ with high codimension).
If the surface remains (semi-)Fano, the explicit expression of the full potential $W$ is available, and is completely determined by the combinatorial data of the toric fans and the locations of the blowup centers. The critical points in this case can be easily counted using the classical theorem of Kushnirenko (see Theorem \ref{Kushnirenko}). In \cite[Section 7]{BCHL}, the locations of critical points of $W$ were calculated through direct computation for degree $5$ del Pezzo surface, that is one of the simplest Fano cases involving non-toric blowups.\\

Our main theorem is the following:
\begin{namedthm*}{Theorem I}[Theorem \ref{main}]\label{thm:intromain}
Let $X$ be a non-toric blowup of a toric surface $X_{\Sigma}$, and let $W^{\mathfrak{b}}$ be the the bulk-deformed potential of $X$, defined on the SYZ mirror $\check{Y}$ of the anticanonical divisor complement $Y:=X\setminus D$. Then for a generic K\"{a}hler form and $\mathfrak{b} \in H^{even} (X;\Lambda_+)$, we have
\[
	\mathrm{Jac} (W^{\mathfrak{b}}) = QH^{*}_{\mathfrak{b}}(X)
\]
where the Jacobian ideal ring on left hand side is semi-simple, or equivalently, $W^{\mathfrak{b}}$ is Morse. Here $\mathfrak{b}$ is required to be the pull-back of a torus-invariant cycle in $X_\Sigma$ when $X$ is beyond semi-Fano.
\end{namedthm*}

One of the novel points is that we cover non-Fano cases, where the complete formula of the potential is not available due to nontrivial sphere-bubble contributions. In Proposition \ref{prop:pindownclasses}, we carefully estimate the energy of such bubbled-off contributions, and show that they cannot serve as (symplectic-)energy minimizers at a geometric critical point of $W$. This observation is crucial since, under certain non-degeneracy conditions, the energy minimizers almost completely determine the critical point information of the potential. 
In the same spirit, for a generic K\"{a}hler form, the terms with nontrivial bulk-parameters fall within the higher-energy part of the potential, and do not influence the critical point calculation. We remark that for some special K\"{a}hler forms, however, the locations of critical points can depend on bulk-parameters, resulting in a\textit{ continuum of critical points}, as pointed out in \cite{FOOO-T2}. See Subsection \ref{subse:continuum} for a detailed example.

When non-degeneracy of leading terms is not guaranteed, the analysis becomes much harder, and handling this issue takes up the most delicate point of our argument. In this case, the leading (first order) term calculation is not enough to determine the critical points, and we have to further expand the potential to the second (or higher) order and carry out a complicated energy estimate on the expansion to find out critical points that are stable under the blowups. A detailed proof is given in Appendix \ref{APPENDIX}.

The remaining part of the proof is the tropical geometric visualization of the critical behavior of the leading terms of $W$ followed by an induction on the energy of holomorphic disks, an idea that essentially goes back to \cite{FOOO-T, FOOO-T2, FOOO10b} and \cite{FW}. The advantage of tropical geometry, under non-degeneracy conditions, is that it confines all critical points to vertices of the tropicalization of the leading term potential (Corollary \ref{cor:Wmin}). At these vertices, the problem reduces to applying the Kushnirenko Theorem to a handful of monomials. If non-degeneracy assumptions fail, critical points can occur over edges of the tropicalization, which requires a further analysis on the next order terms as mentioned above.
In either cases, the tropicalization can be easily calculated through a simple combinatorial procedure on the Newton polytope.

The effect of an individual point-blowup can be singled out from \textbf{Theorem I}. In Sections \ref{subsec:toricblowup1} and \ref{subsec:non-toriclocal}, we carefully extract and keep track of geometric critical points as we perform toric/non-toric blowups, while demonstrating that all such critical points are Morse critical points. In particular, the following is a parallel statement (upon mirror symmetry) to the result of Bayer \cite[Theorem 3.1.1]{Bayer} in our particular geometric situation. The toric blowup case below is particularly related to \cite{GW}, yet we allow $X$ to be non-Fano, and hence $W$ can be significantly different from the Hori-Vafa potential, accordingly.

\begin{namedthm*}{Theorem II, III}[Corollary \ref{cor:toricnum}, Proposition \ref{thm:non-toricblowup}]\label{prop:blowupcrit}

Let $(X_{\Sigma}, D_\Sigma)$ be a toric surface with the mirror potential $W_{\Sigma}$. Consider the toric surface $X_{\widetilde{\Sigma}}$ obtained by taking a sequence of toric blowups of $X_\Sigma$, and the surface $X$ obtained by a sequence of non-toric blowups of $X_\Sigma$. Denote by $W_{\widetilde{\Sigma}}$ and $W$ for the mirror potential of $X_{\widetilde{\Sigma}}$ and $X$, respectively. Then generically:

\begin{enumerate}
\item[(1)]  There exists $r>0$ depending on $\omega$ such that if the sizes of exceptional divisors are smaller than $r$, then $W_{\widetilde{\Sigma}}$ has as many new non-degenerate critical points as the number of exceptional divisors for generic parameters. Moreover, every critical point of $W_{\Sigma}$ are extended to that of $W_{\widetilde{\Sigma}}$ in a valuation-preserving manner. If $W_{\Sigma}$ is Morse, then $W_{\widetilde{\Sigma}}$ is also Morse. 
\item[(2)] Each non-toric blowup point gives rise to a unique new geometric critical point of $W$ which is non-degenerate.
If the blowup point lies in $D_{\Sigma,j+1}$, then the corresponding new critical point is located near the corner $D_{\Sigma,j} \cap D_{\Sigma,j+1}$, while every critical point of $W_{\Sigma}$ are extended to that of $W$ in a valuation-preserving manner.  If $W_{\Sigma}$ is Morse, then $W$ is also Morse. 
\end{enumerate}\vspace*{-10pt}

Analogous statements of (1) and (2) also hold for the bulk-deformed potential $W^{\mathfrak{b}}$ (in place of $W_{\widetilde{\Sigma}}$ and $W$, respectively) for a generic bulk parameter $\mathfrak{b}$ given as in \bf{Theorem I}. \end{namedthm*}

The first part of the statement concerns only toric geometry, and can be shown alternatively by appealing to the toric mirror symmetry of \cite{FOOO10b} and related facts about the quantum cohomology of $X_\Sigma$. Notice that the above theorem implies that the mirror potential $W$ (and $W^{\mathfrak{b}}$) of $X$ is always Morse for generic parameters.

In view of the tropicalization of $W$ (or that of its leading order terms), the new critical point lies over the vertex that is created when one of the unbounded edge of the tropicalization of $W_\Sigma$ branches into multiple edges as in Figure \ref{branching}.\\

Finally, we explore homological mirror symmetry aspects of non-toric blowups of toric surfaces. Due to lack of generation result for the Fukaya category in our geometric context, we consider the sub-Fukaya category $\mathcal{F}_0 (X)$ generated by Lagrangian torus fibers of the SYZ fibration on $X$. Conjecturally, critical fibers from non-toric blowups should generate the corresponding eigenvalue components (with respect to $c_{1}(X) \star -$) of the genuine Fukaya category, based on \cite[Corollary 1.12]{She} which is currently only valid under the monotone assumption. In fact, we shall see that critical points of $W$ have mutually distinct critical values, generically.

$\mathcal{F}_0 (X)$ admits the following simple structural decomposition.
First of all, we see that the torus fibers corresponding to critical points of $W$ have nontrivial Floer cohomologies by the same argument as in  \cite{CO, FOOO}. Therefore these fibers (together with suitable $\Lambda_U$-local systems) form nontrivial objects in $\mathcal{F}_0 (X)$. In particular, these fibers cannot be displaced from themselves by any Hamiltonian diffeomorphism, and we can explicitly locate all non-displaceable fibers from the tropicalization of $W$ (Section \ref{sec:tropcrit}).
We additionally show in Lemma \ref{prop:hffibers} that two different critical fibers have trivial Floer cohomology between them (this fact must be already well-known),
and that the endomorphism of a critical fiber is quasi-isomorphic to the Clifford algebra associated with the Hessian of $W$ at the corresponding point. The latter part uses the similar argument to \cite{cho05}.

\begin{remark}\label{rmk:nondisp}
For the purpose of locating a non-displaceable fiber, it is enough to find a single bulk-parameter $\mathfrak{b}$ that gives a nontrivial Floer cohomology, rather than generic $\mathfrak{b}$. For instance, this is the reason why no special treatment is needed in \cite[Theorem 4.7]{FOOO-T2} for the degenerate case whereas we deal with it separately in Proposition \ref{prop:gnc}.
\end{remark}

In view of \textbf{Theorem II, III}, we obtain a new object of $\mathcal{F}_0(X)$ each time we perform a non-toric blowup, and this new Lagrangian brane appears at a specific location. By comparing local models, we believe this object corresponds to what is called the \emph{exceptional brane} in \cite{VXW}. Interestingly, depending on the location of blowup center, the new brane can be supported on the nodal fiber (its small perturbation), regarded as an immersed Lagrangian $S^2$ boundary-deformed by immersed generators, which reveals necessity of including immersed branes in the Fukaya category. See \ref{subsec:geomnon-toric} for more details.

On the other hand, the mirror $B$-model category has exactly the same description as above. By \textbf{Theorem I}, we have that the singularity category of $W$ decomposes into skyscraper sheaves, or more precisely into their images in the quotient category by perfect complexes. It is well-known that their morphism spaces in the singularity category show the same feature (see \cite{Dy}, \cite{Tel}) for e.g.), which leads to its equivalence to $\mathcal{F}_0 (X)$.

\begin{thm}[Theorem \ref{thm:HMS}]
Let $X_\Sigma$ be a toric surface, and $X$ a non-toric blowup of $X_\Sigma$. Suppose that every exceptional divisor from non-toric blowups has small enough symplectic volume. If $\mathcal{F}_0 (X)$  denotes the sub-Fukaya category generated by torus fibers in $X$, then there is an equivalence between $D^b \mathcal{F}_0 (X)$  and $\oplus_\lambda D^{b}_{sing} (W^{-1} (\lambda)) (\cong MF(W))$ for generic K\"{a}hler forms.

Both categories admit orthogonal decomposition with respect to critical (potential) values $\lambda$, and each summand in the decomposition is generated by the skyscraper sheaf at the unique critical point whose value is $\lambda$ or its corresponding SYZ fiber.
\end{thm}
The other direction of homological symmetry was studied in the work of Hacking-Keating \cite{HK}.  \\

Along the way, we develop a technique for analyzing critical points of a Laurent series $W(z)$ over $\Lambda$, which involves two steps: (i) localizing the series at a certain valuation level $\val(z)=c$ by classifying possible energy minimizing terms, and (ii) inductively solving the critical point equation order-by-order. For step (ii), we may need to look at the second order terms (or higher, depending on the non-degeneracy of first order terms), which is a new feature that has not been addressed in existing literature.
We believe that this method can be applied in more general situations including the mirror potentials induced from almost-toric fibrations as well as higher dimensional toric manifolds, which will be left for future investigations.
\\

The paper is organized as follows. In Section \ref{sec:SYZ}, SYZ mirror symmetry of (non-toric) blowup of toric surfaces will be reviewed. Along the way, basic geometric setup will be provided. In Section \ref{sec:tropcrit}, we look into the tropicalization of a given Laurent polynomial $W$ over $\Lambda$ and its connection with critical points of $W$. In Section \ref{sec:singLG}, we analyze how critical points of the mirror potential change under blowups and prove our main theorem. Finally, homological mirror symmetry aspects of blowups will be studied in Section \ref{sec:HMS}. 
%
%
%

\subsection*{Notations}
Throughout we use the following notations:
\begin{align*}
&\Lambda:= \left\{ \sum_{i=0}^{\infty} c_i T^{\lambda_i}: c_i \in \mathbb{C},\, \lim_{i\to \infty} \lambda_i = \infty \right\},
\\
&\Lambda_0:= \left\{ \sum_{i=0}^{\infty} c_i T^{\lambda_i} \in \Lambda: \lambda_i \geq 0 \right\}, \\
&\Lambda_+:= \left\{\sum_{i=0}^{\infty} c_i T^{\lambda_i} \in \Lambda: \lambda_i > 0\right\}, \\
&\Lambda_U:=\mathbb{C}^\ast \oplus \Lambda_+ .
\end{align*}
We define $\val  : \Lambda \to \mathbb{R}$ by
\[
	\val:  \sum_{i=0}^{\infty} c_i T^{\lambda_i} \mapsto \min_i \{\lambda_i: i=0,1,2,\cdots\}.
\]
For given a Laurent polynomial (or series) $W=  \sum_{v \in \mathbb{Z}^{n}} \alpha_{v}z^{v}$, we write
\begin{align*}
&\supp W:=\{\,  v \in \mathbb{Z}^{n}   \;\, \rvert \;\,  \alpha_{v} \neq 0  \,\},\\
&\Delta_W:=\mbox{the Newton polygon of}\,\,W.
\end{align*}
For a toric manifold $X_\Sigma$ constructed from a fan $\Sigma$, we denote the generators of $1$-cones of $\Sigma$ by
\[
	\nu_{1},\cdots , \nu_N
\]
and their corresponding toric divisors by
\[
	D_{\Sigma}:=D_{\Sigma,1} \cup \cdots \cup D_{\Sigma,N}.
\]
We write $\Delta_\Sigma$ for its moment polytope.
Finally, for a disk class $\beta \in H_{2} (X_\Sigma, L)$, we define 
\[
	\delta(\beta) := \omega(\beta) \,\,\mbox{for}\,\,L\,\, \mbox{the moment fiber over the origin of the moment polytope}.
\]


\begin{center}
{\bf Acknowledgement}
\end{center}
We thank Yoosik Kim, Yu-Shen Lin, Matt Young, Arend Bayer for valuable discussions. 
  The work of the first named author is supported by the National Research Foundation of Korea (NRF) grant funded by the Korea government (MSIT) (No. 2020R1C1C1A01008261 and No.2020R1A5A1016126).



\section{The SYZ mirror of a log Calabi-Yau surface}\label{sec:SYZ}

Our main object is to study the singularity information of the Laurent series that emerges as the mirror of a non-toric blowup of toric surfaces. Non-toric blowup is a crucial step in the toric model of a log Calabi-Yau surface, which enables us to understand its geometric structure through toric geometry. It is known that any log Calabi-Yau surface is isomorphic to some non-toric blowup, with potential modifications of a boundary (anticanonical) divisor. This allows us to follow Auroux's program \cite{auroux07} and construct a special Lagrangian fibration on the anticanonical divisor complement of the surface, where the associated holomorphic disk counting leads to the LG model mirror to the surface. 
In this section, we will review the construction of this LG mirror. In particular, the SYZ mirror can be calculated entirely by tropical geometric terms due to \cite{BCHL}.

\subsection{Lagrangian Floer theory}\label{subsec:LFT}
We begin with a more general situation of Lagrangian torus fibration. Consider a special Lagrangian fibration $\varphi : X \setminus D \to B$ on the complement $X \setminus D$ of a anticanonical divisor $D \subset X$. Then $B$ admits two distinguished system of coordinates as follows.
The complex affine coordinates on $B$ are given by
\[
	 u_i (u) := \int_{C_i} \mathrm{Im}\, \Omega \qquad i=1,\cdots, n
\]
at $u \in B$ where $C_i$ is a $n$-dimensional chain swept out by $f_i$ along a path from the fixed reference point $u_0$ in $B$ to $u \in L$, and $\{e_{1},\cdots, e_{n}\}$ is a chosen basis of $H_{n-1} (L_u;\mathbb{Z}) =H^1 (L_u;\mathbb{Z})$. On the other hand, the symplectic affine coordinates are given by
\begin{equation}\label{eqn:sympaff} 
x_i (u)=  \int_{A_i} \omega \qquad i=1,\cdots,n
\end{equation}
where $A_i$ is a cylinder analogously obtained from a chosen basis $\{f_{1},\cdots, f_{n}\}$ of $H_{1} (L_u;\mathbb{Z})$. For convenience, we assume $\{f_i\}$ and $\{e_j\}$ are dual to each other, $(f_i,e_j) = \delta_{ij}$.

Following \cite{auroux07}, the SYZ mirror of $X$ is given by first taking dual torus fibration $\check{Y}^\mathbb{C}$ of $\varphi$ and equipping it with a Laurent series $W^\mathbb{C}$ called the \emph{potential}, determined by the count of holomorphic disks that intersect $D$. Identifying $\check{Y}^\mathbb{C}$ as 
\[
	\check{Y}^\mathbb{C} = \{(L_u:=\varphi^{-1} (u),\nabla) : u \in B, \nabla \in \Hom (H_{1}(L_u), \mathrm{U}(1)) \}.
\]
The potential $W^\mathbb{C} :\check{Y}^\mathbb{C} \to \mathbb{C}$ can be then written as
\begin{equation}\label{eqn:Wanticancc}
W^\mathbb{C} (L_u,\nabla) = \sum_{\beta \in \pi_{2} (X,L_u), \mu(\beta)=2} N_\beta  e^{ - \int_\beta \omega} hol_{\partial \beta} \nabla,
\end{equation}
where $N_\beta$ is the number of holomorphic disks bounding $L_u$ in class $\beta$ (passing through a generic point of $L_u$). Assuming Fano condition, $N_\beta$ counts disks intersecting $D$ exactly once, and \eqref{eqn:Wanticancc} turns out to be a finite sum. In general, one needs to introduce the non-Archimedean valuation ring $\Lambda$, and substitute $T=e^{-1}$, i.e., 
\begin{equation}\label{eqn:Wanticanll}
W (L_u,\nabla) = \sum_{\beta \in \pi_{2} (X,L_u), \mu(\beta)=2} N_\beta  T^{  \int_\beta \omega} hol_{\partial \beta} \nabla,
\end{equation}
which always converges in the $T$-adic topology. Introducing local coordinates 
\begin{equation}\label{eqn:mirrorzi}
z_i=T^{x_i} hol_{f_i} \nabla,
\end{equation}
the potential $W$ can be written as 
\[
	W(z_{1},\cdots, z_{n})=\sum N_\beta T^{\int_{A_{\partial \beta}} \omega}  z^{\partial \beta},
\]
where $z^{\gamma} := z_{1}^{(\gamma,e_{1})} \cdots z_{n}^{(\gamma,e_{n})}$ for $\gamma \in H_{1} (L_u;\mathbb{Z})$, and $T^{\int_{A_{\partial{\beta}}} \omega}$ is the flux between $L_{u_0}$ and $L_u$ defined similarly to \eqref{eqn:sympaff}.

Alternatively, one can obtain $W$ via family Floer theory of the SYZ fibration. In this case, $W$ is constructed as a gluing of the fiberwise Lagrangian Floer potentials $W_u$ defined implicitly by $\sum_k m_k (b,\cdots, b) = W_u (b) \cdot [L_u]$. Here, $b = \sum y_i e_i$ for a chosen basis $\{e_i\}$ of $H^1 (L,\mathbb{Z})$ and $y_i( \in \Lambda_0)$ are subject to weakly unobstructedness condition. This is related to the previous formulation \eqref{eqn:Wanticancc} by
\begin{equation}\label{eqn:resttofiberwiseW}
W_u (\uz_{1}, \cdots,\uz_{n}) =  W (L_u, \nabla^{(\uz_{1},\cdots, \uz_{n})})
\end{equation}
for $(\uz_{1}, \cdots, \uz_{n}) := (e^{y_{1}},\cdots, e^{y_{n}}) \in \Lambda_U$ where $\nabla^{(\uz_{1},\cdots,\uz_{n})} \in \Hom (H_{1}(L_u) , \Lambda_U)$ is a flat connection having $\uz_i$ as a holonomy along the Poincar\'e dual of $e_i$ ($W$ in \eqref{eqn:Wanticanll} is extended to allow $\Lambda_U$-connections, which does not create any problem). 

In the Family-Floer perspective, the mirror $\check{Y}$ should be replaced by the rigid analytic variety $\check{Y}$. It is still fibered over $B$, but with $T^n$-fibers in $\check{Y}$ replaced by $(\Lambda_U)^n$. The previous complex coordinates are extended so that $(L_u, \nabla^{(\uz_{1},\cdots, \uz_{n})})$ has coordinates on $\check{Y}$ given as $z_i = T^{x_i(u)} \uz_i$.
We will denote the resulting fibration (SYZ-dual to $X \setminus D \to B$) by $\check{\varphi} :\check{Y} \to B$. We will mostly write $\val$ for $\check{\varphi}$ in the main application, since $\check{\varphi}$ can be identified with the restriction of the map $\val:(\Lambda^\times)^n \to \mathbb{R}^n$ when written in terms of coordinates $z_i$. See \cite{Yuan} for a detailed construction of the rigid analytic mirror in the realm of family Floer theory.

\begin{example}\label{ex:toric}
Let $X_\Sigma$ be an $n$-dimensional toric Fano manifold with the toric fan $\Sigma$ whose primitive rays are generated by integral vectors $\nu_{1}, \cdots,  \nu_N$, and $\Delta_\Sigma$ the moment polytope of $X_\Sigma$ (which determines the K\"{a}hler form on $X_\Sigma$) given as
\[
	\langle x, \nu_i \rangle \geq -\lambda_i \quad i=1, \cdots, N.
\]
For later use, we denote by $D_{\Sigma,i}$ the (irreducible) toric divisor associated with $\nu_i$.
In this case, the mirror $\check{Y}$ can be identified as a subset of $\Lambda^n$ consisting of elements whose valuations lie in the interior $B$ of $\Delta$, and the dual fibration can be identified with $\val : \check{Y} \to B$. 

\sloppy There is a one-to-one correspondence between the free $\mathbb{Z}$-module generated by $\{\nu_{1},\cdots,\nu_N\}$ and $H_{2} (X_\Sigma, L)$ for $L$ a toric fiber, and hence we have the corresponding basis $\{\beta_{\nu_{1}},\cdots, \beta_{\nu_{2}}\}$ of $H_{2} (X_\Sigma,L)$. Their boundaries are precisely given as $\nu_{1},\cdots,\nu_N$ in $H_{1}(L)$.                                                                                                                                                                                 The classification result of \cite{CO} tells us that each $\beta_{\nu_i}$ can be represented by a Maslov index $2$ holomorphic disk intersecting $D_{\Sigma,i}$ exactly once.
Consequently, each primitive ray (hence each toric divisor) contributes a monomial $z^{\partial \beta_{\nu_i}}=z^{\nu_i}$ to the potential $W_\Sigma$ on $\check{Y}$ according to the classification of Malsov 2 disks. (Here, $\partial \beta_{\nu_i}$ is viewed as a class in $H_{1}(L,\mathbb{Z})$ for a toric fiber $L$.) These disks are usually called the \emph{basic disks}.
If $X_\Sigma$ is semi-Fano, then there can be additional sphere bubble contributions \cite{chan-lau}, and not much is known about the precise computation of the potential beyond this case.
\end{example}

In the presence of singular fibers, the count $N_\beta$ in \eqref{eqn:Wanticancc} shows a certain discontinuity, which results in the wall-crossing of $W(L_u,\nabla)$. In our main applications below, walls are given as the union of affine lines or rays in the base $B$ with respect to the complex affine structure, and $N_\beta$ remains constant as long as $u$ does not go across one of these walls. When going across the wall, $W$ is changed by a certain cluster-type transformation whose rough shape appears in \eqref{eqn:wctrans}. In general, a wall appears as an affine line segment in $B$  of an integral slope $\gamma$ such that if $u$ is on this wall, $L_u$ bounds Maslov $0$ disks whose boundary class is $\gamma$. (More precisely, the algebraic count of such disks is nonzero for $L_u$.)

In Family Floer perspective, having discontinuity of $W$ accounts for gluing different local charts of the mirror rigid analytic variety by nontrivial coordinate transitions which link local $W_u$'s. The transition map can be actually computed by comparing Fukaya $A_\infty$-algebras (especially their weak bounding cochains) of two nearby fibers $L_u$ and $L_{u'}$ in the adjacent charts via pseudo-isotopies (as known as \emph{Fukaya's trick}). It induces an $A_\infty$-quasi isomorphism between $CF(L_u,L_u)$ and $CF(L_{u'},L_{u'})$, which determines the coordinate change between $W_u$ and $W_{u'}$ given in the form of
\begin{equation}\label{eqn:wctrans}
 W_{u} (z_{1},\cdots, z_{n})= W_{u'} (z_{1}',\cdots, z_{n}'),\qquad z_i' =  T^{\epsilon_i} z_i (1 + f_i) \,\, i=1,\cdots,n
\end{equation}
where $f_i \cong 0$ modulo $\Lambda_+$ and $\epsilon_i$ is a flux that limits to zero as $u$ and $u'$ get closer to each other. We will have to allow a negative (but arbitrarily close to $0$) valuation of $f_i$ in the main application, which is a source of a few interesting features. For instance, see Remark \ref{rmk:failac} and, in that regard, \ref{subsec:geomnon-toric} in addition.

\subsection{Wall-crossing in $\dim_\mathbb{C} X =2$}

Now we suppose $\dim_\C X =2$. Given a generic nodal fiber (topologically a once-pinched torus) of $X \setminus D \to B$, we describe more concretely the shape of the associated wall structure and its wall-crossing transformation.  Let $L_0=\varphi^{-1} (u_0)$ denote the nodal fiber, and consider the vanishing cycle\footnote{It becomes a genuine vanishing cycle in Picard-Lefschetz theory after the hyperk\"{a}hler rotation.} in $H_{1} (L_u ; \mathbb{Z})$ for $L_u$ a nearby fiber. 
The corresponding vanishing thimble produces a holomorphic disk of Maslov $0$. Let us denote its class by $\beta_0$. Torus fibers bounding this $\beta_0$-disk are aligned in an affine line with respect to complex affine coordinates emanating from $u_0$ (see for instance, \cite[Proposition 5.6]{Lin}) which we call an \emph{initial wall} or \emph{initial ray}. Maslov index 2 disks can be glued with this to produce new Maslov index 2 disks which results in nontrivial change of the potential. 
 Note that the affine structure has a nontrivial monodromy around the singular fiber, and one usually chooses a branch-cut (some infinite ray starting from $u_0$) in $B$, across which the affine coordinates jump by this monodromy. 
  
When we go across the initial wall positively as in Figure \ref{fig:wcori}, the two potentials (or the corresponding local mirror charts) are related by the coordinate change
\begin{equation}\label{eqn:wctransgamma}
 z^\gamma \mapsto z^\gamma (1+ f(z^{\partial \beta_0} ))^{\langle \partial \beta_0,\gamma \rangle}  
 \end{equation}
where $f$ can be formulated in terms open Gromov-Witten invariants \cite{Lin}.

\begin{figure}[h]
	\begin{center}
		\includegraphics[scale=0.5]{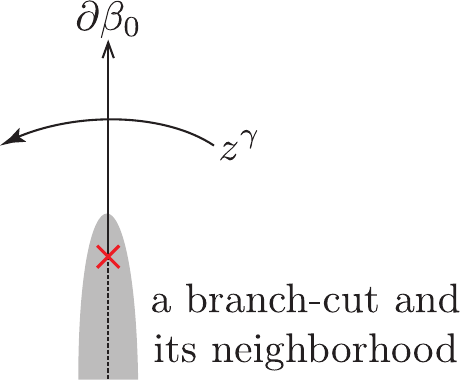}
		\caption{}
		\label{fig:wcori}
	\end{center}
\end{figure}

Initial walls form building blocks of the entire wall structure on $B$. In fact, removing arbitrarily small neighborhoods $B_{sing}$ of branch-cuts, the union of all the walls can be identified as the minimal consistent scattering diagram containing these rays drawn on 
\[
	B_{reg}:=B \setminus B_{sing}.
\]
Here, the scattering diagram is the set of walls (affine rays with integral slopes) coupled with the wall-crossing transformations (of the form \eqref{eqn:wctransgamma}), and the diagram is called \emph{consistent} if for any loop in $B_{reg}$, the composition of all wall-crossing transformations for the intersecting walls is the identity. 
We will denote by $\mathfrak{D}$ the consistent scattering diagram on $B_{reg}$ consisting of walls.
It is known that the additional rays (those other than initial rays) in $\mathfrak{D}$ are produced by colliding of initial rays which are completely determined by consistency. We refer readers to \cite{BCHL} for more details about the scattering diagram and its appearance in SYZ mirror construction.

It is more convenient to understand the wall structure in terms of tropical geometry. In fact, points $u \in B_{reg}$ on the wall can also be characterized as possible ends of Maslov 0 tropical disks in $B_{reg}$ which are defined (in our particular geometric situation) as follows.

\begin{definition}\label{def:trop0} A Maslov $0$ tropical disk in $B_{reg}$ with end at $u \in B_{reg}$ is the image of a continuous map $h : T \to B_{reg}$  away from the union of branch-cuts  satisfying the following. Denote by $T_0$ and $T_{1}$ the sets of vertices and edges of $T$, respectively. 
\begin{itemize}
\item[(i)] $T$ is a rooted tree with a unique root $x$ and $h(x) =u$;
\item[(ii)] $h|_e$ embeds $e \in T_{1}$ onto an affine line segment with integral slope (with respect to complex affine coordinates);
\item[(iii)] $h$ maps a leaf (that belongs to a finite edge) to one of singular fibers $u_0$; 
 \item[(iv)] there exists a weight function $w: T_{1} \to \mathbb{Z}_{>0}$ such that at any $v \in T_0$, one has the balancing condition
\begin{equation}\label{eqn:bal}
 \sum_i w(e_i) v(e_i) = 0
\end{equation}
where the sum is taken over all edges incident to $v$, and $v(e_i)$ denotes the primitive vector along $h(e_i)$ pointing away from $v$. 
\end{itemize}
To a tropical disks, one can assign a relative homology class 
\begin{equation}\label{eqn:classmi0trop}
\beta:=\sum_{x_i \,\, \mbox{a leaf}}  \beta_{h(x_i)} \in H_{2} (X\setminus D , L)
\end{equation}
where $\beta_{h(x_i)}$ is the class of vanishing thimbles for the singular fiber $L_{h(x_i)}$. (Hence, one can also make sense of a boundary class($=\partial \beta \in H_{1} (L_u)$) of the tropical disk.)
\end{definition}

Notice that the initial walls explained above are precisely the simplest kind of Malsov $0$ tropical disks.
Moreover, if there exist two Maslov $0$ tropical disks ending at $u \in B_{reg}$, then they can glue together to produce a new class of disks. Namely, we take the union of these two disks possibly with nontrivial multiplicities and add an additional edge emanating from $u$ which is determined by \eqref{eqn:bal}. This is precisely the reason (in tropical side) we obtain a consistent scattering diagram from the wall structure. 

\subsubsection{Extension of local coordinates to the gluing of local mirrors}\label{subsubsec:globalizecoord}

Consider two small regions $U_{1}$ and $U_{2}$ on opposite sides of a single wall in $B_{reg}$, which do not contain any walls themselves.
Suppose $U_{1}$ and $U_{2}$ above are both away from a neighborhood of the branch-cuts. Denote by $W_{1}$ and $W_{2}$ the local mirror potentials defined on these regions. Each $W_i$ can be described as a single convergent Laurent series on  $U_i \times (\Lambda_U)^2$ after trivializing the fibration (that induces coordinates \eqref{eqn:mirrorzi}).
We can choose a basis $\{f_{1},f_{2}\}$ of $H_{1} (L_u;\mathbb{Z})$ for $u \in U_{1} \cup U_{2}$ such that the associated coordinates $(z_{1},z_{2})$ and $(z_{1}',z_{2}')$ on respectively $U_{1}$ and $U_{2}$ make the wall-crossing transformation in the simple form  
\begin{equation}\label{eqn:clusterf}
z_{1}' = T^{\langle f_{1}, u  - u' \rangle} z_{1}, \quad z_{2}' = T^{\langle f_{2}, u  - u' \rangle} z_{2} (1+ f(z_{1})) 
\end{equation}
compared at $u(=\val (z_{1},z_{2})) \in U_{1}$ and $u'(=\val (z_{1}',z_{2}')) \in U_{2}$, where $f$ is a power series appearing in the wall-crossing transformation \eqref{eqn:wctrans}.
(For e.g., choose $f_{1}$ in the basis to be the class of $\partial \beta_0$ where $\beta_0$ is the class of Maslov zero disks responsible for the wall between $U_{1}$ and $U_{2}$.)

Notice that if $f$ in \eqref{eqn:clusterf} belongs to $\Lambda_+$, then the transformation does not create any extra energy (the exponent of $T$) except the flux between different points in $B$ which only indicates the difference of the locations of $u$ and $u'$. 
Therefore it is possible to extend the coordinates $z_{1},z_{2}$ throughout the second chart $U_{2}$ simply by solving $z_i' = z_i (1+ f_i(z))$ in $z_i$ as long as the solution $z_i$ is well-defined (i.e convergent as power series in $z_i'$ over $\Lambda$). 

\begin{remark}\label{rmk:failac}
In the actual application, $f_i$ in \eqref{eqn:clusterf} can be of negative valuation (see, for e.g., \eqref{eqn:wcformualepsilon}), and it is possible for a certain $z_{1}$ that $\val(z_{2}')$ and $\val(z_{2})$ are significantly different. For this reason, this analytic continuation from $U_{1}$ cannot cover some codimension $1$ subset in $U_{2}$, and vice versa. 
\end{remark}

More concretely, we define $\tilde{z}_i$ on $U_{1} \cup U_{2}$ in such a way that $\tilde{z}_{1}= T^{\int_{A_{1}} \omega} hol_{f_{1}} \nabla(=z_{1}=z_{1}')$ everywhere on $U_{1} \cup U_{2}$ and $\tilde{z}_{2} = z_{2}$ on $U_{1}$, but $\tilde{z}_{2} = z_{2}' (1+ f(z_{1}'))^{-1}$ on $U_{2}$.
In these coordinates, one only needs to keep the expression $W_{1}$ to describe the mirror over $U_{1} \cup U_{2}$ since
$W_{2}(z_{1}',z_{2}') = W_{1} (\tilde{z}_{1},\tilde{z}_{2})$
by definition of $\tilde{z}_i$ as long as the right hand side converges.


%
%

%
%

\subsection{Log CY surfaces and toric models}\label{subsec:toricmodel}
Our main interest is the mirror LG model for the complex surface $X$ obtained by a non-toric blowup of a toric surface $X_\Sigma$, that is, we allow the blowup center to consist of generic points in a toric divisor $D_\Sigma$. $X$ forms a log Calabi-Yau pair together with $D$ the proper transform of $D_\Sigma$. More generally, by \cite[Proposition 1.3]{GHK}, any log Calabi-Yau surface $(X',D')$ can be represented as a blowdown of such a surface $X$ where the blowdown this time contracts divisors in $D$. Namely, for any $(X',D')$, one can find a diagram
\begin{equation}\label{eqn:looijenga}
\xymatrix{ & (X, D ) \ar[dr]^{\pi} \ar[dl]_{\pi'} & \\ 
(X',D') & & (X_\Sigma, D_\Sigma)}
\end{equation}
where $\pi$ is the non-toric blowup, and $\pi'$ is a blowup of $X$ along nodal points in $D'$ ($\pi'$ is often called a toric blowup). \eqref{eqn:looijenga} is called a toric model of $(X',D')$.

Throughout, we assume that the symplectic size of all the exceptional divisors are small enough, so they are located near infinity of $B$ when written in complex affine coordinates. Thus the class of the symplectic form $\omega=\omega_X$ on $X$ is given as $[\omega]=[\pi^\ast \omega_{X_\Sigma} - \sum \epsilon_i E_i]$ where $E_i$'s are exceptional classes and $0< \epsilon_i \ll 1$.  As before, we choose a small neighborhood of each branch cut near $E_i$ and consider its complement $B_{reg}$.
In \cite{BCHL}, the first author constructed a Lagrangian fibration on the divisor complement, $X \setminus D= X' \setminus D' \to B$, which is special on $B_{reg}$, and showed that the resulting scattering diagram on $B_{reg}$ agrees with the one in \cite{GrPS} induced by some algebraic curve counting. Torus fibers are automatically weakly-unobstructed if $D$ is positive for degree reason, and it is the case for general $D$ by Lemma \ref{lem:torusweakunobs}, below.

 If we take a non-toric blowup at a generic point in $D_{\Sigma}$ of the toric model, then the fibration is obtained by gluing in the local model appearing in \cite{auroux09} to the Lagrangian fibration pulled-back from $X_\Sigma$ away from exceptional divisors. 
Each point in the blowup center corresponds to one nodal fiber in the torus fibration on the blowup. If $z_{1}$ denotes the coordinate induced from the boundary class $\partial \beta_0$ of the associated Maslov 0 disk, then the wall crossing formula is given as
\begin{equation}\label{eqn:wcformualepsilon}
\begin{array}{l}
 z_{1} = z_{1}' (1+ T^{-\epsilon} z_{2}') \\
 z_{2} = z_{2}'
 \end{array}
 \end{equation}
where $-\epsilon$ in the formula accounts for the difference between $\pi^\ast \omega_{X_\Sigma}$ and $\omega$.


Each exceptional divisor shares one point with the associated nodal fiber, and hence all the nodal fibers are sitting close to infinity as well. The nontrivial monodromy of the affine structure near the nodal fiber (singularity of the affine structure) can be pushed to a single branch cut, which we take to be the ray from the singular fiber toward infinity. See Figure \ref{fig:wcori}.

\subsection{The mirror LG model for $(X,D)$ and geometric critical points}\label{subsec:geomcrit}

In the situation of Proposition \ref{prop:wcountbrl},
if $u$ remains in a chamber $R$ (a connected component of $B_{reg} \setminus \mbox{union of walls}$), then $W_u$ varies continuously, and hence gives a well-defined function on $\check{Y}_R:=\val ^{-1} (R)$ where $\val : \check{Y} \to B_{reg}$ is a dual torus fibration in rigid analytic setting appearing in \ref{subsec:LFT}.\footnote{Strictly speaking, one has to take $\check{Y}$ in this context to be the family Floer mirror of $\varphi^{-1} (B_{reg}) (\subset X) \to B_{reg}$.} We obtain a global function by gluing these local pieces via wall-crossing transformations \eqref{eqn:wctrans}. 

In view of discussion in \ref{subsubsec:globalizecoord}, it is enough to consider some fixed expression of $W_R$ valid only over some chamber $R$, but with an enlarged domain. The expression \eqref{eqn:Wunon-toricblowup} a priori defines a function on $(\Lambda^\ast)^2$ (as long as it converges).
Note that the coordinates $z_{1},z_{2}$ are actually global as we have removed small neighborhoods of branch-cuts. Indeed, the image of the embedding of $B_{reg}$ into $\mathbb{R}^2$ by (global) symplectic affine coordinates $(\val (z_{1}),\val (z_{2}))$ converges to the moment polytope $\Delta_\Sigma$ as $B_{sing}$ shrinks. This naturally leads to the following definition.

\begin{definition}\label{def:geomcrit}
The critical point of $W_{R}$ 
considered as a function on the maximal domain convergence is called \emph{geometric} if its valuation lies in $B_{reg} \approx \Delta_\Sigma$.
\end{definition}

It is possible that geometric critical points sit over the codimension $1$ subset not covered by the analytic continuation (Remark \ref{rmk:failac}). We can handle this problem in two different ways. First, we can choose the location of the blowup center cleverly so that there exists a chamber which contains every geometric critical points (Lemma \ref{lem:nooutr0}). Alternatively, we can add an extra chart coming from Floer deformation theory of nodal fibers as in \ref{subsec:geomnon-toric}. The latter works only for the semi-Fano case.

 In practice, we will work with some special chamber $R=R_0$ defined as follows. 
%
Recall that the wall structure associated to the torus fibration on $X \setminus D $ is given as the minimal consistent scattering digram that contains initial rays determined by Maslov zero disks emanating from singular fibers.  
For convenience, we will always choose locations of points in the blowup center (for $X \to X_\Sigma$) to be near the corners of $D_\Sigma$, or more precisely right after the corner when traveling around $D_\Sigma$ counterclockwise. If a point in the blowup center lies in an irreducible component $D_{\Sigma,i}$ of $D_\Sigma$, then we require its valuation to be close enough to the corner $D_{\Sigma,i} \cap D_{\Sigma, i-1}$. In this case, the corresponding initial ray (the wall) in the scattering diagram is parallel to $\nu_i$. See Figure \ref{fig:blowuppts}.

\begin{figure}[h]
\centering
\subcaptionbox{%
	Choice of blowup centers.
	\label{fig:blowuppts1}%
	}
	{\makebox[0.45\linewidth][c]
		{\begin{tikzpicture}[scale=0.84]
				
		\coordinate (A) at (0,0);
		\coordinate (B) at (4,0);
		\coordinate (C) at (0,4);

		\draw[line width =0.7pt] (0.5,0) node[cross] {};
 		\draw[line width =0.7pt] (0.2,0) node[cross] {};
		\draw[line width =0.7pt] (0,3.5) node[cross] {};
		\draw[line width =0.7pt] (0,3.2) node[cross] {};
		\draw[line width =0.7pt] (3.7,0.3) node[cross, rotate = 45] {};
		\draw[line width =0.7pt] (3.5,0.5) node[cross, rotate = 45] {};

		\coordinate[label=below:$D_{\Sigma,1}$](c) at ($ (A)!.35!(B) $);
		\coordinate[label=left:$D_{\Sigma,3}$] (b) at ($ (A)!.5!(C) $);
		\coordinate[label=right:$D_{\Sigma,2}$](a) at ($ (B)!.5!(C) $);

		\draw [line width=1pt] (A) -- (B) -- (C) -- cycle;

		\end{tikzpicture}%
	}
	}\hfill%
\subcaptionbox{%
	The resulting scattering diagram
	\label{fig:blowuppts}%
	}
	{\makebox[0.45\linewidth][c]	
		{\begin{tikzpicture}[scale=0.5]
		
		\draw[fill=green!30, opacity=0.4, draw=none] (0.5,-2.4)--(0.5,3.2)--(6.2,3.2);

		\draw [line width=0.7pt] (0.5,-4) -- (0.5,5);
		\draw [line width=0.7pt] (0.2,-4) -- (0.2,5);
		\draw [line width=0.7pt] (-1,3.5) -- (8,3.5);
		\draw [line width=0.7pt] (-1,3.2) -- (8,3.2);
		\draw [line width=0.7pt] (-0.1,-3.4) -- (7.4,4);
		\draw [line width=0.7pt] (-0.3,-3.2) -- (7.2,4.2);

		\coordinate[label=right:$R_0$, text opacity=1](r) at (1.5,1.2);

		\end{tikzpicture}
		}
	}%
\caption{}
\label{fig:convention}
\end{figure}

We claim that with this choice of a blowup center, there always exists an open chamber $R_0$ (a connected component of $B_{reg}$ minus walls that has a nontrivial area) in the scattering diagram that is surrounded by initial rays.

\begin{lemma}\label{lemma:centralchamber}
Let $p_i$ be the blowup point in $D_{\Sigma,i}$ farthest from the corner of $D_{\Sigma}$ that comes right before $D_{\Sigma,i}$ when traveling $D_{\Sigma}$ counterclockwise. Denote by $R_0$ the region in $B$ enclosed by initial rays for $p_i$'s (see Figure \ref{fig:blowuppts}). Then no point in the interior of $R_0$ bounds a Maslov 0 tropical disk in $X \setminus D$.
\end{lemma}
 
 \begin{proof}
 
We auxiliary orient each edge of Maslov $0$ tropical disks in such a way that it points toward the end. We claim that for any edge of a Maslov $0$, $R_0$ lies on the right side of the line extending the edge (with respect to the orientation of the line chosen as above). To see the claim, we proceed with the induction on the number of vertices of tropical disks. Initial rays clearly satisfy this condition.

Suppose now a Maslov 0 tropical disk $h: T \to B_{reg}$ is given, and consider the first vertex $v$ that we meet when starting from $u$ walking in the reverse orientation. Removing the edge $e_0$ of $h(T)$ incident to the end $u$, one obtains two Malsov 0 tropical disks intersecting at $v$, both of which should satisfy the induction hypothesis. We then only need to check the condition for the removed edge $e_0$ of $h(T)$, which is obvious from the balancing condition \eqref{eqn:bal}.
 \end{proof}
 

Suppose now that the divisors $D_\Sigma$, $D$ and $D'$ appearing in \eqref{eqn:looijenga} do not contain any rational curve with a negative Chern number. (Obviously, it is enough to check this for $D$.) 
Note that the disk counting (the number $N_\beta$ for each disk class $\beta$) remains constant on $R_0$. The corresponding local mirror is $\check{Y}_{R_0}$ fibered over $R_0$ with $(\Lambda_U)^2$-fibers, and the count of Malsov index $2$ holomorphic disk defines an analytic function on $\check{Y}_{R_0}$. We then globalize this local LG model on $\check{Y}_{R_0}$ by the trick mentioned in \ref{subsubsec:globalizecoord}. We shall give the tropical description of this in short.

\begin{remark}\label{rmk:r0nonfano}
 If the boundary divisor has spheres with negative Chern numbers, there could be additional Malsov $0$ disks in $X$ obtained as stable disks consisting of higher Maslov disks attached with negative sphere bubbles. Hence, it is possible to have stricitly more walls than those intrinsic to $X \setminus D$. In this case, we can only guarantee that the coefficients of certain leading order terms of $W$ remain constant over $R_0$, which is enough for our purpose of analyzing critical points.
 \end{remark}

\subsection{Tropical description of the Landau-Ginzburg mirror}\label{subsec:LGtoricmodel}

Suppose now that $D$ is positive i.e., $(X,D)$ being Fano. We write $D_{i}$ for the proper transform of $D_{\Sigma,i}$. In this case, the mirror potential can be computed tropically as shown in \cite{BCHL}. We briefly review the definition of tropical disks, especially which accounts for the potential. The count of tropical disks matches that of holomorphic disks due to the correspondence between tropical and holomorphic disks established therein. Let us first focus on $(X,D)$ together with its associated special Lagrangian fibration on $X \setminus D$ explained above. 

\begin{definition} A tropical disk in $B$ with end at $u \in B_{reg}$ is the image of a continuous map $h : T \to B_{reg} $ away from the union of branch-cuts  that satisfies conditions (i) - (iv) in Definition \ref{def:trop0}, but additionally allowing an edge $e \in T_{1}$ to be unbounded subject to the condition below. (Alternatively such $e$ may also be  viewed as an edge incident to a leaf sitting at infinity.) 
\begin{itemize}
\item[(v)] if $e \in T_{1}$ is unbounded, then $h(e)$ is an affine ray in $B_{reg}$ (with respect to complex affine coordinates) which approaches infinity along $- \nu_i$ where $\nu_i$ is a primitive generator of the fan of $\bar{X}$.
\end{itemize}
Finally, the Maslov index of a tropical disk is defined as the number of unbounded edges.
\end{definition}

 
Fix a point $u$, or equivalently the fiber $L_u$. It is automatically weakly-unobstructed by degree reason. Let us first look at a basic Maslov $2$ disk bounding $L_u$, that is, a disk intersecting $D$ exactly once. Therefore, it corresponds to a tropical disk which has a unique unbounded edge $e$ with $h(e)$ an affine ray in $B_{reg}$ approaching infinity along the direction 
perpendicular to $D$. Indeed a basic disk does project to such an affine ray in the base (when written in the complex affine coordinates). By definition, any unbounded edge of a tropical disk has its corresponding basic Maslov $2$ disk, and hence the Maslov index of a tropical disk is a tropical interpretation of the index formula which equates the Maslov index of a holomorphic disk and the intersection number of the anticanonical divisor \cite{auroux07}.

Suppose a tropical disk $h:T\to B_{reg}$ has unbounded edges $e_{1},\cdots, e_l$ that are perpendicular to $D_{i_{1}},\cdots, D_{i_l}$ (proper transforms of $D_{\Sigma,i_{1}},\cdots, D_{\Sigma,i_l}$), and its finite leaves map to singular fibers at $x_{1},\cdots, x_{l'}$. Then one can assigns a relative homology class to this tropical curve (which will be also referred to as the class of $h(T)$) given by
\begin{equation}\label{eqn:classMI2trop}
\beta = \sum_{a=1}^l \widetilde{\beta_{\nu_{i_a}}} + \sum_{b=1}^{l'} \beta_{h(x_b)} \in H_{2} (X,L)
\end{equation}
where $\widetilde{\beta_{\nu_i}}$ denotes the proper transform  of the class $\beta_{\nu_i}$ of the basic Maslov $2$ disk that hits $D_{\Sigma,i}$ exactly once away from the blowup center, and $\beta_{h(x_b)}$ is the class of the vanishing thimble as in \eqref{eqn:classmi0trop}. $\beta_{h(x_b)}$ and $\beta_{h(x_{b'})}$ give the same class if singular fibers at $x_b$ and $x_b'$ arise from blowup points lying in the same toric divisor.

Observe that the intersection of $h(T)$ and the walls in the scattering diagram $\mathfrak{D}$ is a disjoint union of Maslov $0$ tropical disks, say $h_i:T_i \to B$ for $i=1,2,\cdots, l''$. In fact, $h(T)$ joins $h_{1}(T_{1}),\cdots, h_{l''}(T_{l''})$ and initial rays $e_{1},\cdots, e_l$, and the balancing condition  determines the remaining part of $h(T)$. It is easy to see that the sum of classes of $h_i(T_i)$ equals the second summand of \eqref{eqn:classMI2trop}. First introduced in \cite{GrPS}, the complement $h(T) \setminus \cup_i h_i(T_i)$ is usually called a \emph{broken line}, from which one can completely recover $h(T)$ itself by the above discussion. It is a piecewise linear curve in the scattering diagram, whose non-smooth points lie in walls. At a wall, it can bend towards the direction that is given as a positive multiple of the primitive direction of the wall (or the corresponding Maslov $0$ tropical disk). See Figure \ref{fig:brokenlines}. 
The class of a broken line is defined as the class of its associated Maslov $2$ tropical disk.
For a more formal definition of the broken line, see for e.g., \cite[Definition 4.2]{CPS}.

\begin{figure}[h]
\centering
	\begin{tikzpicture}[scale=0.55]

		\draw[line width=0.7pt, dashed] (-1,-2)--(-1,6);
		\draw[line width=0.7pt, dashed] (1,-2)--(1,6);
		\draw[line width=1pt] (-4,0)--(-1,0)--(1,2)--(2.5,5);
		\filldraw (2.5,5) circle (1pt);

		\draw [blue, line width=0.7pt] (1,-1) ellipse (0.3cm and 0.1cm);
		\draw[blue, line width=0.7pt] (1,-1) +(0:0.3cm and 0.1cm) arc (0:-180:0.3cm and 1.1cm);

		\draw [blue, line width=0.7pt] (-1,-1) ellipse (0.3cm and 0.1cm);
		\draw[blue, line width=0.7pt] (-1,-1) +(0:0.3cm and 0.1cm) arc (0:-180:0.3cm and 1.1cm);

		\draw [blue, line width=0.7pt, rotate around={270:(-4,0)}] (-4,1) ellipse (0.3cm and 0.1cm);
		\draw[blue, line width=0.7pt, rotate around={270:(-4,0)}] (-4,1) +(0:0.3cm and 0.1cm) arc (0:-180:0.3cm and 1.1cm);

		\draw (-4.5,0) node[above] {\small{$\widetilde{\beta}$}};

	\end{tikzpicture}
\caption{Broken Lines} \label{fig:brokenlines}
\end{figure}
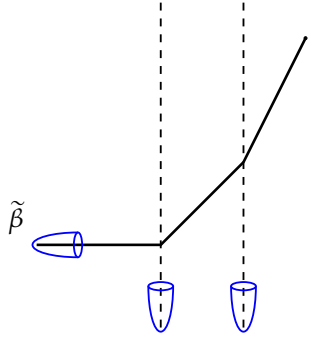

In summary:
\begin{prop}\label{prop:wcountbrl}
Suppose $(X,D)$ is a log Calabi-Yau surface, and each irreducible component of $D$ has a positive Chern number. If $X \setminus D$ carries a special Lagrangian fibration with at worst nodal fibers, then the potential $W_u$ at $u \in B_{reg}$ can be calculated by counting broken lines with their ends at $u$. More precisely,
\begin{equation}\label{eqn:Wunon-toricblowup}
W_u = \sum_{\beta  \in H_{2}(X,L_u)} N^{trop}_{\beta} T^{\omega(\beta)} z^{\partial \beta}
\end{equation}
where $N^{trop}_{\beta}$ is the number of broken lines in class $\beta$.
\end{prop}

The mirror potential for a general log Calabi-Yau surface $(X',D')$ fitting into the toric model \eqref{eqn:looijenga} can be obtained from $W_u$ \eqref{eqn:Wunon-toricblowup} above by removing $\beta$ that nontrivially intersects divisors in $D$ contracted under $\pi'$.

One can extend the discussion above to semi-Fano situation which further allows Chern number zero spheres. We refer readers to \cite{BCHL} for this generalization. Notice that a holomorphic disk attached with such a sphere bubble contributes the same monomial as the disk itself, but with a coefficient that has higher energy (valuation).

\subsection{Non-Fano situation}
When the divisor $D$ has negative Chern number spheres, the weakly-unobstructedness of torus fibers are no longer guaranteed merely by degree consideration. One can connect $(X,D)$ with another toric surface  by inductively moving the blowup point to corners, and make use of the strategy of \cite{FOOO-T, FOOO-T2}. Hence \cite{Yuan} produces the mirror LG model defined on some rigid-analytic domain.

\begin{lemma}\label{lem:torusweakunobs}
A smooth Lagrangian torus fiber $L$  is weakly unobstructed. More precisely, for any $b \in H^1(L;\Lambda_+)$ and $\mathfrak{b} \in H^{even} (X;\Lambda_+)$ pulled-back from a torus invariant cycle in $X_\Sigma$, we have
\begin{equation}\label{eqn:bulkdefunobs}
 m_0^{\mathfrak{b}}(1) + m_{1}^{\mathfrak{b}} (b) + m_{2}^{\mathfrak{b}} (b,b) + \cdots = W^\mathfrak{b} (b) \cdot [L]
\end{equation}
for some $W^\mathfrak{b}(b) \in \Lambda$.
\end{lemma}

\begin{proof}
We use induction on the number of points in the blowup center $C$ for the non-toric blowup $\pi : \widetilde{X} \to X_\Sigma$. If $C$ is empty, that is, if the surface is the toric surface $X_\Sigma$ itself, then the fiber is simply the fiber of the moment map, and is shown to be weakly unobstructed by \cite[Corollary 10.5]{FOOO-T} and \cite[Corollary 6.6]{FOOO-T2} using torus-invariant Kuranishi perturbation.

Suppose now that the statement holds when the blowup center (for the non-toric blowup) consists of less than $k+1$ points, and consider $\pi : \widetilde{X} \to X_\Sigma$ which is the blowup of a toric surface $X_\Sigma$ at $C$ consisting of $k+1$ generic points in the interior of toric divisors. Choose any point $p \in X_\Sigma$ in $C$. There exists a symplectomorphism $F$ sending $p$ to a torus fixed point, say $p'$, in $X_\Sigma$ supported away from $\pi(L)$. 
Denote by the (toric) blowup of $X_\Sigma$ at $p'$ by $X_{\Sigma'}$, and the non-toric blowup of $X_{\Sigma'}$ at $C \setminus \{p\}$ (or, more precisely, its image in $X_\Sigma'$) by $X'$. We require that the exceptional divisor in $X_{\Sigma'}$ associated with $p'$ has the same symplectic size as $\pi^{-1} (p)$.
By construction, the inverse image $L'$ of $L$ in $X'$ is a Lagrangian torus fiber of the SYZ fibration on $X'$. 
Then we have the following commutative diagram
\begin{equation*}
\xymatrix{ X' \ar[rr]^{\tilde{F}} \ar[d] && X \ar[d]^{\pi}\\
X_{\Sigma'} \ar[rr]^{F} && X_{\Sigma}}
 \end{equation*}
where $\tilde{F}$ is a symplectomorphism that lifts $F$. Clearly $F$ maps $L'$ to $L$, and hence, it suffices to prove the weakly unobstructedness of $L'$, which follows from the induction hypothesis.
\end{proof}

\begin{remark}\label{rmk:nodalweakunobs}
We speculate the analogous is true for nodal fibers i.e. any immersed generators form weak bounding cochains for any bulk-parameters. This is the case when $X$ is semi-Fano by degree reason. Notice that the left hand side of \eqref{eqn:bulkdefunobs} can have terms in degree other than $0$ only when there exist contributions from  negative Maslov index disks.
\end{remark}



\section{Critical points of Laurent Polynomials}\label{sec:tropcrit}

In this section, we look into the problem of finding the number of critical points of a given Laurent polynomial. In fact, a well-defined combinatorial formula already exists, given in terms of Newton polytopes, due to Kushnirenko \cite{Kush}(or see Theorem \ref{Kushnirenko}). However, we require a more refined version of this formula, as we wish to further estimate the non-Archimedean valuations of critical points. For this purpose, we will study the tropicalization of the Laurent polynomial, and eventually prove that the critical points must sit over vertices of the resulting tropical curve (Proposition \ref{prop:tropcrit}). This is essentially known from the work of Gonz\'{a}lez-Woodward \cite{GW}, although not phrased in the realm of tropical geometry. Our local argument is also based on the energy induction thereof. On the other hand, our mirror potential has been calculated tropically as shown in \cite{BCHL}, and hence the approach here fits more into our geometric setup.

\subsection{The Kushnirenko Theorem}\label{Kush.}

We first briefly review the classical result of Kushnirenko, which determines the number of critical points of a given Laurent polynomial. Let $W \in \Bbbk[z_{1}^{\pm}, \ldots , z_{n}^{\pm}]$ be a Laurent polynomial with coefficients in some closed field $\Bbbk$ of characteristic $0$ (we will take $\Bbbk = \Lambda$ in our geometric setup). Identifying the lattice $\mathbb{Z}^{n}$ with Laurent monomials, i.e. $z^{v} = z_{1}^{v_{1}}\cdots z_{n}^{v_{n}}$ for $v \in \mathbb{Z}^{n}$, we write 
\[
W =  \sum_{v \in \mathbb{Z}^{n}} \alpha_{v}z^{v}.
\]
Recall that the \emph{Newton polytope} $\Delta_{W}$ of $W$ is defined as the convex hull of the support of $W$, $\supp W := \{\,  v \in \mathbb{Z}^{n}   \;\, \rvert \;\,  \alpha_{v} \neq 0  \,\}$. 

\begin{definition}
A Laurent polynomial $W$ is said to be \emph{convenient} if the point $0 \in \mathbb{R}^{n}$ does not belong to any supporting plane of all $d$-dimensional faces  of $\Delta_{W}$ for $1\leq d \leq n-1$. 
\end{definition}
Note that when $n=2$, a Laurent polynomial $W$ is convenient if and only if two adjacent vertices $v_{1}, v_{2} \in \partial \Delta_{W}$ do not lie on a line through the origin.

For any closed subset $F$ of $\mathbb{R}^{n}$, we write $W_{F} := \sum_{v \in F \cap \mathbb{Z}^{n}}\alpha_{v}z^{v}$. A Laurent polynomial $P$ is \emph{non-degenerate} if for any closed face $F$ of $\Delta_{P}$, the system 
\[
	\Bigl( z_{1}\frac{\partial W_F}{\partial z_{1}}\Bigl)_{F} = \dots  = \Bigl( z_{n}\frac{\partial W_F}{\partial x_{n}}\Bigl)_{F}=0
\]
has no solution in $(\Bbbk^\times)^n$. In this case, the Newton polytope $\Delta_{W}$ is also called non-degenerate.


\begin{thm}\cite[Theorem III]{Kush} \label{Kushnirenko}
Let $\Bbbk$ be an algebraically closed field with characteristic 0. If a convenient Laurent polynomial $W$ is non-degenerate,
\[
\rvert \mathrm{Crit}(W) \rvert = n!V_{n}(\Delta_{W})
\]
where $\rvert \mathrm{Crit}(W) \rvert$ is the number of critical points of $P$ counted with multiplicity, and $V_{n} (\Delta_W)$ is the $n$-dimensional volume of the Newton polytope $\Delta_W$.
\end{thm}

 We are mainly interested in the case $n=2$. In this case, it is not difficult to see that once $P$ is convenient, it is non-degenerate for generic choice of coefficients. Analogous statement should be true in arbitrary dimension since $\Bigl( z_{1}\frac{\partial f}{\partial z_{1}}\Bigl)_{F} = \dots  = \Bigl( z_{n}\frac{\partial f}{\partial x_{n}}\Bigl)_{F}=0$ is overdetermined.

\begin{example}
Consider $W(z) = a z_{1}z_{2} + b z_{1}z_{2}^{2} +  c z_{1}^{2}z_{2}^{2} $ for some $a,b,c \in \Bbbk$. $W$ has no critical points for generic $a,b,c$, whereas $\Delta_W$ has a positive volume. Indeed, $P$ is not convenient due to the two terms $z_{1} z_{2}$ and $z_{1}^{2}z_{2}^{2}$ that are colinear. It has infinitely many critical points for $b=0$, on the other hand.
\end{example}

\subsection{Tropicalization of Laurent Polynomials and the Duality Theorem}\label{subsec:tropdual}

We recall some basics from tropical geometry (mostly without proof). Readers are referred to any expository article in tropical geometry for more details, for e.g., see \cite{BIMS}. Throughout, we fix our base field as the Novikov field $\Lambda$.
Let us take a Laurent polynomial $W(z_{1},\cdots,z_{n})$ over $\Lambda$, and write it as 
$W = \sum_{i=1}^{N}a_{i} z^{v_{i}}$ for $v_i \in \mathbb{Z}^n$. Then its tropicalization is a polyhedral complex in $\mathbb{R}^n$ (e.g., a piecewise linear graph on $\mathbb{R^2}$ when $n=2$) defined as follows:

\begin{defn}\label{Tropicalization}
The \emph{tropicalization} of a Laurent polynomial $W = \sum_{i =1}^{N}a_{i} z^{v_{i}}$ in $\Lambda[z_{1}^{\pm}, \cdots, z_{n}^{\pm}]$ is a subset of $\mathbb{R}^{n}$ given as the corner locus of the piecewise linear function 
\[
	\tau_W : \mathbb{R}^n \to \mathbb{R}, \quad \,\,
	(x_{1},\cdots, x_{n}) \mapsto \min_{i} \, 
	\Bigl\{\,   \lambda_{i} + \left\langle v_{i} \, , \, (x_{1},\cdots,x_{n}) \right\rangle    \;\, \rvert \;\,  i \in \support W  \,\Bigr\}
\]
where $\lambda_i = \val (a_i)$ and $\support W := \bigl\{v_i \in \mathbb{Z}^n \mid  a_i \neq 0\bigr\}$. We denote the tropicalization of $W$ by $\mathrm{Trop}(W)$.
\end{defn}

The tropicalization $\mathrm{Trop}(W)$ consists of polygonal faces of different dimensions. For generic coefficients (hence generic $\lambda_i$), every $(n-1)$-dimension face of $\mathrm{Trop}(W)$ is the locus where exactly two linear functions $\lambda_i + \langle v_i, (x_{1},\cdots, x_{n})$ coincide. We assume that this is always the case from now on. One can assign a \emph{weight} $w_{E}$ to each $(n-1)$-face in the following way. Suppose $\lambda_i + \langle v_i, (x_{1},\cdots, x_{n}) \rangle$ and $\lambda_j + \langle v_j, (x_{1},\cdots, x_{n}) \rangle$ agree along $E$. Then $w_E$ is a maximal positive integer for which $\frac{v_i - v_j}{w_E}$ is an integer vector.

If a $(n-2)$-face $V$ of $\mathrm{Trop}(W)$ is incident to edges $(n-1)$-faces $E_{i_{1}},\cdots, E_{i_k}$, one has the \emph{balancing condition}
\[
	\sum_{l=1}^{k}w_{i_l}\nu_{i_l} = 0,
\]
where $w_{i_l}$ is the weight of $E_{i_l}$ and $\nu_{i_l}$ is the primitive integer vector perpendicular to $E_{i_{1}}$ such that $x+ \delta \nu_{i_l}$ belongs to $E_{i_l}$ for $x \in E_{i_l}$ and a small enough positive $\delta$. When $n=2$, $\nu_{i_l}$ is simply the outward primitive vector along the edge $E_{i_l}$  (pointing away from the vertex $V$).

\begin{remark}\label{rmk:tT}
The tropicalization of a complex hypersurface $\mathcal{C}$ in $(\mathbb{C}^\times)^n$ can be obtained as a degeneration of \textit{tropical amoebas} of a family of hypersurfaces $\mathcal{C}_{t}$ parametrized by (real powers of) $t$;
\[
	\mathrm{Trop}(\mathcal{C}) := \lim_{t\rightarrow \infty}Log_{t}(\mathcal{C}_{t})
\]
where $\mathrm{Log}_{t} : (\mathbb{C}^{*})^{n} \rightarrow \mathbb{R}^{n}$ is given by $(z_{1},\cdots, z_{n}) \mapsto (\log_{t}\rvert z_{1}\rvert,\cdots, \log_{t}\rvert z_{n}\rvert)$. Formally, our Novikov variable $T$ is related with $t$ by $t = 1 /T$, and hence, the limit $t \to \infty$ corresponds to the large structure limit $\omega \to \infty$ in the sense of substitution $T = e^{-\omega}$.

\end{remark}

There is a convenient way for determining the combinatorial type of the tropicalization $\mathrm{Trop}(W)$ using duality between $\mathrm{Trop}(W)$ and the Newton polytope $\Delta_W$. Let $\psi_W : \Delta_W \to \mathbb{R}$ be defined as the \emph{maximal} piecewise linear function satisfying
\begin{equation}\label{eqn:liftedpsi}
\psi_W (v_i) = \lambda_i (=\val (a_i)),
\end{equation} 
which is completely determined by $W$.
The domains of linearity of $\psi_W$ determines a subdivision $\mathcal{S}_W$ of $\Delta_W$ into lattice polygons. $\mathcal{S}_{W}$ is usually refered to as the \emph{Newton subdivision} of the Newton polytope $\Delta_W$. Note that the intersection of two cells in the subdivision must occur along their common lower dimensional cell.  

The tropicalization $\mathrm{Trop}(W)$ is constructed as the dual complex of the polyhedral decomposition $\mathcal{S}_W$ of $\Delta_W$. Specifically, each top-dimensional cell of $\mathcal{S}_W$ corresponds to a vertex in $\mathrm{Trop}(W)$. Two vertices of $\mathrm{Trop}(W)$ are joined by an edge, say $E$, whenever their corresponding top cells in $\mathcal{S}_W$ intersect along a shared $(n-1)$-dimensional cell, say $E^\ast$, with the edge $E$ being perpendicular to $\mathcal{S}_W$. Consequently, the subdivision $\mathcal{S}_W$ completely determines the combinatorial type of $\mathrm{Trop}(W)$ (see Figure \ref{fig:dualintcpx} for an example when $n=2$).


\begin{figure}[h] 
	\begin{tikzpicture}[every node/.style={draw}, scale=0.55]

		\coordinate (A) at (3,0);
		\coordinate (B) at (0,3);
		\coordinate (C) at (-3,-3);
		\coordinate (D) at (-3,0);
		\coordinate (H) at (3,3);
		\coordinate (I) at (3,6);

		\draw [line width=1pt] (A) -- (B);
		\draw [line width=1pt] (C) -- (B); 
		\draw [line width=1pt] (A) -- (C); 
		\draw [line width=1pt] (D) -- (B);
		\draw [line width=1pt] (C) -- (D); 
		\draw [line width=1pt] (A) -- (B); 
		\draw [line width=1pt] (A) -- (H);
		\draw [line width=1pt] (H) -- (B);
		\draw [line width=1pt] (I) -- (B);   
		\draw [line width=1pt] (I) -- (H); 

		\filldraw (-0.9,-0.9) circle (3pt);
		\filldraw (-2.2,-0.3) circle (3pt);
		\filldraw (2,2) circle (3pt);
		\filldraw (2,4) circle (3pt);

		\draw [line width=1pt] (-0.9,-0.9) -- (-2.2,-0.3);
		\draw [line width=1pt] (-5.8,-0.3) -- (-2.2,-0.3);
		\draw [line width=1pt] (-5.1,2.7) -- (-2.2,-0.3);
		\draw [line width=1pt] (-0.9,-0.9) -- (0.3,-3.3);
		\draw [line width=1pt] (2,2) -- (-0.9,-0.9);
		\draw [line width=1pt] (2,2) -- (2,4);
		\draw [line width=1pt] (0.5,5.5) -- (2,4);
		\draw [line width=1pt] (2,4) -- (5,4);
		\draw [line width=1pt] (2,2) -- (5,2);
	\end{tikzpicture}
\caption{The dual subdivision $\mathcal{S}_{W}$ determined by $\Gamma_{W}$}
\label{fig:dualintcpx}
\end{figure}
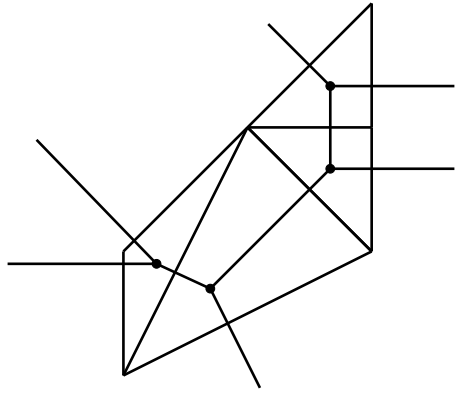


The following Proposition summarizes our discussion so far.

\begin{prop}\label{prop:Duality}
The tropicalizaton $\mathrm{Trop}(W)$ is \textit{dual} to the subdivision $\mathcal{S}_{W}$ of $\Delta_{W}$ determined by the domain of linearity of $\psi_W : \Delta_{W} \to \mathbb{R}$ defined in \eqref{eqn:liftedpsi}.
\end{prop}
\begin{remark}
Due to difference of conventions as in Remark \ref{rmk:tT}, $\mathrm{Trop}(W)$ should be rotated $180$ degrees to become precisely the dual of the subdivision $\mathcal{S}_{W}$.
\end{remark}
Observe that the weight of a top dimensional face $E$ of $\mathrm{Trop}(W)$ equals the integral length of its dual 1-cell $E^\ast$ in $\mathcal{S}_W$. Similarly, we define the weight $|V|$ of a vertex $V$ of $\mathrm{Trop}(W)$ to be the volume $|V|:= n! vol_{n} (F^\ast)$ of the $n$-cell $F^\ast$ in $\mathcal{S}_W$ dual to $V$. 
If exactly $(n+1)$ top dimensional faces intersect at a vertex $V$ of $\mathrm{Trop}(W)$, its multiplicity equals $w_{i_{1}}\cdots w_{i_{n}} \rvert \mathrm{det}(\nu_{i_{1}} \cdots  \nu_{i_{n}})\rvert$ where $w_{j}$ and $\nu_j$ are the weight and primitive of the face $E_j$ incident to $V$, $j=i_{1},i_{2},\cdots, i_{n}$. The balancing condition ensures that the multiplicity is well-defined, i.e. independent of the choice of $n$ incident edges.

Each connected component of $\mathbb{R}^n \setminus \mathrm{Trop}(W)$ is the domain of linearity of $\tau_W$, and hence, is associated with a monomial $a_i z^{v_i}$ of $W$ (or the corresponding linear function $\lambda_i + \langle v_i, (x_{1},\cdots, x_{n})\rangle$). That is, if $\val (z)$ falls in such component, then the \emph{valuation of $a_i z^{v_i}$ becomes the minimum} among the valuations of all monomials of $W$ (this directly follows from the definition of $\mathrm{Trop}(W)$). 

Having this interpretation in mind, the dual top dimensional face $F^\ast$ in $\mathcal{S}_W$, dual to a vertex $V \in \mathrm{Trop}(W)$, can be described as follows:
Let $V$ be a vertex of $\mathrm{Trop}(W)$, and suppose that $a_{i_{1}} z^{v_{i_{1}}}, \cdots, a_{i_k} z^{v_{i_k}}$ are the monomials associated to open components adjacent to the vertex $V$. Then $F^\ast$ is precisely the Newton polytope of $\sum_{l=1}^k a_{i_l} z^{v_{i_l}}$. If $\val (z) = V$, then the valuations of all the monomials $a_{i_{1}} z^{v_{i_{1}}}, \cdots, a_{i_k} z^{v_{i_k}}$ coincide at $z$. In other words, a vertex is where adjacent monomials simultaneously attains minimal valuation.


\subsection{Tropicalization and Critical points of $W$}
We finally extract critical point information of $W$ from its tropicalization $\mathrm{Trop}(W)$. The argument is essentially a tropical interpretation of \cite[Theorem 4.37]{FW} or \cite[Theorem 10.4]{FOOO-T}, which relates critical points of the leading order $W_0$ (with respect to the energy, $\val $) and those of $W$ itself.

\begin{lemma}\label{Woodward}
Let $W = W_{0} + W_{1}$ be a decomposition of $W$ into the terms $W_0$ of the lowest valuations (of ``leading order") and the rest $W_{1}$.
\begin{enumerate}
\item[(a)]
 The lowest order term $\alpha_0$ of a critical point $\alpha$ of $W$ is a critical point of $W_0$.
\item[(b)]
Suppose that $\alpha_0 \in (\Lambda_U)^n$ is a critical point of $W_{0}$ and $W_{0}$ has non-vanishing Hessian $\mathrm{Hess}(W_{0})$ at $\alpha_0$. Then, $\alpha_0$ uniquely extends to a critical point $\alpha$ of $W$, in the sense that $\alpha$ is given by $\alpha = \alpha_0 + \alpha_{1}$ where $val (\alpha_0) < val (\alpha_{1})$. 
\end{enumerate}
\end{lemma}

In particular, if $W_0$ is Morse, then there exists a valuation-preserving one-to-one correspondence between the set of critical points of $W_0$ and that of $W$. By Kushnirenko's theorem (Theorem \ref{Kushnirenko}),   $W_0$ is Morse if and only if it has precisely $2\Delta_{W_0}$ mutually distinct critical points. 
Lemma \ref{Woodward} can be proven by expanding a critical point $p=p_0 + \cdots$ in energy-increasing order with its first order term being one of the critical points $p_0$ of $W_0$, and solving higher order terms order-by-order. At each stage in solving higher order terms, one essentially solves a linear equation determined by the Hessian of $W_0$ at $p_0$. It is thus clear that the valuation of $p$ agrees with that of $p_0$.\\

We first clarify how the decomposition $W= W_{0} + W_{1}$ arises in our context more precisely. In fact, the energy of the terms in $W : (\Lambda^\times)^n \to \Lambda$ depends on the the point in $(\Lambda^\times)^n$ at which we expand $W$ (our main interest is the case $n=2$). A point in $val^{-1} (x_{1}, \cdots, x_{n})$ ($x_i \in \mathbb{R}^n$, $i=1,\cdots, n$) can be written as $(T^{x_{1}} \uz_{1}, \cdots, T^{x_{n}} \uz_{n})$ with $val(\uz_{1}) =\cdots = val(\uz_{n}) =0$. We then take the restriction of $W$ at $val^{-1} (x_{1},\cdots, x_{n})$, that is,
\[
	W_{(x_{1},\cdots, x_{n})} : (\Lambda_U)^n \to \Lambda \qquad (\uz_{1},\cdots, \uz_{n}) \mapsto W(T^{x_{1}} \uz_{1},\cdots, T^{x_{n}} \uz_{n}).
\]
Notice that only coefficients of $W_{(x_{1},\cdots,x_{n})} (\uz_{1},\cdots, \uz_{n})$ can carry nonzero valuation. It is therefore reasonable to write $W_{(x_{1},\cdots,x_{n})}$ as the sum $W_{0} + W_{1}$, where $W_0$ consists of the terms in $W_{(x_{1},\cdots,x_{n})}$ with minimum valuation coefficients. Note that up to a common factor, $W_0$ is essentially a Laurent polynomial on $(\mathbb{C}^\times)^n$, or $T^{-\delta} W_0 \in \mathbb{C}[\uz_{1}^{\pm}, \cdots, \uz_{n}^{\pm}]$ for the lowest valuation $\delta$ of $W$.

One useful consequence of Lemma \ref{Woodward} (a) is that the tropicalization of $W$ confines the possible locations of critical points of $W \in \Lambda [z_{1}^\pm, \cdots, z_{n}^\pm]$. Primarily, we cannot have any critical points that project to $\mathbb{R}^n \setminus \mathrm{Trop}(W)$ under $val$. Indeed, as observed in subsection \ref{subsec:tropdual}, the lowest order terms $W_0$ of $W_{(x_{1},\cdots, x_{n})}$ for $(x_{1},\cdots, x_{n}) \in \mathbb{R}^n \setminus \mathrm{Trop}(W)$ consists a single monomial assigned to the connected component of $\mathbb{R}^n \setminus \mathrm{Trop}(W)$ containing $(x_{1},\cdots,x_{n})$. Clearly, a single Laurent monomial cannot have any critical point in $(\Lambda_U)^n$. 

In general, consider any $k$-dimensional face $E$ of $\mathrm{Trop}(W)$ for $k>0$. Along $E$, the lowest order terms $W_0$ of the restriction $W_{(x_{1},\cdots,x_{n})}$ ($(x_{1},\cdots,x_{n}) \in E$) are those monomials assigned to the components of $\mathbb{R}^n \setminus \mathrm{Trop}(W)$ whose boundaries contain $E$. Note that the exponent $(v_{1},\cdots, v_{n})$ of any such monomial in $W_0$ is orthogonal to the $k$-dimensional linear subspace of $\mathbb{R}^n$ parallel to $E$. (They lie in a cell of $\mathcal{S}_W$ orthogonal to $E$.) Therefore if $W_0$ of the restriction of $W$ along $E$ is convenient, then generically it has no critical points since its Newton polytope is degenerate. Hence we have proven:

\begin{prop}\label{prop:tropcrit}
Let $W$ be a Laurent polynomial over $\Lambda$, $W \in \Lambda [ z_{1}^\pm,\cdots,z_{n}^{\pm}]$.
\begin{enumerate}
\item[(a)] If $z\in (\Lambda^\times)^n$ is a critical point of $W$, then $val(z) \in Trop (W)$.
\item[(b)]
Suppose that the lowest order terms of $W$ restricted at a point in the interior of a positive dimensional cell $E$ of $\mathrm{Trop}(W)$ is convenient. Then the interior of $E$ cannot support a critical point of $W$. 
\item[(c)]
In particular, if the lowest order terms of $W$ is convenient on the interior of any positive dimensional cell, then the critical point of $W$ should map to one of vertices of $\mathrm{Trop}(W)$ under $val$.
\item[(d)]
Suppose that the lowest order terms of $W_{(x_{1},\cdots,x_{n})}$ is convenient for some vertex $V= (x_{1},\cdots,x_{n})$ of $\mathrm{Trop}(W)$. Then for generic coefficients, there exists exactly $|V|$-many non-degenerate critical points over the vertex $V$.
\end{enumerate}
\end{prop}

In locating the critical point of a Laurent polynomial $W$, the ideal situation is therefore when $W$ satisfies the condition in Proposition \ref{prop:tropcrit} (c), and additionally, the restriction of $W$ at each vertex of $\mathrm{Trop}(W)$ is convenient as in (d). 

\begin{definition}\label{def:locconv}
A Laurent polynomial $W$ over $\Lambda$ is said to be \emph{locally convenient} if it satisfies the condition in (c) and (d) of Proposition \ref{prop:tropcrit}, i.e., any subcollection of monomials in $W$ determined by a cell in the subdivision $\mathcal{S}_W$ of its Newton polytope forms a convenient Laurent polynomial. 
\end{definition}

In this case, for generic coefficients, the critical points can be essentially read off from the decomposition $\mathcal{S}_W$ of the Newton polytope of $W$. Recall  $|V| = n! Vol_{n} (F^\ast)$ where $F^\ast$ is a top cell in $\mathcal{S}_W$ dual to $V$.

\begin{remark}
We remark that the condition in Proposition \ref{prop:tropcrit} (c) is not guaranteed for generic coefficients. For instance, $W = a z_{1} +  \frac{b}{z_{1}} +c z_{2} +  \frac{d}{z_{2}}$ (the mirror $\mathbb{P}^1 \times \mathbb{P}^1$) satisfies the condition only when $a=b=c=d$. Such cases will be dealt with separately.
\end{remark}

A more straightforward approach in locating critical points would be to investigate the intersection $\mathrm{Trop}(z_{1}\frac{\partial}{\partial z_{1}}W) \cap \mathrm{Trop}(z_{2}\frac{\partial}{\partial z_{2}}W)$ of the tropicalization of the log partials of $W$, as in \cite{KLS}. However, it poses some difficulties for our purpose, such as ensuring transversality of the intersection or determining how many (if any) critical points are supported by non-displaceable fibers and if they are non-degenerate.



\section{Singularities of mirror Landau-Ginzburg potentials}\label{sec:singLG}

Combining the existing works of \cite{FOOO10b} and \cite{Bayer}, one has that for a toric surface $X_\Sigma$, the bulk-deformed potential $W^{\mathfrak{b}}_{\Sigma}$ has as many non-degenerate critical points as the rank of quantum cohomology of $X_\Sigma$ (for generic parameters).
The main objective of this section is to generalize this result to a larger class of algebraic surfaces, mainly non-toric blowups $X$ of toric surfaces, and establish closed string mirror symmetry. While our main result pertains to bulk-deformed potentials (with restricted classes of bulk-insertions when $X$ is non-Fano), we initially concentrate on potentials without bulk for sake of simplicity. Towards the end of the section, we extend our findings to encompass potentials with bulk $\mathfrak{b} \in H^{4}(X;\Lambda_{+})$.

Our key observation is presented in Theorem \ref{prop:pindownclasses}, which essentially asserts that a holomorphic disk with ``excessive energy'', such as disks with sphere bubbles, has little contribution to geometric critical points (Definition \ref{def:geomcrit}) or, at least no contribution to their locations. This allows us to handle situations where the full potential may be unknown.

For instance, we will see (Theorem \ref{prop:pindownclasses}) that the ``minimal energy terms of $W$" essentially determine the critical behavior of $W$. These terms consist of the Hori-Vafa part of $W$ together with contributions from what we will call basic broken disks (see Definition \ref{def:basicbd} below). We apply the combinatorial tools developed in Section \ref{sec:tropcrit} to this leading order part of $W$.\footnote{In the absence of non-degeneracy, however, one needs to take into account the second order terms as well, which requires a more delicate control on energies. See Proposition \ref{prop:gnc} below.}
We remark that Gonzalez-Woodward \cite{GW} also used a similar method to deduce the analogous result for the Hori-Vafa mirror potentials of toric orbifolds.

We first spell out our geometric setup in detail, mainly to fix notations and terminologies.


\subsection{Geometric setup}

Let $(X_{\Sigma},D_\Sigma)$ be a smooth compact toric surface (not necessarily Fano) from the complete fan $\Sigma$ generated by $\nu_i$, $i=1,\cdots, N$. Denote by  $D_{\Sigma,1},\dots, D_{\Sigma, N}$ their corresponding (irreducible) toric divisors. Then its mirror potential takes the form of
\begin{equation}\label{eqn:potential1}
	W_{\Sigma} = \underbrace{\sum_{i=1}^{N} a_i z^{\nu_{i}}}_{ W^{HV}_{\Sigma} } + \sum_{v} a_{v}z^{v}.
\end{equation}
The first summand is the Hori-Vafa potential $W^{HV}_{\Sigma}$, contributed by disks intersecting one of $D_{\Sigma,i}$ (see Example \ref{ex:toric}), whereas the second summand consists of disks with sphere bubbles whose disk component possibly has a higher Maslov index. If we set its moment polytope to be
\begin{equation}\label{eqn:momentpoly}
 \Delta_{\Sigma}:= \bigl\{ x : \langle x, \nu_i \rangle \geq -\lambda_i \quad i=1, \cdots, N\bigr\}
\end{equation}
as in Example \ref{ex:toric} (which amounts to fixing symplectic affine coordinates), then $a_i$ is given by $a_i = T^{\lambda_i}$.  
We will frequently use the map 
\[
	\val : (\Lambda^\times)^2  \to \mathbb{R}^2,
\]
and a point $z$ will be called \emph{geometric} in the sense that it corresponds to a Lagrangian torus fiber (see \eqref{eqn:mirrorzi}) if $x=\val(z)$ satisfies \eqref{eqn:momentpoly}, i.e. $x$ lies in $\Delta_\Sigma$. 
Note that this is consistent with Definition \ref{def:geomcrit}.

A sequence of non-toric blowups on $X_\Sigma$ gives rise to a surface $X$, coupled with an anticanonical divisor $D$ proper-transformed from $D_\Sigma$. We require that the symplectic sizes of all exceptional divisors are \textit{sufficiently small}, and that they are generic compared to the toric divisors. That is, we choose our symplectic form $\omega$ on $X$ in the class
\[
	[\omega]=[\pi^\ast \omega_\Sigma - \sum e_i E_i ]
\]
as in \ref{subsec:toricmodel} where $E_i$'s are exceptional classes, and we assume $\epsilon_i$'s are generic and arbitrarily small. We have a Lagrangian torus fibration with nodal fibers (in one-to-one correspondence with $\{E_i\}$) which is special away from a small neighborhood of the branch-cut at each nodal fiber. 

Consider a Maslov $2$ disk $u$ in $X$ belonging to the class $\widetilde{\beta_{\nu_j}} + \sum_{i} a_i \beta_i$, where $\widetilde{\beta_{\nu_j}} $ represents the proper transform of the basic disk $\beta_{\nu_j}$ that intersects $D_{\Sigma,j}$ exactly once away from the blowup center, and $\beta_i$ is the proper transform of Maslov $2$ disk passing through the blowup point associated with $E_i$ exactly once (see \ref{subsec:LGtoricmodel}). Then the area of $u$ with respect to the above choice of symplectic form is 
\begin{equation}\label{eqn:areabrd}
\omega_\Sigma(\beta_{\nu_j}) + \sum_i a_i \omega_{\Sigma} (\beta_i) - \sum_i a_i \epsilon_i.
\end{equation}
Here, we used the same notation $\beta_i$ to denote the disk class in  $X_{\Sigma}$ which is proper-transformed to $\beta_i$.
Notice that the symplectic area of $u$ is strictly smaller that that of its projection to $X_\Sigma$ (the sum of first two terms in \eqref{eqn:areabrd}), and the area difference limits to zero as $\epsilon_i$ goes to zero. 
Under holomorphic/tropical correspondence, such $u$ corresponds to a broken line whose infinite edge is parallel to $\nu_j$, and bends nontrivially at the wall induced by $E_i$ if $a_i \neq 0$. To distinguish from basic Maslov $2$ disks lifted from $X_\Sigma$, we will refer to such disks as \emph{broken disks} (of Maslov index $2$).

Among Maslov $2$ broken disks in $X$, we will be particularly interested in those described as follows. Recall that the scattering diagram associated with a special Lagrangian fibration on $X\setminus D$ has a special chamber $R_0$ enclosed by groups of parallel initial rays (Lemma \ref{lemma:centralchamber}). Let $p_{i;1},\cdots,p_{i;l}$ be the points of intersection of the blowup center and the toric divisor $D_{\Sigma,i}$, each of which gives rise to a wall (a ray)  parallel to $\nu_i$. Due to our specific choice of locations of $p_{i;j}$ (as outlined in \ref{subsec:toricmodel}, Figure \ref{fig:blowuppts1}), these walls are  close enough to the divisor $D_{\Sigma,i-1}$ where we label generators $\nu_i$ of $\Sigma$ counterclockwise. Then the basic disks stemming from $D_{\Sigma,i-1}$ are glued with Maslov $0$ disks (proper transforms of basic disks in $X_{\Sigma}$ hitting $p_{i;j}$'s), resulting in Maslov index 2 broken disks that could potentially enter the region $R_0$. Tropically, these disks are represented by broken lines approaching infinity along the $-\nu_{i-1}$-direction and bending only when encountering the walls induced by the blowups at $p_{i;j}$'s.

\begin{definition}\label{def:basicbd}
Let $p_{i;1},\cdots,p_{i;l}$ be the part of the blowup center lying in $D_{\Sigma,i}$ for $\pi:X \to X_\Sigma$. A \emph{basic broken disk} is defined to be the Maslov $2$ holomorphic disk in $X$ which is a proper transform of a disk in $X_\Sigma$ that hits $D_{\Sigma,i-1}$ exactly once and intersects $D_{\Sigma,i}$ possibly multiple times but only at $p_{i;j}$'s. See Figure \ref{fig:basicbrokendisks}.
\end{definition}

Observe that basic broken disks are in the class $\beta_{\nu_{i-1}} + k \beta_{\nu_i}$ (more precisely its proper transform under $\pi:X \to X_{\Sigma}$) where $k$ is the number of bends in their corresponding broken lines. These disks are responsible for the terms $z^{\partial(\beta{\nu_{i-1}} + k \beta_{\nu_i})} = z^{\nu_{i-1} + k \nu_i}$ in the potential. We will see that these basic broken disks have minimal energies among all Maslov $2$ disks in $X$ when the Lagrangian boundary is positioned close enough to the corner formed by $D_{\Sigma,i-1}$ and $D_{\Sigma,i}$.

\begin{figure}[h]
\centering
\subcaptionbox{%
	Tropical disks corresponding to basic broken disks.
	\label{basicbrokendisks}%
	}
	{\makebox[0.45\linewidth][c]
		{\begin{tikzpicture}[scale=0.65]
				
		\draw[line width=0.7pt, dashed] (-1,-2)--(-1,6);
		\draw[line width=0.7pt, dashed] (1,-2)--(1,6);
		\draw[line width=1pt] (-4,0)--(-1,0)--(1,2)--(2.5,5);
		\draw[line width=1pt] (-4,2)--(-1,2)--(2.5,5);
		\filldraw (2.5,5) circle (1pt);
		\draw[line width =0.7pt] (-1,-2) node[cross] {};
		\draw[line width =0.7pt] (1,-2) node[cross] {};

		\end{tikzpicture}%
	}
	}\hfill%
\subcaptionbox{%
	Tropical disks corresponding to non-basic broken disks.
	\label{nonbasicbrokendisks}%
	}
	{\makebox[0.45\linewidth][c]	
		{\begin{tikzpicture}[scale=0.65]
		
		\draw[line width=0.7pt, dashed] (-1,-2)--(-1,6);
		\draw[line width=0.7pt, dashed] (-4,1)--(4,1);

		\draw[line width=1pt] (-4,0)--(-1,0)--(0,1)--(0,4);

		\filldraw (0,4) circle (1pt);
		\draw[line width =0.7pt] (-1,-2) node[cross] {};
		\draw[line width =0.7pt] (4,1) node[cross] {};
	
		\end{tikzpicture}
		}
	}%
\caption{}
\label{fig:basicbrokendisks}
\end{figure}

In the chamber $R_0$, the symplectic affine coordinates are pulled back from $X_\Sigma$, providing the mirror coordinates $z=(x,y)$. Once coordinates are fixed, one can write down the projection of $z \in \check{Y}$ onto the base $B_{reg}$ as a map from $\check{Y}$ to $\mathbb{R}^2$, which we still denote by $\val$. As discussed in \ref{subsec:geomcrit}, this is possible away from a small neighborhood of the branch-cut from each nodal fiber.

Throughout this section, we will work with the potential obtained by counting disks bounding Lagrangians sitting over this chamber. As discussed in \ref{subsec:geomcrit}, geometric interpretation, within the context of SYZ mirror symmetry, is present only when considering geometric critical points of the potential $W$ that resides in the SYZ base. Namely, while $W$ can a priori be defined on a bigger domain, we will only be interested in critical points whose valuations live in $B_{reg} \approx \Delta_\Sigma$, the moment polytope of $X_\Sigma$. It turns out that when $R_0$ sufficiently covers the interior of the moment polytope (expressed in symplectic affine coordinates), it can be proven that every geometric critical point indeed lies within $R_0$ (Lemma \ref{lem:nooutr0}).



\subsection{Geometric critical points and their energy minimizing disks.}

As mentioned earlier, working with non-Fano surfaces poses several challenges, the obvious one being sphere bubbling phenomena due to divisors with negative Chern number. This results in (possibly infinite) correction terms being added to the Hori-Vafa potential, and it is difficult to write down explicit formulas for the full potential. While there are a few examples where such formulas have been obtained, they are often limited to specific examples or require indirect methods (see \cite[3.2]{auroux09} for an example). 

Surprisingly, however, it turns out that the few terms with minimal energy determine the critical behavior of the potential in our situation. In fact, the precise expressions of higher energy terms in the potential are not required for our purpose, as far as the locations of critical points are concerned.
%

We will use (a) of Proposition \ref{prop:tropcrit} crucially, which says that, under $\val:X\setminus D \to B$ (defined away from branch-cuts),
a critical point of $W$ maps to a point where at least two (or three if $\mathcal{S}_{W}$ is  locally convenient) Maslov 2 disks attain the minimum energy simultaneously. More concretely, if $z=\alpha$ is a critical point of $W$, then there exists distinguished classes $\beta_{1},\cdots,\beta_{m \geq 2} \in H_{2}(X,L_{\val(\alpha)})$ supporting Maslov index 2 (stable) disks, such that at $z=\alpha$,
\begin{equation}\label{eqn:minenergyxi}
	 \omega(\beta_{1}) = \cdots =  \omega(\beta_m) \leq  \omega(\beta),
\end{equation}
for all  $\beta$ supporting Maslov index $2$ disks. For obvious reasons, these disks of class $\beta_{1},\dots,\beta_m$ will be referred to as \emph{energy minimizing disks} at $\alpha$.
They give rise to monomials $T^{\delta (\beta_i)} z^{\partial \beta_i}$ for $1\leq i \leq m$ in the potential, where
\begin{equation}\label{eqn:deltabetai}
\delta (\beta_i) := \omega(\beta_i) - \langle \val(\alpha), \partial \beta_i \rangle.
\end{equation}
For instance, we have $\delta(\beta)= \lambda_i$ when $\beta$ is a proper transform of a basic disk in class $\beta_{\nu_i}$, but does not intersect any points in the blowup center. For simplicity, disks in such classes will still be referred to as basic disks when  there is no danger of confusion.
In general, $\delta(\beta)$ for a class $\beta$ supporting a stable holomorphic disk away from exceptional divisors is a linear combination of $\lambda_j$ and $\omega(D_{\Sigma,j})$ ($j=1,\cdots,N$) over nonnegative integers.
%
%

\begin{prop}\label{prop:pindownclasses}
Let $\alpha$ be a geometric critical point of $W$, and let $\beta_{1},\cdots, \beta_m$ be energy minimizing disks given as in \eqref{eqn:minenergyxi}, which are responsible for the terms $T^{\delta (\beta_{1})} z^{\partial \beta_i}, \cdots, T^{\delta(\beta_m)} z^{\partial \beta_i}$ with $\delta (\beta_i)$ given in \eqref{eqn:deltabetai}. Then $\beta_i$ should be of the class of a basic disk (from $X_\Sigma$) or a basic broken disk (Definition \ref{def:basicbd}). In particular, it cannot have any sphere bubbles.
\end{prop}

\begin{proof}
Since we are only interested in geometric critical points, the valuation $\val(\alpha)$ of $\alpha$ must satisfy 
\begin{equation}\label{condi:geo}
	\langle \val(\alpha), \nu_i \rangle  + \lambda_i \geq 0
\end{equation}
for all $i$ so that it lies inside $\Delta_\Sigma$ \eqref{eqn:momentpoly}.
On the other hand, applying \eqref{eqn:minenergyxi} to the basic disks in $X_\Sigma$ (or more precisely, their proper-transforms in $X$), we have
\begin{equation}\label{condi:trop}
\langle \val(\alpha), \partial \beta_{1} \rangle + \delta(\beta_{1})=\cdots =\langle \val(\alpha), \partial \beta_m \rangle + \delta(\beta_m) 
\leq \, \langle \val(\alpha), \nu_i \rangle + \lambda_{i}   
\end{equation}
for $i=1,\cdots, N$.

Since $\nu_{1},\cdots,\nu_N$ generates a complete fan $\Sigma$, $\partial \beta_i := (p,q)$ is contained in some cone of the fan, say, the cone spanned by $\nu_{1}$ and $\nu_{2}$ denoted by 
\[
	\mathrm{Cone}(\nu_{1},\nu_{2}): =\bigl\{ a \nu_{1} + b \nu_{2} : a,b \in \mathbb{Z}_{\geq 0} \bigr\}.
\]
After appropriate affine coordinate change, we may assume that $\nu_{1}=(1,0)$, $\nu_{2}=(0,1)$ and $\lambda_{1}=\lambda_{2}=0$. In this case all  other $\lambda_i$'s must be strictly positive, as they represent (the limits of) energies of basic disks $\beta_{\nu_{i}}$ whose boundary lies in the fiber over the origin of the moment polytope, stemming from a divisor inside the first quadrant of $\mathbb{R}^2$.
Since $\delta(\beta_i)$ is a linear combination of $\lambda_j$'s and $\omega(D_{\Sigma,j})$'s with nonnegative coefficients, it follows that $\delta(\beta_i) \geq 0$.
 
Applying the inequalities \eqref{condi:geo} and \eqref{condi:trop} to $\nu_{1}=(1,0)$ and $\nu_{2}=(0,1)$, we see that $(x_{0},y_{0}):=\val(\alpha)$ satisfies
\begin{equation}\label{condi:standard}
	x_{0} \geq 0, \quad y_{0} \geq 0, \quad px_{0} +qy_{0}+\delta(\beta_i) \leq x_{0}, \quad px_{0} +qy_{0}+\delta(\beta_i) \leq y_{0}, \quad p,q \geq 0.
\end{equation}

We claim that $\alpha$ is a geometric critical point only if $\delta(\beta_i) = 0$. We first assume that the sizes of exceptional divisors arising from the non-toric blowup $X \to X_\Sigma$ is zero, $\epsilon_{j}=0$, which can be thought of as $X$ equipped with a degenerate K\"{a}hler form.
\vspace{0.15cm}

\noindent \emph{(Case I: $p=0$).} Note that since $(0,0) \notin \support W$, we must have $q \geq 1$. 
\[
	px_{0} +qy_{0}+\delta(\beta_i) \leq y_{0}  
\,\, \Leftrightarrow \,\, \underbrace{(q-1)y_{0}}_{\geq 0} + \underbrace{\delta(\beta_i)}_{\geq 0} \leq 0 
\,\, \Leftrightarrow \,\, (q-1)y_{0}=\delta(\beta_i)=0.
\]
\noindent \emph{(Case II: $p\geq 1$).} Similarly,
\[
	px_{0} +qy_{0}+\delta(\beta_i) \leq x_{0}
\, \Leftrightarrow \, \underbrace{(p-1)x_{0}}_{\geq 0} + \underbrace{qy_{0}}_{\geq 0} + \underbrace{\delta(\beta_i)}_{\geq 0} \leq 0
\, \Leftrightarrow \,(p-1)x_{0} = qy_{0} = \delta(\beta_i) = 0.
\]
 
If $\epsilon_j$ is nonzero, then the left hand side of inequalities above additionally have $-\epsilon_j$. Still, it can be negative only when $\delta(\beta_i)=0$ since any nonzero $\delta(\beta_i)$ is dominantly positive compared to $\epsilon$ which is arbitrarily small.

Notice that $\delta(\beta_i) = 0$ if and only if $\pi_\ast(  \beta_i) (\in H_{2}(X_\Sigma, \pi(L_u))) $ is a linear combination of $\beta_{\nu_{1}}$ and $\beta_{\nu_{2}}$'s. First of all, we see that $\beta_i$ cannot involve sphere-bubble components, as otherwise it could only be attached with some exceptional classes (in $\ker \pi_\ast$) which have positive Chern numbers contradicting $\mu(\beta_i)= 2$. 

If $\pi_\ast (\beta_i) = \beta_{\nu_{1}}$ or $\pi_\ast(\beta_i)=\beta_{\nu_{2}}$, then it is a basic disk.
Now suppose that $\pi_\ast(\beta_i)=a_{1} \beta_{\nu_{1}} + a_{2} \beta_{\nu_{2}}$ with $a_{1},a_{2} >0$. In order for $\beta$ to be of Maslov $2$ in $X$, the only possibility is that all of $\beta_{\nu_{1}}$ and $\beta_{\nu_{2}}$ involved are vanishing thimbles emanating from nodal fibers (initial rays) except one which intersects $D_{\Sigma}$ once. Due to our special choice of the chamber $R_0$, this happens only when $a_{1}=1$ and $\beta_i$ is a class of a basic broken disk. It is obvious that basic broken lines are the only broken lines bending only at walls normal to $D_{\nu_{1}}$ and $D_{\nu_{2}}$ still entering $R_0$. In conclusion,
 \begin{equation}\label{eqn:possiblemin}
\pi_\ast(\beta_i)=\beta_{\nu_{1}}, \beta_{\nu_{2}}, \beta_{\nu_{1}} + k \beta_{\nu_{2}}
\end{equation}
where $k\geq 1$ is at most the number of blowup points located in $D_{\Sigma,2}$.
Applying the above argument on each cone $\mathrm{Cone}(\nu_{j},\nu_{j+1})$, our proof is complete.

\end{proof}

We denote by $W_{\mathfrak{min}}$ the sum of contributions of basic disks and basic broken disks. 
In other words, this is the collection of all (low energy) terms of $W$ that can possibly have the minimal energy at some geometric critical points.
More concretely, $W_{\mathfrak{min}}$ can be expressed as follows. Let $\nu_{1},\cdots, \nu_N$ in the fan $\Sigma$ be counterclockwisely ordered. We define $\mathcal{A}_j$ as
\[
	\mathcal{A}_{j} := \Bigl\{\, \beta_{\nu_{j}}+k \beta_{\nu_{j+1}} \,\, : \,\, 1 \leq k \leq \left|\mbox{blowup points in} \,\, D_{\Sigma,j+1}\right| \,\Bigr\}  
\]
so that $\partial \mathcal{A}_j \subset \mathrm{Cone}(\nu_{j},\nu_{j+1})$. In other words, $\mathcal{A}_{j}$ is the set of all disk classes with boundaries in $\mathrm{Cone}(\nu_j,\nu_{j+1})$ that can be represented by basic broken lines (bending at walls perpendicular to $D_{\Sigma,j+1}$). Then we have
\[
	W_{\mathfrak{min}} := W^{HV}_{\Sigma} + \sum_{j=1}^{n}\sum_{\pi _\ast (\beta) \in \mathcal{A}_{j}}N_\beta T^{\delta_{\beta}}z^{\partial\beta}
\]
where, as before, $N_\beta$ is the count of disks in class $\beta$ passing through a generic point in $L_u$. Observe that we can have $z^{\partial \beta} = z^{\partial \beta'}$ for different $\beta$ and $\beta'$. This happens when both $\beta$ and $\beta'$ are basic broken disks which project to the same class in $X_\Sigma$ (i.e., $\pi_\ast (\beta) = \pi_\ast (\beta')$), but passing through different points in the blowup center. 
Exponents appearing in $W$ are boundaries of classes in \eqref{eqn:possiblemin}. Note that the expression of the leading terms $W_\mathfrak{min}$ in $W$ remains constant on $R_0$ even when $X$ is non-Fano as mentioned in Remark \ref{rmk:r0nonfano}.

On the other hand, by the duality in Proposition \ref{prop:Duality}, a monomial of $W$ is assigned to each chamber adjacent to a given point $\val(\alpha) \in \mathrm{Trop}(W)$ in such a way that it has a minimal valuation at $\alpha$ among all monomials of $W$. Proposition \ref{prop:pindownclasses} implies that any such monomial should appear in $W_{\mathfrak{min}}$ if $\alpha$ is a geometric critical point. Thus we have:

\begin{cor}\label{cor:Wmin}
For a geometric critical point $\alpha$ of $W$, $\val(\alpha)$ lies on $\mathrm{Trop}(W_{\mathfrak{min}})$. 
\end{cor} 

\begin{remark}\label{rmk:Wmin}
Indeed, the above discussion implies that $\mathrm{Trop}(W_{\mathfrak{min}})$ agrees with $\mathrm{Trop}(W)$ on a compact region containing all the geometric critical points. For this reason, we will not distinguish these two when there is no danger of confusion.
\end{remark}

According to Proposition \ref{prop:pindownclasses}, we can divide geometric critical points of $W$ into two different types depending on whether its associated energy minimizing disks contain (basic) broken disks or not. 
Let us first consider the case where energy minimizers include at least one basic broke disk. Such a geometric critical point will be called \emph{non-toric critical points}. 
We have the following lemma for  non-toric critical points, by further examining the inequalities of \emph{(Case I)} and \emph{(Case II)} in the proof of Proposition \ref{prop:pindownclasses}.

\begin{lemma}\label{lemma:interiorcritpts}
Let $\alpha$ be a geometric critical point of $W$. If one of energy minimizing disks at $\alpha$ maps to a basic broken disk of class $\beta_{\nu_{j}}+k \beta_{\nu_{j+1}}$ under $\pi_\ast$, then $\alpha$ lies arbitrarily close to a corner $D_{\Sigma,j} \cap D_{\Sigma,j+1}$ of the moment polytope $\Delta_\Sigma$.
\end{lemma}

\begin{proof}
We first prove that $\val(\alpha)=(x_0,y_0)$ sits exactly at the corner when the sizes of the exceptional divisors are $0$. 
As before, we assume $\partial \beta_i \in \mathrm{Cone} (\nu_{1},\nu_{2})$ and we set $\nu_{1} = (1,0)$, $\nu_{2} = (0,1)$ and $\lambda_{1}=\lambda_{2}=0$ by a suitable affine coordinate change.
If $\beta_i$ represents a basic broken disk $\pi_\ast(\beta_i) = \beta_{\nu_{1}} + k \beta_{\nu_{2}}$, then 
he inequality \eqref{condi:standard} becomes:
\[
	\underbrace{(1-1)x_{0}}_{=0} + ky_{0} + \underbrace{\delta(\beta_i)}_{=0} \leq 0, \quad 
	x_{0} + \underbrace{(k-1)y_{0}}_{=0} + \underbrace{\delta(\beta_i)}_{=0}\leq 0.
\]
We have $y_{0}=0,\,x_{0}=0$ from the first and second inequality, respectively. When exceptional divisors have positive, but small sizes, then $x_0$ and $y_0$ become small positive numbers, since the left hand sides of the inequalities additionally have $-\epsilon_j$'s (which account for $\omega(\beta_i)- \pi^\ast \omega_\Sigma (\beta_i) < 0$).
\end{proof}

This leads to the following classification of possible energy minimizing disks for a non-toric critical point. In fact, having at least one basic broken disk in the minimizers  significantly restricts types of other energy minimizing disks.

\begin{prop}\label{prop:non-toric-crits}
Let $\alpha$ be a critical point of $W$ such that one of associated energy minimizing disks lies in $\mathcal{A}_j = \Bigl\{\, \beta_{\nu_{j}}+k \beta_{\nu_{j+1}} \, : \, 1 \leq k \leq \left|\mbox{blowup points in} \,\, D_{\Sigma,j+1}\right| \,\Bigr\}$ after the projection $\pi_\ast$. Then $\alpha$ is geometric only if all the other minimizers have classes lying in $\mathcal{A}_j \cup \{ \beta_{\nu_j},\beta_{\nu_{j+1}}\}  $. 
\end{prop}

We will perform a detailed local model calculation for such non-toric critical points later in \ref{subsec:non-toriclocal}.

\begin{proof}

Without loss of generality, let $j=1$, i.e., there exists an energy minimizer at $\alpha$ lying in $\mathcal{A}_{1}$. As in the proof of Theorem \ref{prop:pindownclasses}, we may assume, after coordinate changes, that $\nu_{1}=(1,0)$, $\nu_{2}=(0,1)$, and $\lambda_{1}=\lambda_{2}=0$. By Lemma \ref{lemma:interiorcritpts}, we have $\val(\alpha)=(0,0)$ when the sizes of exceptional divisors limit to zero. 

Now suppose that there exists a class $\beta$ that has a minimal energy at the critical point $\alpha$. Then from \ref{condi:geo} and \ref{condi:trop}, we have:
\[
	 \langle \val(\alpha), \partial \beta \rangle +\delta(\beta)=\,\, \delta(\beta) \leq 0 \,\, = \langle \val(\alpha), (1,0)\rangle + \lambda_{1}
\]
at the limit. This implies $\delta(\beta) = 0$ which is possible only when $\pi_\ast(\beta) \in \mathcal{A}_{1}\cup \bigl\{ (1,0), (0,1)\bigr\}$ precisely by the same argument in the proof of Theorem \ref{prop:pindownclasses}. Therefore, in order for $\alpha$ to be geometric, the $\pi_\ast$-image of every minimizer should lie in $\mathcal{A}_{1} \cup \bigl\{ \beta_{(1,0)}, \beta_{ (0,1)}\bigr\}$.
\end{proof}

We next look into geometric critical points all of whose associated energy minimizing disks are basic disks. They will be referred to as \emph{toric critical points} making obvious distinction from non-toric critical points discussed previously.

In valuation perspectives, these toric critical points can be completely captured from critical points of $W_\Sigma^{HV}$ (Corollary \ref{cor:Wmin}), but there is still a subtlety especially when $X_\Sigma$ is non-Fano. Namely, the number of critical points of $W_\Sigma^{HV}$ calculated by  Kushnirenko's theorem (Theorem \ref{Kushnirenko}) still exceeds the expected number of (geometric) toric critical points. 
Indeed, as long as the resulting surface is semi-Fano, the toric blowup changes the Newton polytope by attaching a unimodal 2-cell whose vertices correspond to the two divisors containing the blowup point and the exceptional divisor. The $2$-cell has volume $1$ since the exceptional divisor has self-intersection $-1$. As a result, we obtain exactly one new non-degenerate critical point of $W_{\Sigma}^{HV}$ after blowup in this case. 

However, if $X_\Sigma$ becomes non-Fano after a blowup, the convexity of the Newton polytope forces a region of bigger volume to be added to the Newton polytope, thereby increasing the number of the critical points of $W_{\Sigma}^{HV}$ (as a function on $(\Lambda^\times)^2$) beyond the rank of $QH^\ast (X_\Sigma)$ (see Figure \ref{fig:nonfano}). We therefore show that no geometric critical points arises from such extra region in the Newton polytope (of a non-Fano surface). More concretely, this region consists of triangles with vertices $\nu_i$, $\nu_j$ and $\nu_k$ satisfying 
\begin{equation}\label{eqn:nonconvexnewton}
	\mathrm{det}\Big\rvert	\begin{matrix}	\nu_{i}- \nu_{j} \\ \nu_{k}- \nu_{j}	\end{matrix}	\Big\lvert >0 ,
\end{equation}
(see Figure \ref{fig:non-geo-crits1}) where $\nu_i$, $\nu_j$ and $\nu_k$ appearing in $\Sigma$ in the counterclockwise order .

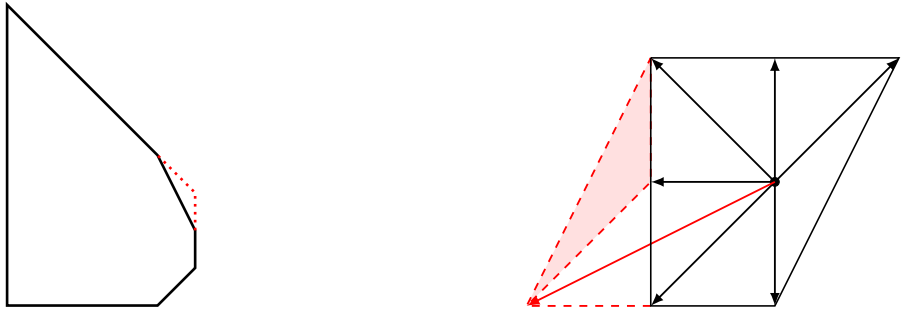
\begin{figure}[h]
\centering
\subcaptionbox{%
	The moment polytope when $X_{\Sigma}$ becomes non-Fano. The blowup is represented in red.
	\label{fig:nonfano1}%
	}
	{\makebox[0.45\linewidth][c]
		{\begin{tikzpicture}[scale=0.5]
				
			\draw[line width=1pt] (0,8)--(0,0)--(4,0)--(5,1)--(5,2)--(4,4)--cycle;
			\draw[line width=1pt, dotted, red] (5,2)--(5,3)--(4,4);
			\node[circle, opacity=0] (a) at (0,-1) {};
		
		\end{tikzpicture}%
	}
	}\hfill%
\subcaptionbox{%
	The Newton polytope when $X_{\Sigma}$ becomes non-Fano. The red lattice vector is added after blowup, and the volume of the Newton polytope is increased by 2.
	\label{fig:nonfano2}%
	}
	{\makebox[0.45\linewidth][c]	
		{\begin{tikzpicture}[scale=0.33]
		
				\filldraw (0,0) circle (5pt);
				\draw[line width =0.7pt, -latex] (0,0)--(5,5);
				\draw[line width =0.7pt, -latex] (0,0)--(0,5);
				\draw[line width =0.7pt, -latex] (0,0)--(-5,5);
				\draw[line width =0.7pt, -latex] (0,0)--(0,-5);
				\draw[line width =0.7pt, -latex] (0,0)--(-5,-5);
				\draw[line width =0.7pt, -latex] (0,0)--(-5,0);

				\draw[line width =0.7pt, -latex, red] (0,0)--(-10,-5);
				\fill[red!30, draw=none, opacity=0.4] (-5,5)--(-5,0)--(-10,-5);
				\draw[line width =0.7pt,red, dashed] (-5,5)--(-5,0)--(-10,-5)--cycle;
				\draw[line width =0.7pt,red, dashed] (-5,-5)--(-10,-5);
				\draw[line width=0.6pt] (-5,5)--(0,5)--(5,5)--(0,-5)--(-5,-5)--(-5,5);
				\node[circle, opacity=0] (a) at (0,-6.5) {};
		
		\end{tikzpicture}
		}
	}%
\caption{Excessive region added when $X_{\Sigma}$ becomes non-Fano.}
\label{fig:nonfano}
\end{figure}

\begin{remark}
The issue on the excess number of critical points of $W_{\Sigma}^{HV}$ in non-Fano situtation has already observed in \cite{GW} (in arbitrary dimension), where they excluded non-geometric critical points of $W_\Sigma^{HV}$ implicitly through a similar consideration on energies of disks.
We still work with the full potential $W$ while we reduce the argument to $W_\Sigma^{HV}$ by a refined estimate. 
\end{remark}

\begin{prop}\label{prop:toric-crits}
Let $\alpha$ be a critical point of $W$ such that every energy minimizing disks are basic disks. Then $\alpha$ is geometric only if for any three classes $\beta_{\nu_i}$, $\beta_{\nu_j}$, and $\beta_{\nu_k}$ among minimizers, the inequality
\[
	\mathrm{det}\Big\rvert	\begin{matrix}	\partial \beta_{\nu_i}-\partial\beta_{\nu_j} \\ \partial \beta_{\nu_k}- \partial \beta_{\nu_j}	\end{matrix}	\Big\lvert  = 	\mathrm{det}\Big\rvert	\begin{matrix} {\nu_i}- {\nu_j} \\  {\nu_k}-  {\nu_j}	\end{matrix}	\Big\lvert \leq 0,
\]
holds, whenever
$\nu_i$, $\nu_{j}$ and $\nu_{k}$ are arranged in $\Sigma$ in the counterclockwise order. 
\end{prop}

We emphasize that $\nu_{i}$, $\nu_{j}$, and $\nu_{k}$ need not be adjacent.

\begin{proof}
Let $\alpha$ be a geometric critical point of $W$ with $\val (\alpha)=(x_{0},y_{0})$. Choose any energy minimizing basic disks whose classes project to $\beta_{\nu_i}$, $\beta_{\nu_j}$ and $\beta_{\nu_k}$ under $\pi_\ast$, and suppose that their boundaries $\nu_i$, $\nu_j$ and $\nu_k$ appear in counterclockwise order. 
After appropriate coordinate changes, one has $\nu_i=(a,b)$, $\nu_{j}=(0,1)$ and $\nu_{k}=(c,d)$ with $a>0$ and $c<0$ (so that they are in counterclockwise order). Note that $a\neq 0$, $c \neq 0$. We may assume that the edge of the moment polytope $\Delta_\Sigma$ dual to $\nu_j$ contains the origin in its interior by a suitable translation of symplectic affine coordinates, that is, $\lambda_j=\delta(\beta_j)=0$ and $\lambda_{l} > 0$ for all $l \neq j$. In particular, $\Delta_\Sigma$ lies in the upper half plane $\{y\geq0\}$. Figure \ref{fig:nongeotoric} shows the setup.

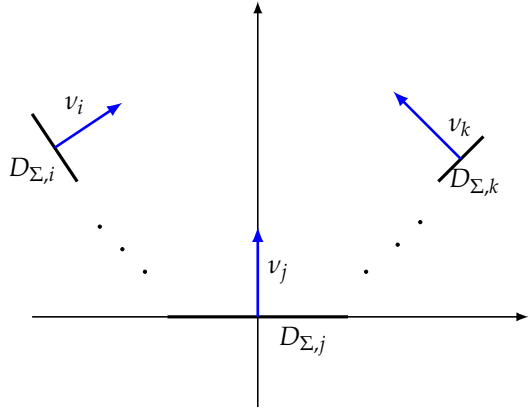
\begin{figure}[h]
\centering

\tikzset{cross/.style={cross out, draw=black, minimum size=2*(#1-\pgflinewidth), inner sep=0pt, outer sep=0pt},
cross/.default={3.5pt}}
\tikzset{crossb/.style={cross out, draw=blue, minimum size=2*(#1-\pgflinewidth), inner sep=0pt, outer sep=0pt},
crossb/.default={3.5pt}}

\begin{tikzpicture}[scale=0.6]

\draw[line width=0.6pt, -latex] (-5,0)--(6,0);
\draw[line width=0.6pt, -latex] (0,-2)--(0,7);

\draw[line width=1.2pt] (-2,0)--(2,0);
\draw[line width=1pt, -latex, blue] (0,0)--(0,2);
\draw (0,1) node[right] {\small{$\nu_{j}$}};
\draw (1,0) node[below] {\small{$D_{\Sigma,j}$}};

\filldraw (-3.5,2) circle (1pt);
\filldraw (-3,1.5) circle (1pt);
\filldraw (-2.5,1) circle (1pt);
\draw[line width=1.2pt] (-4,3)--(-5,4.5);
\draw (-5,3.75) node[below] {\small{$D_{\Sigma,i}$}};
\draw[line width=1pt, -latex, blue] (-4.5,3.75)--(-3,4.75);
\draw (-4.5,4.7) node[right] {\small{$\nu_{i}$}};

\filldraw (2.4,1) circle (1pt);
\filldraw (3.1,1.6) circle (1pt);
\filldraw (3.6,2.1) circle (1pt);

\draw[line width=1.2pt] (4,3)--(5,4);
\draw (4.8,3.5) node[below] {\small{$D_{\Sigma,k}$}};
\draw[line width=1pt, -latex, blue] (4.5,3.5)--(3,5);
\draw (4,4.2) node[right] {\small{$\nu_{k}$}};

\end{tikzpicture}

\caption{The moment polytope after coordinate change.} \label{fig:nongeotoric}
\end{figure}

Now suppose that on the contrary, $\mathrm{det}\Big\rvert	\begin{matrix}	\nu_{i}-\nu_{j} \\ \nu_{k}-\nu_{j}	\end{matrix}	\Big\lvert = \mathrm{det}\Big\rvert	\begin{matrix}	a & b-1 \\ c & d-1	\end{matrix}	\Big\lvert > 0$. Since we begin with disk classes 
simultaneously attaining minimum energy, \eqref{condi:trop} reads in our case 
\[
	y_{0}=ax_{0}+by_{0}+ \lambda_i =cx_{0}+dy_{0}+ \lambda_k
\]
Solving this equation, we find
\[
	y_{0}= \frac{c \, \lambda_k - a \, \lambda_i}{a(d-1)-c(b-1)}<0
\]
since $\lambda_i, \lambda_k >0$ due to our choice of symplectic affine coordinates, and hence $(x_0,y_0)$ lies outside the moment polytope. Thus, $\alpha$ cannot be geometric. 
\end{proof}
The part that has to be ruled out by Proposition \ref{prop:toric-crits} is shown in Figure \ref{fig:non-geo-crits1}, whereas Figure \ref{fig:non-geo-crits2} depicts the region ruled out by Proposition \ref{prop:non-toric-crits}. After filtering out all non-geometric critical points, what remains is the star-shaped region obtained by joining heads of $\nu_i$'s. 

It is worthwhile to mention that if we replace the ordinary convex hull (for the Newton polytope) by this star-shaped region, we can extract only the geometric critical points, removing any non-geometric ones. 
In fact, as we will see in the subsequent subsections, the (normalized) volume of this star-shaped region coincides with the expected number of geometric critical points.
In other words, this star-shaped replacement of the Newton polytope precisely excludes the region where the classical Kushnirenko Theorem fails in our geometric context.  From this perspective, our work so far can be seen as a modified  version of the Kushnirenko theorem concerning the number of \emph{geometric} critical points.

\begin{figure}[h]
\centering
\subcaptionbox{%
	Non-geometric critical point in Proposition \ref{prop:toric-crits}.
	\label{fig:non-geo-crits1}%
	}
	{\makebox[0.45\linewidth][c]
		{\begin{tikzpicture}[scale=0.17]
				
		\draw[line width =0.7pt, -latex] (0,0)--(-15,20);
		\draw (-7,9) node[left] {\small{$\nu_{k}$}};

		\draw[line width =0.7pt, -latex] (0,0)--(0,8);
		\draw (0,5.5) node[left] {\small{$\nu_{j}$}};

		\draw[line width =0.7pt, -latex] (0,0)--(8,8);
		\draw (5,3) node[right] {\small{$\nu_{i}$}};

		\draw[dashed, color=blue!50, line width = 0.5pt] (-15,20) -- (0,8) -- (8,8)--cycle;
		\fill[blue!30, draw=none, opacity =0.4] (-15,20) -- (0,8) -- (8,8);

		\node[circle, opacity=0] (a) at (-20,-2) {};
		\node[circle, opacity=0] (a) at (15,-2) {};

		\end{tikzpicture}%
	}
	}\hfill%
\subcaptionbox{%
	Non-geometric critical point in Proposition \ref{prop:non-toric-crits}.
	\label{fig:non-geo-crits2}%
	}
	{\makebox[0.45\linewidth][c]	
		{\begin{tikzpicture}[scale=0.16]
		
		\draw[line width =0.7pt, -latex] (0,0)--(8,0);
		\draw (5,0) node[below] {\small{$\nu_{1}$}};
		\draw[dashed, line width =0.7pt] (8,0)--(19,0);

		\draw[line width =0.7pt, -latex] (0,0)--(0,8);
		\draw (0,5) node[left] {\small{$\nu_{2}$}};
		\draw[dashed, line width =0.7pt] (0,8)--(0,30);

		\draw[line width =0.7pt, -latex] (0,0)--(-10,4);
		\draw (-5,1) node[below] {\small{$\nu_{3}$}};
		\draw[dashed, line width =0.7pt] (-10,4)--(-25,10);

		\draw[dashed, line width =0.7pt, color=red!50] (8,0)--(8,32)--(0,8);
		\draw[dashed, line width =0.7pt, color=red!50] (8,24)--(0,8)--(8,16);
		\draw[dashed, line width =0.7pt, color=red!50] (8,8)--(0,8)--(8,0);

		\draw[dashed, line width =0.7pt, color=red!50] (0,8)--(-20,16)--(-10,4);
		\draw[dashed, line width =0.7pt, color=red!50] (-10,12)--(-10,4);
		\draw[dashed, line width =0.7pt, color=red!50] (0,8)--(-10,4);

		\fill[red!30, draw=none, opacity =0.4] (8,0)--(8,32)--(0,8);
		\fill[red!30, draw=none, opacity =0.4] (0,8)--(-20,16)--(-10,4);

		\draw (3,15) node[right] {\small{$\mathcal{A}_{1}$}};
		\draw (-10,7) node[above] {\small{$\mathcal{A}_{2}$}};

		\draw[dashed, line width =0.7pt, color=blue!50] (8,32)--(-20,16);
		\fill[blue!30, draw=none, opacity =0.4] (8,32)--(-20,16)--(0,8);

		\node[circle, opacity=0] (a) at (0,-2) {};

		\end{tikzpicture}
		}
	}%

\caption{Classification of geometric critical points.}
\label{fig:classification}
\end{figure}
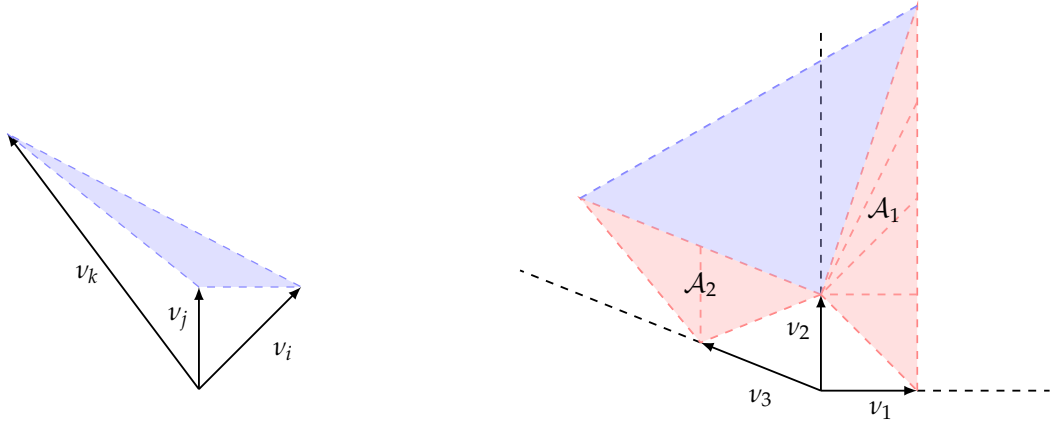


In fact, the following Lemma \ref{lem:unimodal1} asserts that for generic parameters, the Newton subdivision of this star-shaped Newton polytope is always a locally convenient unimodal triangularization (apart from the central cell(s) containing the origin).  Lemma \ref{lem:unimodal1} plays a crucial role in the proof of our main theorem Theorem \ref{main}, where Lemma \ref{lem:unimodal1} guarantees that generically, the potential $W_{\Sigma}$ of the toric model $(X_{\Sigma}, D_{\Sigma})$ is not only Morse, but also has distinct critical values.

\begin{lemma}\label{lem:unimodal1}
For generic parameters, the tropicalization $\mathrm{Trop}(W)$ (restricted to the compact region containing only geometric critical points) is trivalent at every vertex, all of which supports distinct critical points of $W$. In fact, for generic parameters, each vertex is of weight $1$, with the possible exception of the single vertex dual to the convenient $2$-cell $\sigma_{0}$ containing the origin in its interior, or the two vertices dual to the non-convenient $2$-cells adjacent to the edge containing the origin. 
\end{lemma} 
\noindent Recall that a tropical curve is called \emph{smooth} if all of its vertices are of weight $1$, or equivalently, if its dual Newton subdivision consists of only unimodal cells.
\begin{proof}
Note that by Corollary \ref{cor:Wmin} and Remark \ref{rmk:Wmin}, it suffices to work with $W_\mathfrak{min}$ instead of $W$.
Let $\mathcal{S}_{W_{\mathfrak{min}}}$ be the Newton subdivision of the leading term potential $W_{\mathfrak{min}}$. We denote by $\sigma_{0} \in \mathcal{S}_{W_{\mathfrak{min}}}$ the cell which contains the origin. Note that $\sigma_{0}$ can be either a $2$-cell or a $1$-cell (i.e. an edge containing the origin).

Proposition \ref{prop:non-toric-crits} and \ref{prop:toric-crits} guarantees that a $2$-cell $\sigma \in \mathcal{S}_{W_{\mathfrak{min}}}$ is convenient if $\sigma \neq \sigma_{0}$ when $\sigma_{0}$ is a $2$-cell (or $\sigma \cap \sigma_{0} = \emptyset$ when $\sigma_{0}$ is a $1$-cell). Hence for any such $2$-cell $\sigma$, there exists \textit{at least three} disk classes, say $\beta_{\nu_{i_{1}}}, \dots, \beta_{\nu_{i_{k}}}$, which simultaneously attain minimum valuation at $V=\sigma^{\ast}$. Let $\alpha$ be a geometric critical point supported at this vertex $V$. We have our usual energy minimizing equality
\[
	\langle \val(\alpha), \beta_{\nu_{i_{1}}} \rangle + \lambda_{i_{1}} = \cdots = \langle \val(\alpha), \beta_{\nu_{i_{k}}} \rangle + \lambda_{i_{k}}
	\leq \, \langle \val(\alpha), \nu_i \rangle + \lambda_{i}   
\]
for $i=1,\dots, N$. But since the defining equalities of the moment polytope 
\[
	\langle \val(\alpha), \nu_i \rangle  + \lambda_i \geq 0
\]
are open conditions, $\alpha$ would still be a geometric critical point even after we perturb $\lambda_{i}$'s slightly. Hence for generic coefficients, there are exactly three disks which has minimum energy at any given geometric critical point. In other words, the Newton subdivision $\mathcal{S}_{W_{\mathfrak{min}}}$ is generically a triangulation of $\Delta_{W_{\mathfrak{min}}}$ which is locally convenient, except possibly at the unique edge containing the origin. In terms of the tropicalization, this is equivalent to say that all vertices are trivalent.

We further claim that if the origin is not contained in (the closure of) the cell $\sigma \in \mathcal{S}_{W_{\mathfrak{min}}}$, then $\sigma$ is unimodal. Suppose that a $2$-cell $\sigma$ has $2\mathrm{Vol}(\sigma) >1$, and denote by $\nu_{i}$, $\nu_{j}$, $\nu_{k}$ its three vertices. Note that this implies that $\sigma$ contains at least one extra lattice point, say $\nu_{i_{0}}=(p,q)$, other than $\nu_{i}$, $\nu_{j}$, and $\nu_{k}$. But since any toric surface $X_{\Sigma}$ is obtained via a sequence of blowups from $\mathbb{F}_{k}$, it is clear that $\mathcal{S}_{W_{\mathfrak{min}}}$ can be obtained by continuously deforming an existing unimodal configuration (constructed later in Proposition \ref{prop:toricblowup} and Theorem \ref{thm:non-toricblowup}), during which no lattice points are lost in order for the fan to remain simplicial. 
Hence we may assume $\nu_{i_{0}} \in \support W_{\mathfrak{min}}$, which is to say that there exists a disk class $\beta_{i_{0}} \in H_{2}(X_{\Sigma}, L)$ with $\partial \beta_{i_{0}} =\nu_{i_{0}}$. Let us denote its corresponding term in the potential by $a_{i_{0}}z^{\partial \beta_{i_{0}}}$ where $a_{i_{0}} \in \Lambda$ with $\val(a_{i_{0}}) = \lambda_{i_{0}}$. Note that it suffices to show that $\nu_{i_{0}} \in \mathcal{S}_{W_{\mathfrak{min}}}$ has at least one connected edge.

Suppose that $\nu_{i_{0}}=(p,q) \in \mathcal{S}_{ W_{\mathfrak{min}} }$ has no connected component. In tropical geometry terms, this implies that $\val (a_{i_{0}}z^{\partial \beta_{i_{0}}})$ is never the minimum. Equivalently,
\[
	\langle (x,y), \,\nu_{i_{0}} \rangle + \lambda_{i_{0}} \geq \langle (x,y), \, \nu_{i} \rangle + \lambda_{i} \quad 
	\mbox{for all } (x,y) \in \mathbb{R}^{2}, \, i=1, \dots, N,
\]
where equality holds only if $i=i_{0}$. Notice that this is true only if $\lambda_{i_{0}} = \infty$, that is, $a_{i_{0}} = 0$. Hence $\nu_{i_{0}} = \partial \beta_{i_{0}} \notin \support W_{\mathfrak{min}}$ which is a contradiction.
\end{proof}

\begin{remark}\label{rmk:volmany}
Consider the case where there exists precisely three energy minimizing disks at a vertex $\alpha$ of $\mathrm{Trop}(W)$, and suppose that the corresponding triangle is convenient. It is easy to prove by direct calculation that there are as many \textit{distinct} critical points over $\alpha$ as the volume of the triangle. We will see in Proposition \ref{prop:gnc} that an analogous statement is true for non-convenient cells.
\end{remark}

\subsection{Toric blowups of $(X_{\Sigma}, D_{\Sigma})$}\label{subsec:toricblowup1}
Let us first consider the toric blowup of a given toric surface, and study how the critical behavior of its mirror (Floer) potential changes. Here, the toric blowup refers to a blowup of a toric surface $X_\Sigma$ at one of nodal points, say $D_{\Sigma,j} \cap D_{\Sigma,j+1}$, in the toric divisor $D_\Sigma$. It modifies the fan $\Sigma$ by adding additional primitive generator $\nu'$ between $\nu_j$ and $\nu_{j+1}$. Note that $\nu'$ is uniquely determined by requiring the resulting fan $\widetilde{\Sigma}$ is still simplicial: $\nu'=\nu_j + \nu_{j+1}$.
%

When performing (a sequence of) toric blowups, there is a caveat concerning the choice of size of the blowups. Namely, if $X_{\widetilde{\Sigma}}$ is obtained from a toric blowup on $X_{\Sigma}$ and if the blowup size is not ``small enough'',  the Newton subdivision $\mathcal{S}_{ W^{HV}_{\Sigma} }$ may not be preserved. In other words, $\mathcal{S}_{W^{HV}_{\widetilde{\Sigma}}}$ might not contain cells from $\mathcal{S}_{W_{\Sigma}^{HV}}$. This is problematic, since critical points a priori to blowups are no longer extended in a valuation-preserving manner. We therefore require the exceptional divisor to be ``small enough'', given precisely by \eqref{condi:blowup-size} during the proof of Proposition \ref{prop:toricblowup} .

\begin{remark}\label{rmk:Morse}
If $X_{\Sigma}$ is obtained from a sequence of blowups on a Hirzebruch surface $\mathbb{F}_{k}$, it would be convenient to have all blowup sizes sufficiently small. This ensures that the polyhedral decomposition $\mathcal{S}_{W_{\widetilde{\Sigma}}^{HV}}$ not only contains the cells in $\mathcal{S}_{W_{\Sigma}^{HV}}$ but also includes all cells that appear during the inductive sequence of blowups starting from $\mathcal{S}_{W_{\mathbb{F}_k}^{HV}}$. However, we will make no such assumptions, and only assume that the potential $W_{\Sigma}^{HV}$ is Morse. In fact, Lemma \ref{lem:unimodal1} and Proposition \ref{prop:gnc} tell us that $W_{\mathfrak{min}}$ (and hence $W_{\Sigma}^{HV}$) is indeed Morse for generic parameters, and that critical points of $W_{\mathfrak{min}}$ extend to those of its blowup even if they are not locally convenient.
\end{remark}

\begin{prop}\label{prop:toricblowup}
Let $(X_{\Sigma}, D_\Sigma)$ be a smooth toric surface and let  $X_{\widetilde{\Sigma}}$ be a toric surface obtained by blowing up $X_\Sigma$ at one of the nodal points $D_{\Sigma,j} \cap D_{\Sigma,j+1}$ in $D_\Sigma$. Then there exists $r>0$ depending on $\omega$ such that if the size of the exceptional divisor is less than $r$, then $ W_{\widetilde{\Sigma}} $ has precisely $1$ new geometric critical point which is non-degenerate. Moreover, if  $ W_{\Sigma} $ is Morse, then $ W_{\widetilde{\Sigma}} $ is also Morse.
\end{prop}

Here, `new critical point' means that apart from this, there is one-to-one correspondence between geometric critical points of $W_\Sigma$ and $W_{\widetilde{\Sigma}}$. We will see that the correspondence matches the lowest energy terms of the potential at the corresponding pair of critical points of $W_\Sigma$ and $W_{\widetilde{\Sigma}}$.

\begin{proof}

As before, it suffices to work with $W_{\mathfrak{min}}$ (which in this case is just the Hori-Vafa part) instead of $W_{\widetilde{\Sigma}}$. Without loss of generality, let $X_{\widetilde{\Sigma}}$ be the blowup of $X_\Sigma$ at the corner $D_{\Sigma,1} \cap D_{\Sigma,2}$ and assume $\nu_{1}=(1,0)$, $\nu_{2}=(0,1)$ and $\lambda_{1}=\lambda_{2}=0$. Then the new primitive generator of $\widetilde{\Sigma}$ has to be $\nu'=\nu_{1}+\nu_{2}=(1,1)$. Moreover,
the Hori-Vafa potential for $X_{\tilde{\Sigma}}$ with this choice of cooridnates is by definition given as
\begin{align}
W^{HV}_{\widetilde{\Sigma}}  &= W^{HV}_{\Sigma} +  T^{-\eta}z_{1}z_{2}  \nonumber \\
 &= z_{1} + z_{2} + T^{-\eta}z_{1}z_{2} + \sum_{i=3}^{N}T^{\lambda_{i}}z^{\partial\beta_{\nu_{i}}} \label{eqn:z1z2}
\end{align}
for some $\eta>0$, where $z_{1}$ and $z_{2}$ are the counts of $\beta_{\nu_{1}}$ and $\beta_{\nu_{2}}$ basic disks.

The new critical point can be obtained by solving the critical point equation for the first three terms in \eqref{eqn:z1z2} and applying Lemma \ref{Woodward}, once there is a guarantee that the last summand $\sum_{i=3}^{N}T^{\lambda_{i}}$ in \eqref{eqn:z1z2} has higher energy than these three. 
For this reason we restrict the size of the blowup so that every toric critical point $\alpha_{j}$ with $\val (\alpha_{j})=(x_{j},y_{j})$ lies in either of the shaded regions in Figure \ref{fig:blowupsize}. Specifically, we require that the size of the toric blowup $\eta$ satisfies 
\begin{equation}\label{condi:blowup-size}
0 < \eta < r:=\min_j\max\{x_{j}, y_{j}\}.
\end{equation}
This ensures that $\beta_{\nu'}$ does not become an energy minimizing disk of any  existing critical points of $W_\Sigma$. Indeed, for any toric critical point $\alpha_{j}$ with an energy minimizer $\beta_{\nu_k}$ ($1\leq k\leq N)$,
\begin{equation}\label{eqn:xjyj}
	\big\langle (x_{j},y_{j}), \nu_k \big\rangle + \lambda_k \,
	\leq \, \min\{ x_{j}, y_{j}\} < x_{j} + y_{j} - \eta, \,\,\, \text{ for all }\,\, j
\end{equation}
where the first inequality says that the area of $\beta_{\nu_k}$ does not exceed those of $\beta_{\nu_{1}}$ and $\beta_{\nu_{2}}$. \\

Now suppose the blowup size $\eta$ is less than $r$. \\

\noindent\emph{(i)} We first show that every critical points of $W_\Sigma$ extend to critical points of $W_{\widetilde{\Sigma}}$ in valuation-preserving manner.
\begin{itemize}
\item[(a)] If $W_\Sigma^{HV}$ satisfies local convenience (Definition \ref{def:locconv}), then each critical point of $W_\Sigma^{HV}$ has its corresponding critical point of $W_{\widetilde{\Sigma}}^{HV}$ which shares the same lowest energy term. This is because every vertices of $\mathrm{Trop}(W_\Sigma^{HV})$ are also vertices of $\mathrm{Trop}(W_{\widetilde{\Sigma}}^{HV})$ from the above discussion. On the other hand, higher energy terms in the full potential do not affect the valuations of geometric critical points by Proposition \ref{prop:pindownclasses}, as before.
\item[(b)]
If $W_\Sigma^{HV}$ is not locally convenient, we will prove separately in \ref{prop:gnc} that $W_\Sigma^{HV}$ fails to be locally convenient over a unique edge $e$ \eqref{eqn:locconvfail} of its tropicalization, and that the edge supports four critical points of $W_\Sigma$ each of which extends to that of $W_{\widetilde{\Sigma}}$ in a valuation-preserving manner. For the other critical points, we still have local convenience, and hence can proceed as in (a).
\end{itemize}
We conclude that evert critical points $\alpha_{j}$ of $W_{\Sigma}$ are extended to those of $W_{\widetilde{\Sigma}}$ keeping the same lowest energy terms.\\


\noindent\emph{(ii)} We next seek for a new critical point. Observe that $\beta_{\nu_{1}}$, $\beta_{\nu_{2}}$, and $\beta_{\nu'}=\beta_{\nu_{1}}+\beta_{\nu_{2}}$ are the only energy minimizing disk classes at the critical point $\alpha_0=(-T^\eta,-T^\eta)$ of $z_{1}+z_{2} + T^{-\eta} z_{1} z_{2}$ (with their energies being $\eta$).
This is because if the equality 
\[
	\eta=\underbrace{\big\langle \val(\alpha_0), \, \beta_{\nu_k} \big\rangle  +  \lambda_{k}}_{\omega_{\widetilde{\Sigma}}(\beta_{\nu_k}) }
\]
holds for some $k\neq1,2$, then $\beta_{\nu'}$ would be a minimizer at a critical point of $W_\Sigma^{HV}$ determined by $\beta_{\nu_{1}}$, $\beta_{\nu_{2}}$ and $\beta_{\nu_k}$, contradicting \eqref{eqn:xjyj}.

Note that $\nu_{1}$, $\nu_{2}$, and $\nu_{1}+\nu_{2}$ forms a convenient unimodal 2-cell, implying that $\alpha_0$ is indeed non-degenerate. (In particular, the Hori-Vafa potential $W^{HV}_{\widetilde{\Sigma}}$ is locally convenient at the new vertex $\val(\alpha)$.)  By Lemma \ref{Woodward}, $\alpha$ uniquely extends to a critical point $\alpha$ of $W_{\widetilde{\Sigma}}$.

\begin{figure}[h]
\centering

\tikzset{cross/.style={cross out, draw=black, minimum size=2*(#1-\pgflinewidth), inner sep=0pt, outer sep=0pt},
cross/.default={3.5pt}}
\tikzset{crossb/.style={cross out, draw=blue, minimum size=2*(#1-\pgflinewidth), inner sep=0pt, outer sep=0pt},
crossb/.default={3.5pt}}

\begin{tikzpicture}[scale=0.4]

\draw[fill=blue!30, opacity=0.4, draw=none] (3,0)--(3,3)--(10,10)--(10,0);
\draw[fill=red!30, opacity=0.4, draw=none] (0,3)--(3,3)--(10,10)--(0,10);

\draw[line width =1pt] (0,10)--(0,3);
\draw[line width =1pt] (3,0)--(0,3);
\draw[line width =1pt] (3,0)--(10,0);

\draw[line width =0.5pt, color=blue] (-2,3)--(3,3);
\draw[line width =0.5pt, color=blue] (3,-2)--(3,3);
\draw[line width =0.5pt, color=blue] (3,3)--(10,10);

\draw (11,10) node[above] {$x=y$};
\draw (-2,3) node[left] {$y=\eta$};
\draw (3,-2) node[below] {$x=\eta$};

\draw (8,2) node[above] {$x < x + y - \eta $};
\draw (3.5,7) node[above] {$y < x + y - \eta $};

\end{tikzpicture}
\caption{Restriction of sizes of toric blowups} \label{fig:blowupsize}
\end{figure}
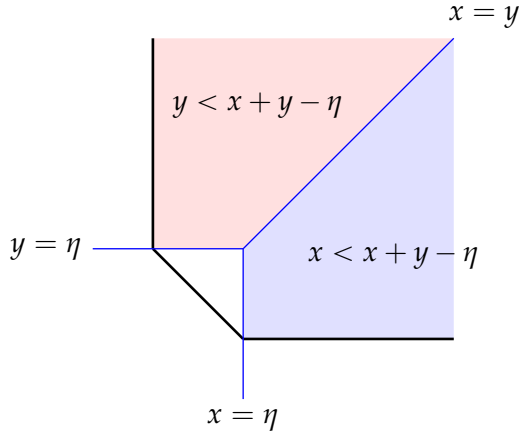


It remains to prove that there exist no other geometric critical points of $W^{HV}_{\widetilde{\Sigma}}$ having $\beta_{\nu'}$ as an energy minimizing disk. Suppose that such a critical point exists, say $\alpha'$ with $\val (\alpha')=(x_{0},y_{0})$, assuming $x_{0}\leq y_{0}$ without loss of generality. Let $ \beta_{\nu_k}$  be another energy minimizing disk class at $\alpha'$. We have $\partial \beta_{\nu_k} = \nu_k = (p,q) \notin \mathrm{Cone}(\nu_{1}, \nu_{2})$.
From \eqref{condi:trop}, we have
\begin{equation}\label{eqn:px0qy0}
	\underbrace{x_{0} + y_{0} - \eta}_{\omega_{\widetilde{\Sigma}}(\beta_{\nu'})}
	=\underbrace{px_{0} + qy_{0} + \lambda_k}_{\omega_{\widetilde{\Sigma}}(\beta_{\nu_k})}
	\leq \underbrace{x_{0}}_{\omega_{\widetilde{\Sigma}} (\beta_{\nu_{1}})}, \,  \underbrace{y_{0}}_{\omega_{\widetilde{\Sigma}} (\beta_{\nu_{2}})}.
\end{equation}
The first part of the inequality implies $x_{0}, \, y_{0} \leq \eta$, while the second part gives us 
\[
	0 
	\geq (p-1)x_{0}+qy_{0}+\lambda_k
	\geq  (p+q-1)x_{0} + \lambda_k.
\]
Note that $\lambda_k > 0$ since $\nu_k \notin \mathrm{Cone}(\nu_{1}, \nu_{2})$. Hence we have $p+q-1 < 0$. 

On the other hand, since $\beta_{\nu_k}$ is not energy minimizing at the critical point $\alpha$ by our hypothesis,
\begin{align*}
\langle \val(\alpha), \nu_{1}+\nu_{2} \rangle -\eta   
	< \langle \val(\alpha), (p,q) \rangle + \lambda_k.
\end{align*}
Plugging in $\val(\alpha)=(\eta,\eta)$, we get
\[
	(1-p-q)\eta <\lambda_k.
\]
Thus, 
\begin{align*}
\omega_{\widetilde{\Sigma}}(\beta_{\nu_k}) = px_{0} + qy_{0} + \lambda_k &> px_{0} + qy_{0} +(1-p-q)\eta \\
 &= p(x_{0}-\eta) + q(y_{0}-\eta) + \eta  \\
 &\geq (p+q)(x_{0}-\eta) + \eta  &\because x_{0} \leq y_{0}\\
 &=(p+q-1)(x_{0}-\eta)+x_{0}  &\because p+q-1 < 0, \, x_{0} \leq \eta \\
 &\geq x_{0},
\end{align*}
contradicting \eqref{eqn:px0qy0}. Hence no such $\alpha'$ can exist, and we conclude that $\alpha$ is the only new geometric critical point. \end{proof}

By inductively blowing up corners while imposing similar restrictions as in condition \eqref{condi:blowup-size}, we have:

\begin{cor}[\emph{Theorem II}]\label{cor:toricnum}
Let  $X_{\widetilde{\Sigma}}$ be the toric surface obtained by taking a sequence of toric blowup of another toric surface $X_\Sigma$. There exists $r>0$ depending on $\omega$ such that if the sizes of exceptional divisors are smaller than $r$, then $W_{\widetilde{\Sigma}}$ has as many new non-degenerate critical points as the number of exceptional divisors for generic parameters. Moreover, every critical point of $W_{\Sigma}$ are extended to that of $W_{\widetilde{\Sigma}}$ in a valuation-preserving manner. If $W_{\Sigma}$ is Morse, then $W_{\widetilde{\Sigma}}$ is also Morse. 
\end{cor}


Finally, we demonstrate how critical points of $W_{\Sigma}$ extend to those of its blowup when $W_{\Sigma}$ is not local convenient. (Note that we do not distinguish toric/non-toric blowups here.) If $W_{\Sigma}$ fails to be locally convenient, then it does so over (the closure of) the edge $e$ dual to a unique edge in the Newton subdivision that passes through the origin. Under basis change, we may assume that $e$ in the Newton subdivision joins $(0,1)$ and $(0,-1)$. Lemma \ref{lem:unimodal1} tells us that generically, $e$ is contained in two triangles in the Newton subdivision. Let us write $(i,j)$ and $(k,l)$ for the vertices of the triangles other than $(0,\pm1)$. See Figure \ref{fig:gnc}.

\begin{figure}[h]
\centering
\subcaptionbox{%
	Generic configuration of when the Newton subdivision is not locally convenient.
	\label{fig:gnc1}%
	}
	{\makebox[0.45\linewidth][c]
		{\begin{tikzpicture}[scale=0.5]
				
		\draw[line width=0.7pt, -latex] (0,-4.5)--(0,4.5);
		\draw[line width=0.7pt, -latex] (-4.5,0)--(4.5,0);
		
		\draw[line width=1pt] (0,3)--(-2,-2)--(0,-3)--(2,1)--cycle;
		
		\draw[line width=1pt, -latex] (0,0)--(0,3);				
		\draw[line width=1pt, -latex, color=red] (0,0)--(-2,-2);
		\draw[line width=1pt, -latex] (0,0)--(0,-3);
		\draw[line width=1pt, -latex, color=red] (0,0)--(2,1);

		\draw (0,3) node[left] {\small{$(0,1)$}};
		\draw (0,-3.2) node[right] {\small{$(0,-1)$}};
		\draw (-2,-2) node[left] {\small{$(k,l)$}};
		\draw (2,1) node[right] {\small{$(i,j)$}};						
										
		\end{tikzpicture}%
		}%
	}\hfill%
\subcaptionbox{%
	The moment polytope when the Newton subdivision is not locally convenient.
	\label{fig:gnc2}%
	}
	{\makebox[0.45\linewidth][c]	
		{\begin{tikzpicture}[scale=0.55]
		
		\draw[line width=0.9pt, color=red] (-3,3)--(-2,0.7);			
		\draw[line width=1pt, dotted] (-1.6,0.4)--(-0.3,0.1);			
		\draw[line width=0.9pt] (0,0)--(6,0);
		\draw[line width=1pt, dotted] (7.2,0.3)--(7.8,1.6);
		\draw[line width=0.9pt, color=red] (8,2)--(7,4);
		\draw[line width=1pt, dotted] (6.7,4.2)--(5.4,4.9);		
		\draw[line width=0.9pt] (5,5)--(-1,5);					
		\draw[line width=1pt, dotted] (-1.9,4.6)--(-2.8,3.4);

		\draw[line width=1.3pt] (0,2.3)--(5,2.3);
		\draw (2.2,2.3) node[above] {\small{$e$}};		
		
		\filldraw (3,0) circle (2pt);
		\draw (3,0) node[below] {\small{$(0,0)$}};			
		
		\end{tikzpicture}
		}%
	}%
\caption{When the Newton subdivision is not locally convenient.}
\label{fig:gnc}
\end{figure}
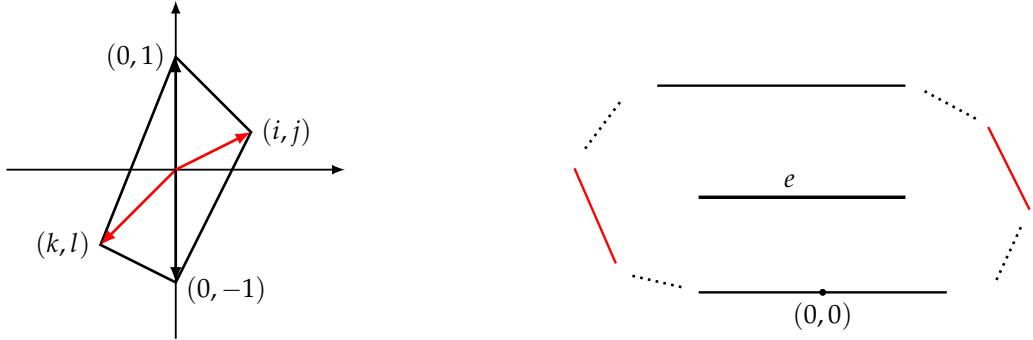

In accordance with energies of disks measured at a point in $e$, $W$ is given as
\[
	 W = z_{2} + \frac{T^b}{z_{2}} + T^{\lambda} z_{1}^i z_{2}^j + T^{\eta} z_{1}^k z_{2}^l + \text{h.o.t.},
\]
where the higher order terms contain contributions from disks with sphere bubbles, or from broken disks intersecting exceptional divisors. Without loss of generality, we may assume $i>0$, $k<0$ and $\lambda,\eta\geq0$. Let $\alpha$ be a geometric critical point of $W$ lying over $e$ with $\val (\alpha)$ on $e$. Clearly, $e$ is a part of $\{\val(z_{2})=b/2\}$, so we may set $z_{2} = T^{b/2} \underline{z_{2}}$ with $\underline{z_{2}}\in \Lambda_U$ to obtain
\begin{equation}\label{eqn:gncunderliney}
 W = T^{b/2} \left(\underline{z_{2}} + \frac{1}{\underline{z_{2}}} \right) + T^{\lambda + \frac{jb}{2}} z_{1}^i \underline{z_{2}}^j + T^{\eta + \frac{lb}{2}} z_{1}^k \underline{z_{2}}^l + \text{h.o.t.} \,.
\end{equation}
By (a) of Lemma \ref{Woodward}, we conclude that we must have $\underline{z_{2}} = \pm 1 + z_{2}^{+}$ for some nonzero $z_{2}^{+} \in \Lambda_+$ (since $\underline{z_{2}}=\pm1$ is obviously not a critical point).

Note that direct calculation shows that the $x$-coordinates of the end points of the edge $e$ are
\begin{equation}\label{eqn:gncrangee}
\frac{1}{i} \left( \frac{b}{2} - \lambda - \frac{jb}{2} \right) < \frac{1}{k} \left( \frac{b}{2} - \eta - \frac{lb}{2} \right)
\end{equation}
from left to right, and within this rage, the first two terms in \eqref{eqn:gncunderliney} are of minimal energy. 
Notice that the above inequality is a necessary condition for the aforementioned two triangles to be a part of the Newton subdivision.
Thus we have an expansion of $W$ in $z_{2}^{+}$
\begin{equation}\label{eqn:gncunderliney2}
W = T^{b/2} \left(	{(z_{2}^{+})}^{2} -   {(z_{2}^{+})}^{3} + \cdots \right) 
+ T^{\lambda + \frac{jb}{2}} z_{1}^i  (1+z_{2}^{+})^j + T^{\eta + \frac{lb}{2}} z_{1}^k (1+z_{2}^{+})^l + \text{h.o.t.}
\end{equation}
modulo constant.
Note that the higher order terms of \eqref{eqn:gncunderliney} are still of higher order in \eqref{eqn:gncunderliney2}, since the leading order term of $T^{\lambda^{'}}z_{1}^{p}(1+z_{2}^{+})^{q}$ is still $T^{\lambda^{'}}z_{1}^{p}$ due to the fact that $z_{2}^{+} \in \Lambda_+$.

In the expansion \eqref{eqn:gncunderliney2}, the only possible minimal energy (nonconstant) terms over $e$ are  $T^{b/2} {(z_{2}^{+})}^{2}$, $T^{\lambda + \frac{jb}{2}} z_{1}^i$, and $T^{\eta + \frac{lb}{2}} z_{1}^k$. Now if \eqref{eqn:gncunderliney2} admits a critical point, then again it must solve the critical equation of the lowest order part, and in this case, since $z_{2}^{+} \in \Lambda_+$, we should have the last two terms 
\begin{equation}\label{gncsolvex}
T^{\lambda + \frac{jb}{2}} z_{1}^i + T^{\eta + \frac{lb}{2}} z_{2}^k
\end{equation}
sitting in the lowest degree in order for the potential to admit critical points over $e$. We will see shortly that this is indeed the case.

Critical points of \eqref{gncsolvex} occur when $\val(z_{1}) = \frac{2(\eta-\lambda) + (l-j)b}{2(i-k)}$, so we may set
\[
	z_{1} =  T^{\frac{2(\eta-\lambda) + (l-j)b}{2(i-k)}} \underline{z_{1}}.
\]
Expanding $\underline{z_{1}}$ at the critical points of \eqref{gncsolvex} leads to
\[
	\underline{z_{1}} = \rho + z_{1}^{+}, \quad z_{1}^{+} \in \Lambda_+
\]
where $\rho$ is one of $(i-k)$-th roots of $-\frac{k}{i}$. Note that there are $(i-k)$-many choices for $\rho$. Hence we can conclude
\[
	\val(\alpha)= \left(     \frac{2(\eta-\lambda) + (l-j)b}{2(i-k)} , \frac{b}{2}  \right)
\]
if $\alpha$ does exist.\\

Substituting $\underline{z_{1}}= T^{\frac{2(\eta-\lambda) + (l-j)b}{2(i-k)}} (\rho + z_{1}^{+})$, the (leading terms of) second order expansion (over $e$) reads 
\[
	W = T^{b/2} \left( {(z_{2}^{+})}^{2} - {(z_{2}^{+})}^{3} + \cdots \right) + 
	T^\delta (\rho + z_{1}^{+})^i (1+ z_{2}^{+})^j + T^\delta (\rho + z_{1}^{+})^k (1+z_{2}^{+})^l
\]
modulo constants, where 
\[
	\delta= \lambda + \frac{jb}{2} + \frac{i(2 (\eta-\lambda) + (l-j) b)}{2(i-k)}=\eta + \frac{lb}{2} + \frac{k(2 (\eta-\lambda) + (l-j) b)}{2(i-k)}.
\]
(In a more symmetric form, this equals $ \frac{1}{2(i-k)} ( -2k \lambda + 2i \eta + (il -kj)b)$.) Notice that $\rho$ obtained from the critical point equation should make two linear terms in $z_{1}^{+}$ cancel. Hence the possible minimal energy terms are
\[
	T^{b/2} {(z_{2}^{+})}^{2}, \,\, T^\delta {(z_{1}^{+})}^{2}, \,\, T^\delta z_{2}^{+},
\]
followed by $T^\delta z_{1}^{+}z_{2}^{+}$. Taking partial derivatives, we have 
\begin{align*}
\partial_1 W 	&= iT^{\delta}(\rho + z_{1}^{+})^{i-1}(1+z_{2}^{+})^{j} + kT^{\delta}(\rho + z_{1}^{+})^{k-1}(1+z_{2}^{+})^{l}\\
				&= C_{1} + \left(a_{1}T^{\delta} +\text{h.o.t.} \right) z_{1}^{+} 
				+ \left(a_{2}T^{\delta} + \text{h.o.t.} \right)z_{2}^{+} + \sum \lambda_{pq} {(z_{1}^{+})}^{p} {(z_{2}^{+})}^{q},\\
\partial_2 W	&= \left(2T^{\frac{b}{2}}+\text{h.o.t.} \right) z_{2}^{+} + jT^{\delta}(\rho + z_{1}^{+})^{i}(1+z_{2}^{+})^{j-1} 
				+ lT^{\delta}(\rho + z_{1}^{+})^{k}(1+z_{2}^{+})^{l-1}\\
				&= C_{2} + \left(b_{1}T^{\delta} + \text{h.o.t.} \right) z_{1}^{+} + \left(b_{2}T^{\frac{b}{2}} + \text{h.o.t.} \right) z_{2}^{+}
				+ \sum \eta_{pq} {(z_{1}^{+})}^{p} {(z_{2}^{+})}^{q}
\end{align*}
with $\val (C_{1}) > \delta$ and $\val (C_{2}) \geq \delta$. Note that the two summations above are higher order terms, in the sense that $\val (\lambda_{pq}) \geq \delta$, $\val (\eta_{pq}) \geq \frac{b}{2}$, and  $\val (\eta_{pq}) \geq \delta$ for $p \geq 1$.

Crucial observation here is that $b/2 < \delta$. This is in fact equivalent to \eqref{eqn:gncrangee} that was originated from our geometric assumption (on the Newton subdivision). Therefore by applying Lemma \ref{lemma:appendix}, we see that $(\partial_1 W - C_{1}, \partial_2 W - C_{2}) = (0,0)$ obtained by substituting
\begin{equation}\label{nonconvcrit}
	(z_{1}, z_{2}) = \left( T^{\frac{2(\eta-\lambda) + (l-j)b}{2(i-k)}} (\rho + z_{1}^{+}), T^{b/2} (\pm 1 + z_{2}^{+}) \right)
\end{equation}
has a unique solution, as the left hand side is a small perturbation of the identity map. Moreover, the leading term calculation tells us that $\val(z_{2}^{+}) = \delta - \frac{b}{2}$. 
In this case, $T^{b/2} {(z_{2}^{+})}^{2}$ can be checked to have a higher valuation than the other two among
\[
	T^{b/2} {(z_{2}^{+})}^{2},\,\, T^{\lambda + \frac{jb}{2}} z_{1}^i,\,\, T^{\eta + \frac{lb}{2}} z_{1}^k
\]
where $\val(z_{1}) = \frac{2(\eta-\lambda) + (l-j)b}{2(i-k)}$, which justifies our earlier assumption. Consequently, we obtain $2(i-k)$ distinct critical points all supported at $\big( \frac{2(\eta-\lambda) + (l-j)b}{2(i-k)} , \frac{b}{2} \big)$, which in fact coincides with the normalized volume of the Newton polytope. We have thus proven:

\begin{prop}\label{prop:gnc}
Let $(X,D)$ be obtained from $(X_{\Sigma},D_{\Sigma})$ via a sequence of toric/non-toric blowups, and suppose that its Newton subdivision contains a non-convenient edge whose adjoint $2$-cells, say $\sigma_{1}$, $\sigma_{2}$, consists of Hori-Vafa terms, as in Figure \ref{fig:gnc}. Then $W_{\Sigma}$ has $2(\mathrm{V}(\sigma_{1})+ \mathrm{V}(\sigma_{2}))$-many non-degenerate geometric critical points on the dual edge $e \in \mathrm{Trop}( W_{\Sigma} )$, all of which extends to $W$.
\end{prop}

\begin{example}

Notice that the Hirzebruch surface $\mathbb{F}_{k}$ is the case when $(i,j) = (1,0)$ and $(k,l) = (-1,-k)$, and it follows that $\mathbb{F}_{k}$ has $2(1-(-1))=4$ non-degenerate geometric critical points. To give a more concrete description of the location of these critical points, let us consider the following scenario where the moment polytope of $\mathbb{F}_{k}$ is given by 
\[
	 x \geq 0, \quad y \geq 0, \quad b-y \geq 0, \quad a-x-ky \geq 0.
\]
The edge $e$ where local convenience fails is given by 
\begin{equation}\label{eqn:locconvfail}
e:=\bigl\{ (x,y) \in \mathbb{R}^2 : b/2 \leq x \leq a - (b(k+1)/2), \quad y= b/2\bigr\}.
\end{equation}
 (When $k=0$, we assume $a >b$ without loss of generality.) From Proposition \ref{prop:gnc}, it follows that $\mathbb{F}_{k}$ has 4 non-degenerate geometric critical points for $k \geq 1$, all of which are of valuation $(\frac{1}{2}a-\frac{k}{4}b, \frac{1}{2}b) \in e$.

\begin{figure}[h]
\centering
\begin{tikzpicture}[scale=0.65]

		\draw[line width =1pt] (0,0)--(0,5)--(5,5)--(20,0)--cycle;
		\draw[line width =0.7pt, dashed] (0,5)--(0,20/3)--(5,5);
		\draw[line width =0.7pt, dashed] (0,2.5)--(2.5,2.5)--(2.5,0);
		\draw[line width =0.7pt, dashed] (7.5,2.5)--(7.5,0);
		
		\draw (20,0) node[below] {$a$};
		\draw (0,4.8) node[left] {$b$};
		\draw (0,2.3) node[left] {$b/2$};
		\draw (2.3,0) node[below] {$b/2$};
		\draw (7.6,0) node[below] {$a - (b(k+1)/2)$};
		\draw (0,7.2) node[left] {$a/k$};

		\draw[line width =0.7pt, color=blue] (0,5)--(2.5,2.5)--(0,0);
		\draw[line width =0.7pt, color=blue] (5,5)--(7.5,2.5);
		\draw[line width =0.7pt, color=red] (7.5,2.5)--(2.5,2.5);
		\draw[line width =0.7pt, color=blue] (7.5,2.5)--(20,0);
		
		\draw (3.3,2.5) node[above] {$e$};
		\filldraw (5,2.5) circle (2.2pt);
		\draw (5.2,2.5) node[below] {$p$};

\end{tikzpicture}
\caption{The moment polytope $\Delta_{\mathbb{F}_k}$ and the tropicalization $\mathrm{Trop}(W_{\mathbb{F}_k}^{HV})$. The edge $e$, dual to the edge of $\mathcal{S}_{W_{\mathbb{F}_k}^{HV}}$ where local convenience fails, is represented in red. The point $p$ is where all geometric critical points lie.} \label{fig:momentpolytopeofFk}
\end{figure}
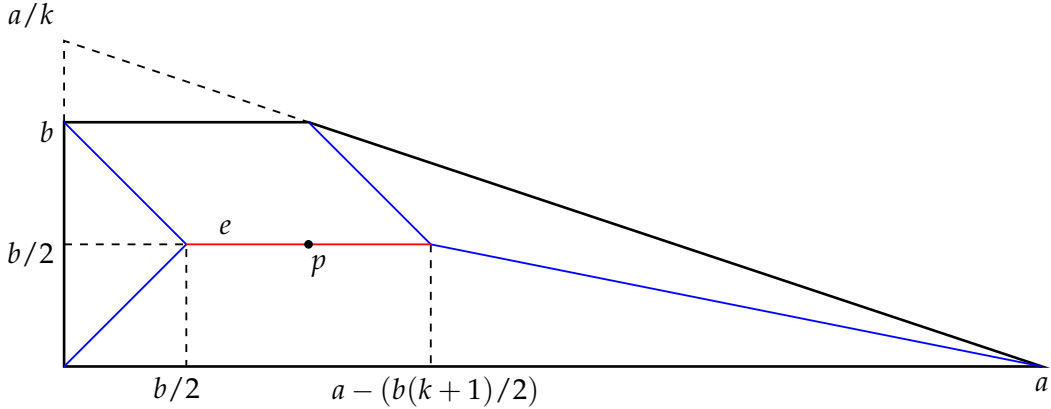
\end{example}



\subsection{Non-toric blowups $(X,D)$ of a toric surface $(X_{\Sigma}, D_\Sigma)$}\label{subsec:non-toriclocal}
We now proceed to examine the changes in the number of critical points during the process of non-toric blowups on a given toric surface $(X_{\Sigma}, D_\Sigma)$. We will explicitly locate every non-toric critical points of $W$, and show that they are Morse critical points. The main result is the following.

%

\begin{thm}[\emph{Theorem III}]\label{thm:non-toricblowup}
Let $(X_{\Sigma}, D_\Sigma)$ be a toric surface, 
 and let $(X, D)$ be the surface obtained after a sequence of non-toric blowups equipped with a symplectic form described in \ref{subsec:toricmodel} with generic parameters $\epsilon_i$. Then
\begin{enumerate}
\item[(a)] each (non-toric) blowup at a point in the interior of a toric divisor is responsible for a unique new geometric critical point of $W$ which is non-degenerate. If the blowup point lies in $D_{\Sigma,j+1}$, the corresponding new critical point is located near the corner (the nodal point) $D_{\Sigma,j} \cap D_{\Sigma,j+1}$.
\item[(b)] if  $W_\Sigma$ is Morse, then $W$ is also Morse. Moreover, every critical point of $W_\Sigma$ extends to that of $W$ in a valuation-preserving manner.
\end{enumerate}
\end{thm}


\begin{proof}
To each blowup point in the toric divisor $D_{\Sigma,j+1}$, we will associate a non-toric critical point one of whose energy minimizing disks projects to the class $\beta_{\nu_j} + k \beta_{\nu_{j+1}}$ under $\pi_\ast$ for some $k$. If such a critical point exists, then it is located near the corner $D_{\Sigma,j} \cap D_{\Sigma,j+1}$ by Proposition \ref{prop:non-toric-crits}, and all the other energy minimizers should also lie in $\mathcal{A}_j \cup \{\beta_{\nu_j},\beta_{\nu_{j+1}} \} $ after $\pi_\ast$ by Lemma \ref{lemma:interiorcritpts}. 

For this purpose, it is natural to begin with a local expansion of $W$ near the corner $D_{\Sigma,j} \cap D_{\Sigma,j+1}$. Obviously, if the point is close enough to this corner, then the images of low energy contributions to $W$ under $\pi_\ast$ belong to classes in $\mathcal{A}_j \cup \{\beta_{\nu_j},\beta_{\nu_{j+1}} \} $. Indeed, if a Lagrangian boundary is close enough to $D_{\Sigma,j} \cap D_{\Sigma,j+1}$, then the only Maslov $2$ disks in its small neighborhood are basic broken disks.

Thus $W$ can be decomposed as $W = \overline{W}_j + W_+$ with
\begin{equation}\label{eqn:wlocalzero}
\begin{array}{lcl}
\displaystyle	\overline{W}_{j} &=& T^{\lambda_{j}}z^{\partial\beta_{\nu_{j}}} + T^{\lambda_{j+1}}z^{\partial\beta_{\nu_{j+1}}} + \displaystyle\sum_{\beta \in \mathcal{A}_{j}} N_\beta T^{\delta(\beta)}z^{\partial \beta} \\
	&=& T^{\lambda_{j}}z^{\partial\beta_{\nu_{j}}} + T^{\lambda_{j+1}}z^{\partial\beta_{\nu_{j+1}}} + {\displaystyle\sum_{k=1}^{l}} { l \choose k}T^{\lambda_j + k \lambda_{j+1}}z^{ \nu_j + k \nu_{j+1}}
\end{array}
\end{equation}
when we set the sizes of exceptional divisors to be zero. Here, $l$ is the number of blowup points in the interior of $D_{\Sigma,j+1}$, so $\mathcal{A}_j = \bigl\{ \, \beta_{\nu_j} + k \beta_{\nu_{j+1}} : 1 \leq k \leq l \, \bigr\}$.
From the above discussion, the valuation of each monomial in $W_+$ is bigger than any of those in $\overline{W}_j$. 

In what follows, we will work with $j=1$ for notational simplicity, where we further assume $\nu_{1}=(1,0)$, $\nu_{2}=(0,1)$, and $\lambda_{1}=\lambda_{2}=0$ by some affine coordinate change as before. However we set the sizes of the exceptional divisors to be generic from now on. Thus, using the energy formula \eqref{eqn:areabrd}, we have
\begin{align}\label{stdlocalmodel}
\begin{split}
\overline{W}_{1} = z_{1}+z_{2}		& +(T^{-\epsilon_{1}}+T^{-\epsilon_{2}}+ \cdots + T^{-\epsilon_{l}})z_{1}z_{2} \\
 								& +(T^{-\epsilon_{1}-\epsilon_{2}}+T^{-\epsilon_{1}-\epsilon_{3}}+ \cdots + T^{-\epsilon_{l-1}-\epsilon_{l}})z_{1}z_{2}^{2} \\
								& + \cdots + T^{-\epsilon_{1}-\epsilon_{2}- \cdots -\epsilon_{l}}z_{1}z_{2}^{l}
\end{split}
\end{align}
which goes back to \eqref{eqn:wlocalzero} when $\epsilon_j=0$ for all $j$. 

Observe that
\begin{align*}
\frac{\partial}{\partial z_{1}}\overline{W}_{1} &= 1 	+ (T^{\epsilon_{1}} + \cdots + T^{\epsilon_{l}})z_{2} + (T^{-\epsilon_{1}}T^{-\epsilon_{2}}+ \cdots + T^{-\epsilon_{l-1}}T^{-\epsilon_{-l}})z_{2}^{2} +\cdots + (T^{-\epsilon_{1}}\cdots T^{\epsilon_{l}})z_{2}^{l}&\\
											&=(1+T^{-\epsilon_{1}}z_{2})(1+T^{-\epsilon_{2}}z_{2})\cdots(1+T^{-\epsilon_{l}}z_{2}).
\end{align*}
Solving $\frac{\partial}{\partial z_{1}}\overline{W}_{1}=0$, we obtain
\[
	z_{2}=-T^{\epsilon_{i}}   \quad \text{for } \,\, 1\leq i \leq l.
\]
Therefore, as long as $\epsilon_{j}$'s are generic, $\overline{W}_{1}$ has $l$-many distinct critical points. 
In this case, one can check that $z_{1} = (-1)^{l}T^{\epsilon_{i}+\sum_{j<i}\epsilon_{j}-\epsilon_{i}}$ modulo higher order terms. 
See \ref{subsec:critval} for the detailed calculation and some related discussion.
Note that these points are all geometric.
For example, when $l=3$, we have three solutions with valuations
\[
	\val (\alpha_{1}) = (\epsilon_{1},\epsilon_{1}), \quad \val (\alpha_{2}) = (\epsilon_{1},\epsilon_{2}), \quad \val (\alpha_{3}) = (\epsilon_{1}+\epsilon_{2}-\epsilon_{3},\epsilon_{3}),
\]
where $\epsilon_{1}>\epsilon_{2}>\epsilon_{3}>0$.  In fact, the valuation of the critical point $(z_{1},z_{2})$ lies inside $R_0$ (Lemma \ref{lemma:centralchamber}) if every blowup points in $D_{\Sigma,i}$ locate close enough to the corner $D_{\Sigma,i} \cap D_{\Sigma,i-1}$ for each $i$ (observe that the valuation does not depend on the location of the blowup center).

\sloppy On the other hand, applying Kushnirenko's theorem (Theorem \ref{Kushnirenko}), we find that $\rvert \mathrm{Crit}(\overline{W}_{1}) \lvert = 2 V(\Delta_{\overline{W}_{1}}) = l$. This implies that each of the above $l$ distinct points has multiplicity $1$, i.e., they are non-degenerate. Every such critical point $\alpha_{j}$ of $\overline{W}_{1}$ can be successfully extended to a geometric critical point of $W$ using Lemma \ref{Woodward} thanks to non-degeneracy of its associated minimal energy terms.

%
%

Hence for each boundary divisor of $X$ containing $l$ points in the blowup center, we have successfully identified $l$ distinct Morse (non-toric) critical points for $W$, all located on the left-hand corner of the respective divisors. Since we have worked with all possible minimal energy terms near each corner of the moment polytope, Lemma \ref{Woodward} guarantees that no new geometric (non-toric) critical points arise when patching up local models. This proves (a).
%

We next show that the non-toric critical points arising from newly added first order terms do not interfere with any of the previous toric critical points from the Hori-Vafa part, essential reason being that non-toric blowup are done in relatively small scale. To see this, observe that geometric critical points of $W$ other than the newly added ones after blowup are all toric critical points by Theorem \ref{prop:pindownclasses}, i.e., their energy minimizers are basic disks (we have exhausted all possible non-toric critical points of $W$ above). 

We need to show that these toric critical points are in one-to-one correspondence with the (geometric) critical points of $W_{\Sigma}$. The argument is completely parallel to the part \emph{(i)} in the proof of Proposition \ref{prop:toricblowup}. Namely, 
\begin{itemize}
\item[(a)]
when $W_{\Sigma}^{HV}$ is locally convenient, the valuation of a toric critical point of $W$ sits at a vertex of $\mathrm{Trop}(W_{\Sigma}^{HV})$ which establishes one-to-one correspondence between geometric critical points of $W_{\Sigma}$ and toric critical points of $W$ (since they are both determined by inductively solving the same leading order terms at each vertex of $\mathrm{Trop}(W_{\Sigma}^{HV})$). 
\item[(b)] If $W_{\Sigma}^{HV}$ is not locally convenient, it fails to be so on the unique edge $e$ \eqref{eqn:locconvfail} over which we pick up four nongenerate critical points of $W$ by Proposition \ref{prop:gnc}. The other critical points can be handled in the same way as in (a).
\end{itemize}

So far, we have worked with the fixed expression of $W$ that is valid on $R_0$ (given in Lemma \ref{lemma:centralchamber}). As the analytic continuation does not completely cover the complement of $R_0$ (it misses some codimension $1$ region, see Remark \ref{rmk:failac}), we should additionally prove that no critical point occur outsider $R_0$. This will be done in Lemma \ref{lem:nooutr0} below, which finishes the proof.
\end{proof}

%
%
%
%
%

\begin{lemma}\label{lem:nooutr0}
If all the blowup points on $D_{\Sigma,i}$ are close enough to the corner $D_{\Sigma, i} \cap D_{\Sigma, i-1}$ for every $i$, then the glued LG potential (regardless of the choice of a chamber) does not have geometric critical points outside $R_0$.
\end{lemma}

\begin{proof}
From the discussion in \ref{subsec:toricblowup1}, we see that the valuations of toric critical points of $W$ only depend on $W_\Sigma^{HV}$, and in particular there are only finitely many. Since $R_0$ exhausts the whole moment polytope as points in the blowup centers approach the nearby corners, it contains all toric critical point at some point. 

On the other hand, Theorem \ref{lemma:interiorcritpts} justifies that the local calculation in the proof of Theorem \ref{thm:non-toricblowup} (see \eqref{stdlocalmodel}) is still valid on any chamber, except that we need to remove terms containing $\epsilon_i$ as we pass to the left side of the corresponding wall (speaking in the setting of the proof of Theorem \ref{thm:non-toricblowup}). In particular, if we approach close enough to the boundary of the moment polytope, we are only left with $z_{1}+z_{2}$ in the local expression of $W$ which does not admit any critical point.
\end{proof}

The following corollary is straightforward.

\begin{cor}\label{cor:critnum}
In the situation of (b) Theorem \ref{thm:non-toricblowup}, the number of critical of $W$ agrees with the rank of the cohomology of $X$.
\end{cor}

\begin{proof}
Corollary \ref{cor:toricnum} tells us that $W_\Sigma$ has as many geometric critical points as the rank of $H^\ast(X_\Sigma)$. This is based on the induction on the number of (toric) blowups, where the base case of the induction is covered by Lemma \ref{lem:unimodal1} and Proposition \ref{prop:gnc}.  
Then Theorem \ref{thm:non-toricblowup} implies that the number of geometric critical points of $W$ equals the rank of cohomology of $X$ since they are both increased by the number of (non-toric) blowups from $X_\Sigma$ to $X$. 
\end{proof}

We conclude that performing non-toric blowups (with generic $\epsilon_{i}$) results in the attachment of unimodal 2-cells to the Newton polytope of $W_{\Sigma}$, which translates to a ``splitting'' of an unbounded branch of the tropicalization $\mathrm{Trop}( W_{\Sigma} )$.

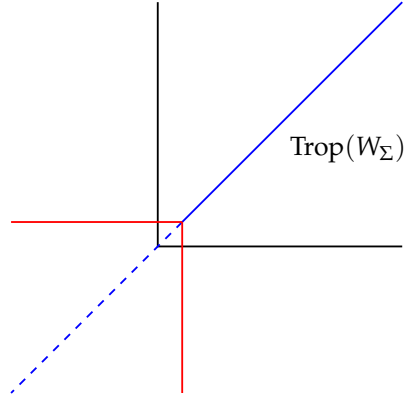
\begin{figure}[h]
\centering
\begin{tikzpicture}[scale=0.65]

	\draw[line width=0.7pt] (0,5)--(0,0)--(5,0);
	\draw[line width=0.7pt, blue] (0.5,0.5)--(5,5);
	\draw[line width=0.7pt, blue, dashed] (0.5,0.5)--(-3,-3);
	\draw[line width=0.7pt,red] (0.5,0.5)--(-3,0.5);
	\draw[line width=0.7pt,red] (0.5,0.5)--(0.5,-3);

	\draw (2.5,2) node[right] {\small{$\mathrm{Trop}( W_{\Sigma} )$}};

\end{tikzpicture}
\caption{Branching of $\mathrm{Trop}( W_{\Sigma} )$ locally.} \label{branching}
\end{figure}

 In particular, performing non-toric blowups leaves $\mathrm{Trop}( W_{\Sigma} )$ unchanged but one branch, specifically the branch representing the corner near the blowup center. The unique new geometric critical point is supported at the point where the branch splits, whereas the non-geometric critical points that has been removed precisely corresponds to the vertices of $\mathrm{Trop}( W_{\mathfrak{min}} )$ arising from the intersection between the resulting split branches and other branches.\\

\begin{example}\label{ex:non-toricblowups}

Let $(X_{\Sigma},D_{\Sigma}=\Sigma_{i=1}^{5}D_{\Sigma,i})$ be the semi-Fano surface obtained by blowing up $\mathbb{P}^{2}$ twice as illustrated with the dotted lines in Figure \ref{fig:non-toric-example1}. The Hori-Vafa part of the potential is given by 
\[
	W^{HV}_{\Sigma} = z_{1} + z_{2} + T^{a}\frac{1}{z_{1}z_{2}} + T^{b}\frac{1}{z_{2}} + T^{c}\frac{z_{1}}{z_{2}}.
\]
The gray region in Figure \ref{fig:non-toric-example2} represents the Newton subdivision $\mathcal{S}_{W^{HV}_{\Sigma}}$ of $X_{\Sigma}$, which is locally convenient. Now consider the log Calabi-Yau surface $(X,D)$ obtained by taking three non-toric blowups on $D_{\Sigma,2}$ and one non-toric blowup on $D_{\Sigma,4}$ (indicated by crosses). Then four unimodal 2-cells, namely $\sigma_{1}$, $\sigma_{2}$, $\sigma_{3} \in \mathrm{Cone}(\nu_{1},\nu_{2})$ and $\sigma_{4}\in \mathrm{Cone}(\nu_{3},\nu_{4})$ are attached to $\mathcal{S}_{W^{HV}_{\Sigma}}$. Note that the resulting surface $X$ is no longer Fano. The blue line in Figure \ref{fig:non-toric-example1} depicts the tropicalization $\mathrm{Trop}(W_{\mathfrak{min}})$ of the minimal energy terms, and the blue dotted lines indicate branches of $\mathrm{Trop}( W^{HV}_{\Sigma} )$ before splitting. The critical point information of $W$ can be read off of the Newton subdivision $\mathcal{S}_{ W_{\mathfrak{min}} }$ and the tropicalization $\mathrm{Trop}( W_{\mathfrak{min}} )$ as follows:

\noindent\emph{1.} The three vertices of $\mathrm{Trop}( W_{\mathfrak{min}} )$ lying away from the boundary divisors corresponds to the three 2-cells of $\mathcal{S}_{W^{HV}_{\Sigma}}$, each of which supports $1$, $1$, and $3$ toric critical points.

\noindent\emph{2.} The three vertices of $\mathrm{Trop}( W_{\mathfrak{min}} )$ near the origin, as well as the vertex near the corner $D_{\Sigma,3} \cap D_{\Sigma,4}$, are dual to the unimodal 2-cells $\sigma_{1}$, $\sigma_{2}$, $\sigma_{3}$, and $\sigma_{4}$, respectively. Each of these vertices supports a single non-toric geometric critical point.

\noindent\emph{3.} On the other hand, the red vertex above $D_{\Sigma,4}$ lying outside the moment polytope is dual to the 2-cell excluded by Proposition \ref{prop:non-toric-crits}, running across different cones. The excluded 2-cell is  represented with dashed line in Figure \ref{fig:non-toric-example2}.

\begin{figure}[h]
\centering
\subcaptionbox{%
	The moment polytope of $X_{\Sigma}$ and the tropicalization of $W_{\mathfrak{min}}$.
	\label{fig:non-toric-example1}%
	}
	{\makebox[0.45\linewidth][c]
		{\begin{tikzpicture}[scale=0.11]
				
		\coordinate (A) at (0,0);
		\coordinate (B) at (0,20);
		\coordinate (C) at (10,30);
		\coordinate (D) at (20,30);
		\coordinate (E) at (50,0);

		\draw[line width = 1pt] (A) -- (B) -- (C) -- (D) -- (E) -- cycle;
		\draw[dotted, line width = 1pt] (D) -- (0,50) -- (0,20);
		\draw[dotted, line width = 1pt] (0,30) -- (10,30);

		\draw [line width=0.7pt, color=blue!70] (-10,3)--(3,3)--(3,2)--(-10,2);
		\draw [line width=0.7pt, color=blue!70] (3,2)--(4,1)--(-10,1);
		\draw [line width=0.7pt, color=blue!70] (4,1)--(26,-9);
		\draw [line width=0.7pt, color=blue!70] (3,3)--(15,15)--(20,15)--(20,29)--(10,39)--(5,49);
		\draw [line width=0.7pt, color=blue!70] (20,29)--(50,29);
		\draw [line width=0.7pt, color=blue!70] (20,15)--(60,-5);
	
		\draw [line width=0.7pt, color=blue!70] (-10,20)--(10,20)--(10,39);
		\draw [line width=0.7pt, color=blue!70] (10,20)--(15,15);
	
		\draw [dashed, line width=0.7pt, color=blue!70] (3,3)--(-10,-10);
		\draw [dashed, line width=0.7pt, color=blue!70] (20,29)--(20,49);	
		\draw[line width =0.7pt] (1,0) node[cross] {};
		\draw[line width =0.7pt] (2.3,0) node[cross] {};
		\draw[line width =0.7pt] (3.7,0) node[cross] {};
  		\draw[line width =0.7pt] (19,30) node[cross] {};
		\filldraw[color=red] (10,39) circle (15pt);
	  	  
		\draw (0,30) node[left] {$b$};
		\draw (0,17) node[left] {$c$};
		\draw (0,48) node[left] {$a$};

		\draw (27,0) node[below] {\small{$D_{\Sigma,2}$}};
		\draw (35,17) node[above] {\small{$D_{\Sigma,3}$}};
		\draw (14,30) node[below] {\small{$D_{\Sigma,4}$}};
		
		\end{tikzpicture}%
	}
	}\hfill%
\subcaptionbox{%
	The Newton subdivision $\mathcal{S}_{ W_{\mathfrak{min}}}$ of the minimal terms.
	\label{fig:non-toric-example2}%
	}
	{\makebox[0.45\linewidth][c]	
		{\begin{tikzpicture}[scale=0.13]
		
		\draw[line width = 1pt] (10,30)--(10,-10)--(0,-10)--(-10,-20)--(-10,-10)--cycle;

		\draw[line width = 1pt] (0,10)--(10,20);
		\draw[line width = 1pt] (0,10)--(10,10);
		\draw[line width = 1pt] (0,10)--(10,0);
		\draw[line width = 1pt] (10,0)--(-10,-10);
		\draw[line width = 1pt] (-10,-10)--(0,-10);
		\draw[line width = 1pt] (0,-10)--(10,0);

		\fill[green!30, draw=none, opacity =0.4] (0,-10)--(-10,-20)--(-10,-10);

		\fill[green!30, draw=none, opacity =0.4] (10,10)--(10,0)--(0,10);
		\fill[red!30, draw=none, opacity =0.4] (10,10)--(10,20)--(0,10);
		\fill[blue!30, draw=none, opacity =0.4] (10,30)--(10,20)--(0,10);
		\fill[gray, draw=none, opacity =0.4] (0,10)--(10,0)--(10,-10)--(-10,-10);

		\draw[line width =0.5pt, -latex] (0,0)--(-10,-10);
		\draw[line width =0.2pt] (0,0)--(-20,-20);
		\draw[line width =0.5pt, -latex] (0,0)--(0,-10);
		\draw[line width =0.2pt] (0,0)--(0,-20);
		\draw[line width =0.5pt, -latex] (0,0)--(10,-10);
		\draw[line width =0.2pt] (0,0)--(15,-15);
		\draw[line width =0.5pt, -latex] (0,0)--(10,0);
		\draw[line width =0.2pt] (0,0)--(15,0);
		\draw[line width =0.5pt, -latex] (0,0)--(0,10);
		\draw[line width =0.2pt] (0,0)--(0,20);

		\draw (9.5,7) node[left] {\small{$\sigma_{1}$}};
		\draw (9.5,13) node[left] {\small{$\sigma_{2}$}};
		\draw (9.5,19.8) node[left] {\small{$\sigma_{3}$}};
		\draw (-9.5,-14) node[right] {\small{$\sigma_{4}$}};

		\draw (12,0) node[below] {\small{$\nu_{1}$}};
		\draw (0,10) node[left] {\small{$\nu_{2}$}};
		\draw (-12,-10) node[above] {\small{$\nu_{3}$}};
		\draw (0,-12) node[right] {\small{$\nu_{4}$}};
		\draw (10,-10) node[right] {\small{$\nu_{5}$}};

		\node[circle,fill=black,inner sep=0pt,minimum size=3pt] (a) at (10,30) {};
		\node[circle,fill=black,inner sep=0pt,minimum size=3pt] (a) at (10,20) {};
		\node[circle,fill=black,inner sep=0pt,minimum size=3pt] (a) at (10,10) {};
		\node[circle,fill=black,inner sep=0pt,minimum size=3pt] (a) at (10,0) {};
		\node[circle,fill=black,inner sep=0pt,minimum size=3pt] (a) at (10,-10) {};
		\node[circle,fill=black,inner sep=0pt,minimum size=3pt] (a) at (0,-10) {};
		\node[circle,fill=black,inner sep=0pt,minimum size=3pt] (a) at (-10,-20) {};
		\node[circle,fill=black,inner sep=0pt,minimum size=3pt] (a) at (-10,-10) {};
		\node[circle,fill=black,inner sep=0pt,minimum size=3pt] (a) at (0,10) {};

		\draw[dashed] (-10,-20) -- (10,-10);
		
		\end{tikzpicture}
		}
	}%
\caption{Example \ref{ex:non-toricblowups}.}
\label{fig:non-toric-example0}
\end{figure}
\end{example}


\subsection{Bulk-deformed potential (with toric bulk insertions only in non-Fano)}

We finally return to potential functions with bulk, in particular the bulk-deformation $W^{\mathfrak{b}}$ of $W$ associated with the bulk parameter $\mathfrak{b} \in H^{even} (X;\Lambda_+)$ which are pulled back from some torus invariant cycles in $X_\Sigma$. This restriction is to guarantee the weakly-unobstructedness (Lemma \ref{lem:torusweakunobs}), and we can consider more general bulk-insertions in the semi-Fano case where the weakly-unobstructedness is automatic by the obvious degree reason.

%
Following \cite{FOOO, FOOO-T2}, $W^\mathfrak{b}$ is defined as 
\begin{align*}
W^{\mathfrak{b}}_u (\underline{z}) = \sum_{k=0}^{\infty} \,  \sum_{\beta} \,  \frac{1}{k!}  \, \mathfrak{q}_{k}(\beta; \mathfrak{b}, \dots, \mathfrak{b}) \, T^{\omega (\beta)} \, \underline{z}^{\partial \beta}   
\end{align*}
where the sum is taken over $\beta \in \pi_{2}(X, L_{u})$. Here, $\mathfrak{q}_{k}(\beta; \mathfrak{b}, \dots, \mathfrak{b})$ is roughly the count of holomorphic disks with $k$ interior marked points incident to the ambient cycle $\mathfrak{b}$ and passing through a generic point of the fiber $L_{u}$. We refer readers to \cite{FOOO} or \cite{FOOO-T2} for more details. 
%

Note that the Maslov index of the class $\beta$ can be strictly bigger than $2$ when it admits the insertion from the point class ($\in H^4(X)$), and hence its contribution to $W^{\mathfrak{b}}$ has relatively high energy in general. This is also the case for degree-2 bulk insertion $\mathfrak{b}$, as $\mathfrak{b}$ itself carries a positive energy in its $\Lambda_+$-coefficient. In fact, one can prove that such extra terms of $W^\mathfrak{b}$ arising from nontrivial $\mathfrak{b}$ do not affect the leading term potential at any geometric critical points. To see this, let us denote by $W^{\mathfrak{b}}_{\mathfrak{min}}$ the collection of terms in $W^{\mathfrak{b}}$ induced from the disks that can possibly have minimal energy at some geometric critical point of $W^{\mathfrak{b}}$.

\begin{lemma}
For generic parameters, $W^{\mathfrak{b}}_{\mathfrak{min}} = W_{\mathfrak{min}}$.
\end{lemma}

\begin{proof}
Let $\beta$ be any holomorphic disk with $ \mathfrak{q}_{k}(\beta; \mathfrak{b}, \dots, \mathfrak{b}) \neq 0$, and suppose that $\partial \beta$ is contained in $\mathrm{Cone}(\nu_{1},\nu_{2})$, where we assume $\nu_{1}=(1,0)$, $\nu_{2}=(0,1)$ and $\,\lambda_{1}=\lambda_{2}=0$ without loss of generality. Our argument in the proof of Theorem \ref{prop:pindownclasses} remains valid, showing that disks in the class $\beta$ can be energy minimizing at a geometric critical point only if $\delta(\beta) = 0$ (see \eqref{eqn:deltabetai}). As before, this forces $\pi_\ast(\beta)$ to be a linear combination of $\beta_{\nu_{1}}$ and $\beta_{\nu_{2}}$. Previously, we concluded at this point that $\beta$ are either s basic or a basic broken disk due to Maslov 2 constraint. Since we now allow bulk-deformations via the point cycle $D_{\Sigma,1} \cap D_{\Sigma,2}$, there is the third possibility of $\beta$ being the proper transform of a higher Maslov disk in $X_\Sigma$, hitting the corner $D_{\Sigma,1} \cap D_{\Sigma, 2}$ (included in $\mathfrak{b}$) multiple times. However these disks should have strictly bigger energies than the proper transforms of $\beta_{\nu_{1}}$ and $\beta_{\nu_{2}}$ at any $u \in B$,
and hence cannot serve as a energy minimizing disk anywhere.
%
\end{proof}

Thus we have an analogue of Corollary \ref{cor:critnum} for bulk-deformed potentials.

 \begin{cor}\label{cor:bulkcritnum}
Let $(X_\Sigma,D_\Sigma)$ be a toric surface,  and let $(X, D)$ be the surface obtained after a sequence of non-toric blowups.
Generically, the number of critical of $W^\mathfrak{b}$ agrees with the rank of the cohomology of $X$. Here, $\mathfrak{b} \in H^{even} (X;\Lambda)_+$ is a bulk parameter (which requires to be the lift of a torus-invariant cycle in $X_\Sigma$ when $D$ contains a sphere with a negative Chern-number).
\end{cor}

We summarize our discussion so far in the frame of closed-string mirror symmetry for algebraic surfaces $(X,D)$ obtained by taking (a sequence of) non-toric blowup on a toric surface $(X_\Sigma,D_\Sigma)$. We consider the quantum cohomology ring $QH^\ast_\mathfrak{b} (X) =H^\ast (X;\Lambda)$ equipped with deformed cup product $\star_\mathfrak{b}$. The coefficient of $z$ in $x \star_\mathfrak{b} y$ is given by the Gromov-Witten invariants $\sum_{n} \sum_{A \in H_{2}(X)} GW_{0,n+3}^A (x,y,z, \mathfrak{b},\mathfrak{b},\cdots,\mathfrak{b})$ which roughly counts the number of holomorphic spheres passing through $x$, $y$, $z$ and arbitrary many $\mathfrak{b}$'s. For more details, see, for e.g., \cite[Chapter 11]{McS}. 

On the mirror side, we take the Jacobian ideal ring of the potential function $W$. Lemma \ref{lem:nooutr0} allows us to regard $W^{\mathfrak{b}}$ as a function defined on $\val^{-1} (R_0)$, where $R_0$ is sufficiently close to the entire base $B$ (which approximates the moment polytope of $X_\Sigma$). Then the Jacobian ideal ring $\mathrm{Jac} (W^{\mathfrak{b}})$ of $W^{\mathfrak{b}}$ can be defined as in \cite[Definition 1.3.10]{FOOO10b}, but with the moment polytope replaced by $R_0$. Roughly speaking, the Jacobian ideal ring is the quotient of the convergent power series ring $\Lambda \ll z_{1}^{\pm 1}, z_{2}^{\pm 1} \gg$, consisting of power series that converge with respect to the valuation $z^{\partial \beta_{\nu_j}} \mapsto \lambda_j + \langle u, \partial \beta_{\nu_j} \rangle$ for all $j$ and $u \in R_0$ or its associated norm.\footnote{Our $\lambda_j$ corresponds to $-\lambda_j$ in \cite{FOOO10b}.} (Recall $z_{1} = z^{\partial \beta_{\nu_{1}}}$ and $z_{2}=z^{\partial \beta_{\nu_{2}}}$.) Notice that the valuation is nothing but the areas of basic disks. 
The ideal in the quotient process is generated by partial derivatives of $W^{\mathfrak{b}}$, but we may need to take its closure with respect to $T$-adic topology. We refer readers to \cite{FOOO10b} for more details.\\

We are now ready to state our main theorem:

\begin{thm}[\emph{Theorem I}]\label{main}
In the situation of Corollary \ref{cor:bulkcritnum}, we have
\[
	\mathrm{Jac} (W^{\mathfrak{b}}) = QH^{*}_{\mathfrak{b}}(X).
\]
\end{thm}

\begin{proof}
Recall from Lemma \ref{lem:unimodal1} that for generic parameters, the Newton subdivision is a locally convenient triangulation, apart from the cell(s) containing the origin (in their closure). For each convenient triangle $\sigma$, direct computation on the leading terms yields the desired number ($= \mathrm{Vol}(\sigma)$) of non-degenerate critical points (Remark \ref{rmk:volmany}), whereas Proposition \ref{prop:gnc} covers the case when there exists a non-convenient edge. That is, for generic parameters, the bulk-deformed potential $W^{\mathfrak{b}}$ is Morse.
Hence, by \cite[Proposition 1.3.16]{FOOO10b}, the left hand side is isomorphic to the semi-simple ring $\prod_{\alpha \in crit(W_{R_0})} \Lambda$ where $\alpha$ runs over the set of critical points of $W$ over $R_0$ which equals the set of all geometric critical points of $W$ due to Lemma \ref{lem:nooutr0}. On the other hand, the right hand side is also known to be semi-simple by \cite{Bayer}. We see that both sides are semi-simple and have the same rank by Corollary \ref{cor:bulkcritnum}, which completes the proof.
 \end{proof}

\subsection{Continuum of critical points}\label{subse:continuum}

Even though we can pin-down the precise location of all (geometric) critical points of $W^{\mathfrak{b}}$ for generic $\omega$, for some special $\omega$, the bulk deformed potential might have a continuum of critical points, as observed in \cite{FOOO-T2}. Interestingly, this phenomena occurs when $X$ with a non-convenient Newton polytope is equipped with a ``locally monotone'' symplectic form. The following example analyzes the critical behavior of $\mathbb{F}_{k}$ blown up once at the origin.

\begin{figure}[h]
\centering
\subcaptionbox{%
	Blowup of $\mathbb{F}_{k}$ when $c < \frac{b}{2}$. The red point $(\frac{1}{2}a-\frac{k}{4}b,\frac{b}{2})$ supports four toric critical points, and the green point $(c,c)$ supports a single toric critical point.
	\label{fig:continuum1}%
	}
		{\begin{tikzpicture}[scale=0.54]
				
		\draw[line width =1pt] (0,1.7)--(0,5)--(5,5)--(20,0)--(1.7,0)--cycle;
		\draw[line width =0.7pt, dashed] (0,2.5)--(2.5,2.5)--(2.5,0);
		
		\draw (20,0) node[below] {$a$};
		\draw (0,5.4) node[left] {$b$};
		\draw (0,1.7) node[left] {$c$};		
		\draw (3,0) node[below] {$\frac{b}{2}$};
		\draw (1.5,0) node[below] {$c$};				
		
		\draw[line width =0.7pt, color=blue] (0,5)--(2.5,2.5)--(1.7,1.7);
		\draw[line width =0.7pt, color=blue] (0,1.7)--(1.7,1.7)--(1.7,0);		
		\draw[line width =0.7pt, color=blue] (5,5)--(7.5,2.5);
		\draw[line width =0.7pt, color=blue] (7.5,2.5)--(2.5,2.5);
		\draw[line width =0.7pt, color=blue] (7.5,2.5)--(20,0);		
		\filldraw[color=green] (1.7,1.7) circle (2pt);
		
		\draw[line width =0.7pt, dashed] (2.5+5/2,2.5)--(2.5+5/2,0);
		\draw (3,0) node[below] {$\frac{b}{2}$};
		\draw (2.5+5/2+0.08,2.59) node[above] {$P$};						
		\draw (6,0) node[below] {$\frac{1}{2}a-\frac{k}{4}b$};
		\filldraw[color=red] (2.5+5/2,2.5) circle (2pt);
		
		\end{tikzpicture}%
	}\hfill%
\subcaptionbox{%
	Blowup of $\mathbb{F}_{k}$ when $c = \frac{b}{2}$. The green point at $(\frac{1}{3}(a-\frac{k}{2}b + \frac{b}{2}),\frac{b}{2})$ supports three critical points, and the red point at $(\frac{1}{2}a-\frac{k}{4}b, \frac{b}{2})$ supports two.
	\label{fig:continuum2}%
	}
		{\begin{tikzpicture}[scale=0.54]
				
		\draw[line width =1pt] (0,2.5)--(0,5)--(5,5)--(20,0)--(2.5,0)--cycle;
		\draw[line width =0.7pt, color=blue] (0,2.5)--(2.5,2.5)--(2.5,0);
		
		\draw (20,0) node[below] {$a$};
		\draw (0,5.4) node[left] {$b$};
		\draw (0,2.3) node[left] {$c=\frac{b}{2}$};		
		\draw (1.5,0) node[below] {$c=\frac{b}{2}$};
				
		\draw[line width =0.7pt, color=blue] (0,5)--(2.5,2.5);	
		\draw[line width =0.7pt, color=blue] (5,5)--(7.5,2.5);
		\draw[line width =0.7pt, color=blue] (7.5,2.5)--(2.5,2.5);
		\draw[line width =0.7pt, color=blue] (7.5,2.5)--(20,0);
		
		\draw[line width =0.7pt, dashed] (2.5+5/3,2.5)--(2.5+5/3,0);
		\filldraw[color=green] (2.5+5/3,2.5) circle (2pt);
		\filldraw[color=red] (2.5+5/2+0.05,2.5) circle (2pt);
		
		\draw (5.5,0) node[below] {$\frac{1}{3}(a-\frac{k}{2}b + \frac{b}{2})$};
		\draw (2.5+5/2+0.08,2.59) node[above] {$P$};	
		\draw (2.5+5/3,2.5) node[above] {$Q$};			
		
		\end{tikzpicture}
	}\hfill%
\subcaptionbox{%
	The critical behavior of the bulk-deformed potential $W^{\mathfrak{b}}$ when $c=\frac{b}{2}$. 
	\label{fig:continuum3}%
	}
		{\begin{tikzpicture}[scale=0.54]
				
		\draw[line width =1pt] (0,2.5)--(0,5)--(5,5)--(20,0)--(2.5,0)--cycle;
		\draw[line width =0.7pt, dashed] (0,2.5)--(2.5,2.5)--(2.5,0);
		
		\draw (20,0) node[below] {$a$};
		\draw (0,5.4) node[left] {$b$};
		\draw (0,2.3) node[left] {$c=\frac{b}{2}$};		
				
		\draw[line width =0.7pt,dashed] (0,5)--(2.5,2.5);	
		\draw[line width =0.7pt, dashed] (5,5)--(7.5,2.5);
		\draw[line width =0.7pt, dashed] (7.5,2.5)--(2.5,2.5);
		\draw[line width =0.7pt, dashed] (7.5,2.5)--(20,0);
		
		\draw[line width =1.5pt,-latex, color=red] (2.5,2.5)--(2.5+5/3,2.5);
		\draw[line width =1.5pt,-latex, color=blue] (2.5+5/2,2.5)--(2.5+5/3,2.5);		
				
		\draw[line width =0.7pt, dotted] (2.5+5/3,2.5)--(2.5+5/3,0);
		\draw[line width =0.7pt, dotted] (2.5+5/2,2.5)--(2.5+5/2,0);
		\filldraw (2.5+5/3,2.5) circle (2pt);		
		
		\filldraw (2.5+5/2,2.5) circle (2pt);
		\draw (2.5+5/2+0.08,2.59) node[above] {$P$};
		\draw (2.5+5/3,2.5) node[above] {$Q$};
		
		\end{tikzpicture}
	}%
\caption{}
\label{fig:continuum0}
\end{figure}

 
Let $X$ be the blowup of $\mathbb{F}_{k}$ at the origin (Figure \ref{fig:continuum1}). The Hori-Vafa part of the potential is given by 
\[
	W^{HV}_{\Sigma} = z_{1} + z_{2} + \frac{T^{b}}{z_{2}} + \frac{T^{a}}{z_{1}z_{2}^{k}} + T^{-c}z_{1}z_{2}.
\]
First note that if we keep our assumption $c < \frac{b}{2}$ from \eqref{condi:blowup-size}, then the three terms $z_{1}$, $z_{2}$, and $T^{-c}z_{1}z_{2}$ gives rise to a new toric critical point $\alpha$ with $\val (\alpha) = (c,c)$. We observe the change in the critical behavior of $W_{\Sigma}$  and that of the bulk deformed potential $W^{\mathfrak{b}}$ when $c = \frac{b}{2}$.

Let $c = \frac{b}{2}$. We follow the steps of the proof of Proposition \ref{prop:gnc} analogously. Normalizing with respect to $z_{2}=T^{b/2}\underline{z_{2}}$, we have 
\[
	W^{HV}_{\Sigma} = z_{1} + T^{b/2}\Big( \underline{z_{2}} + \frac{1}{\underline{z_{2}}} \Big) + T^{a-\frac{k}{2}b}\frac{1}{z_{1}\underline{z_{2}}^{k}} + T^{-c+\frac{b}{2}}z_{1}\underline{z_{2}},
\]
and substituting $\underline{z_{2}}=-1+z_{2}^{+}$ gives
\begin{align*}
W^{HV}_{\Sigma} &= z_{1} + T^{\frac{b}{2}}\Big( -({z_{2}}^{+})^{2} - {(z_{2}^{+})}^{3} -  \cdots \Big) \pm  
				\frac{T^{a-\frac{k}{2}b}}{z_{1}}\Big(1+z_{2}^{+}+{(z_{2}^{+})}^{2}+\cdots\Big)^{k} + z_{1}(-1+z_{2}^{+})  \\
 &= - T^{\frac{b}{2}}{(z_{2}^{+})}^{2}  \pm  \frac{T^{a-\frac{k}{2}b}}{z_{1}} + z_{1}z_{2}^{+}  + \cdots.
\end{align*}
where the two terms $z_{1}$ and $-T^{-c+\frac{b}{2}}z_{1}$ have canceled out one another. The leading terms become convenient, hence by Proposition \ref{prop:tropcrit} (after normalizing), we conclude that there are three  critical points located at $Q = (\frac{1}{3}(a-\frac{k}{2}b + \frac{b}{2}), \frac{b}{2}) \in e$, and two critical points at $ P =(\frac{1}{2}a-\frac{k}{4}b, \frac{1}{2}b)$. See Figure \ref{fig:continuum2}.\\

Now consider (the $z_{2}^{+}$ expansion of) the bulk-deformed potential:
\[
	 W^{\mathfrak{b}} = - T^{\frac{b}{2}}{(z_{2}^{+})}^{2}  \pm  \frac{T^{a-\frac{k}{2}b}}{z_{1}} - z_{1}z_{2}^{+} + T^{\epsilon}z_{1} + \cdots
\]
Notice that $T^{ \epsilon}z_{1}$ can now become an energy minimizer. 
If $0<\epsilon<\frac{1}{3}(a-\frac{k}{2}b -b )$, one can show by following the steps of the proof of Proposition \ref{prop:gnc}, that $T^{ \epsilon}z_{1}$ indeed serves as an energy minimizer\footnote{To be precise, one needs to expand the potential to the third order so that the coefficients satisfies the conditions in Lemma \ref{lemma:appendix}.}. We obtain two critical points that are located at
\[
	\Big(	\frac{1}{2}(a-\frac{k}{2}b -\epsilon), \, \frac{1}{4}(a-\frac{k}{2}b - b + \epsilon)	\Big),
\]
and a single critical point at
\[
	(\epsilon + \frac{b}{2} , \epsilon)
\]
as shown in Figure \ref{fig:continuum3}. The former two critical points' change in position (as $\epsilon$ varies) is depicted in a blue arrow, whereas the latter is depicted in a red arrow. 
We omit the details.

\section{Homological Mirror Symmetry}\label{sec:HMS}

We now discuss open-string (homological) mirror symmetry within our geometric context. For a surface $X$ obtained by taking a non-toric blowup of a toric surface $X_\Sigma$, let us consider the Fukaya category $\mathcal{F}_{0} (X)$ of $X$ consisting the fibers of the SYZ fibration on $X$. More precisely, these Lagrangian torus fibers are equipped with $\Lambda_U$ flat connections, and their Floer cohomologies are well-defined as long as their curvature $m_0(1)$'s coincide (recall they are all unobstructed by Lemma \ref{lem:torusweakunobs}).
We allow $\mathcal{F}_{0}$ to contain nodal fibers with weak-bounding cochains from immersed generators if $X$ is semi-Fano, in which case we can guarantee their weakly-unobstructedness (see Remark \ref{rmk:nodalweakunobs}).

Nontrivial objects in $\mathcal{F}_0$ can be detected from singularity information of $W$. Recall from \ref{subsec:LFT} that each point in the domain $\check{Y}$ of the mirror LG model $W:\check{Y} \to \Lambda$ represents a Lagrangian brane supported over a SYZ fiber. An object of $\mathcal{F}_0 (X)$ corresponding to a point $p \in \check{Y}$ is nontrivial if and only if $p$ is a critical point of $W$ (see \ref{subsec:critofpot} below). 
On the other hand, we consider the category of singularities for $W$ on the mirror side. Assuming generic parameters, all of its (geometric) critical points are non-degenerate, and hence the category decomposes into skyscraper sheaves supported at critical points of $W$, where the morphism space between any two different skyscraper sheaves is trivial.

\begin{remark}
Although we focus on $W$ to keep the exposition simple, the entire argument in this section has a straightforward generalization to the bulk-deformed potential $W^\mathfrak{b}$. 
\end{remark}

We begin by finding nontrivial objects (generators) of $\mathcal{F}_0 (X)$ associated with critical points of $W$.

\subsection{Critical points of $W$ and the Floer cohomology of the associated Lagrangian}\label{subsec:critofpot}

Suppose $p=(z_{1},z_{2})$ is a critical point of $W$. By taking the (unique) factorization $p=(T^{\val(z_{1})} \underline{z}_{1},T^{\val(z_{2})} \underline{z}_{2})$ for $(\underline{z}_{1},\underline{z}_{2}) \in \Lambda_U^2$, we can assign a Lagrangian torus $L_u$ fiber sitting over $u:=(\val(z_{1}),\val(z_{2}))$ equipped with the bounding cochain $b = x_{1} d \theta_{1} + x_{2} d \theta_{2} \in H^1 (L_u;\Lambda_+) \cong (\Lambda_+)^2$ such that $e^{b} =( \underline{z}_{1},\underline{z}_{2})$. In view of \eqref{eqn:resttofiberwiseW}, $p$ being a critical point is equivalent to $b$ being a critical point of $W_u (b)$, where
\begin{equation}\label{eqn:weakmcb1}
W_u (b) \cdot [L_u] = m_0 (1) + m_{1}(b) + m_{2}(b,b) + \cdots .
\end{equation}

Differentiating \eqref{eqn:weakmcb1}, we see that the differential $m_{1}^{b,b}$ on $\hom ((L_u,b),(L_u,b)) \cong H^\ast (L_u,\Lambda)$ vanishes on $H^1 (L_u,\Lambda)$ as it acts trivially on the generators $d\theta_{1},d\theta_{2}$. Moreover, differentiating \eqref{eqn:weakmcb1} once more, we see that $m_{2}^{b,b,b}$ on these generators satisfies
\[
	m_{2}^{b,b,b} (d \theta_{1}, d \theta_{2}) + m_{2}^{b,b,b} (d \theta_{2}, d \theta_{1}) = \partial_{12} W \cdot [L_u].
\]
Notice that $d \theta_{1} \wedge d \theta_{2}$ cannot be an image of $m_{1}^{b,b}$ which vanishes on $d \theta_{1}$ and $d \theta_{2}$, and $m_{2}^{b,b,b} (d \theta_{1}, d \theta_{2}) = d \theta_{1} \wedge d \theta_{2} + c \cdot [L_u]$ for some $c \in \Lambda$ where $[L_u]$ serves as the unit. Consequently, the Lagrangian brane $(L_u, b)$, considered as an object of $\mathcal{F}_0 (X)$, has its endomorphism algebra quasi-isomorphic to the Clifford algebra with respect to the quadratic form $Hess_b (W_u)$ (the Hessian of $W$ at $b$), which equals $Hess_p (W)$. Observe that this calculation is still valid even if there exist nontrivial contributions of stable disks with negative Maslov indices, and we use crucially the fact that $\dim L_u$ is $2$. Thus we have the following:

\begin{prop}\label{prop:hffibers} For an object $L(p):=(L_{(\underline{z}_{1},\underline{z}_{2})}, b=(x_{1},x_{2}))$ corresponding to the critical point $p = (\underline{z}_{1} e^{x_{1}}, \underline{z}_{2} e^{x_{2}}) $ in $X$, its endomorphism space in $\mathcal{F}_0 (X)$ is quasi-isomorphic to the Clifford algebra with respect to $Hess_p (W)$.

Moreover, for two distinct critical points $p_{1}$ and $p_{2}$ of $W$, $\hom_{\mathcal{F}_0(X)} (L(p_{1}),L(p_{2}))$ is trivial.
\end{prop}

\begin{proof}
It only remains to prove the last assertion, so consider $p_{1}\neq p_{2}$ with $\val(p_{1}) = \val(p_{2})$ and $W(p_{1}) = W(p_{2})$ (as otherwise, $\hom_{\mathcal{F}_0(X)} (L(p_{1}),L(p_{2}))$ would be automatically trivial, or not even defined). The situation boils down to computing the cohomology of 
\[
	m_{1}^{b_{1},b_{2}}:CF( (L_u,b_{1}), (L_u,b_{2})) \to CF( (L_u,b_{1}), (L_u,b_{2}))
\]
for two different weak bounding cochains $b_{1}$ and $b_{2}$ in $H^1 (L_u;\Lambda_+)$ such that $W_u (b_{1}) = W_u (b_{2})$. We use the Morse-Bott model for $CF( (L_u,b_{1}), (L_u,b_{2})$, in that it is still generated by $[L_u]$ and $d \theta_{1} \wedge d \theta_{2}$ in even degree, and $d\theta_{1}, d\theta_{2}$ in odd degree.

Observe first that $m_{1}^{b_{1},b_{2}} ([L_u])= b_{1} - b_{2} \neq 0$. Therefore, modulo the image, the odd degree part of $HF( (L_u,b_{1}), (L_u,b_{2})$ is generated by one of $d \theta_i$, say $d \theta_{1}$, or it vanishes. However, due to the classical part (the de Rham differential) of $m_{1}^{b_{1},b_{2}}$, 
\begin{equation}\label{eqn:m1d1}
 m_{1}^{b_{1},b_{2}} (d \theta_{1}) = C d \theta_{1} \wedge d \theta_{2} + E [L_u] \neq 0
\end{equation}
for some constants $C$ and $E$ with $C \neq 0$ ($C$ would cancel out if $b_{1}= b_{2}$), which gives a contradiction. 

On the other hand, \eqref{eqn:m1d1} implies that the even degree part of $HF( (L_u,b_{1}), (L_u,b_{2})$ should be generated by $[L_u]$ unless it is trivial. Since $[L_u]$ is not a cycle, the even degree component should also vanish.
\end{proof}


\subsection{Nontrivial objects in $\mathcal{F}_0(X)$ from non-toric critical points}\label{subsec:critval}

We look into the nontrivial object of $\mathcal{F}_0 (X)$ arising from the blowup process to see how the category changes under the blowup. Recall that each blowup (at a non-torus-fixed point) introduces a critical point to $W$ which we call a non-toric critical point. We are interested in the associated Lagrangian torus fiber (coupled with a suitable bounding cochain), especially in its behavior when the position of a blowup point changes.

We first find the precise location of such a critical point on the SYZ base (i.e., the valuation of a non-toric critical point). We also show that critical values for non-toric critical points of $W$ are mutually distinct. Clearly, the leading order terms $\overline{W}$ of $W$ only matter for our purpose.

As before, we look at the local expression of $W$ around exceptional divisors arising from blowing up at points in $D_{\Sigma,2}$ (and hence, near $D_{\Sigma,1} \cap D_{\Sigma,2}$ by our earlier choice) in $X_\Sigma$. Without loss of generality, we assume $\nu_{1}=(1,0)$, $\nu_{2}=(0,1)$, and $\lambda_{1}=\lambda_{2}=0$. In  the moment polytope $\Delta_\Sigma$, critical points of our interest lie on the $x$-axis, close to the origin. 
We stick to this setting throughout \ref{subsec:critval} and \ref{subsec:geomnon-toric}.

Suppose that there are $k$ distinct such blowup points in $D_{\Sigma,2}$, and that the sizes of exceptional divisors are $\epsilon_{1},\cdots,\epsilon_k$ in increasing order. The leading terms $\overline{W}$ of $W$ locally near the origin of $\Delta_\Sigma$ are given as in \eqref{stdlocalmodel}. We have seen in \ref{subsec:non-toriclocal} that there are $k$-many critical points of $\overline{W}$ all of which extend to critical points of $W$ via energy induction. Let  $\alpha_{1},\cdots, \alpha_k$ denote these critical points (of $W$) where $\alpha_i$ corresponds to the exceptional divisor of size $\epsilon_i$.

Alternatively, we can express $\overline{W}$ as
\[
	\overline{W} = z_{2} + z_{1} \prod_{i=1}^{k}{\left(1+T^{-\epsilon_{i}}z_{2}\right)}. 
\]
Taking the partial derivative with respect to $z_{2}$, we obtain
\[
	\frac{\partial}{\partial z_{2}}\overline{W} = 1 + z_{1} \sum_{i=1}^{k} T^{-\epsilon_{i}}\prod_{j\neq i}{\left(1+ T^{-\epsilon_{j}}z_{2}\right)}.
\]
Substituting $z_{2}=-T^{\epsilon_{i}}$ and solving for $z_{1}$, we have 
\[
z_{1} = - T^{\epsilon_{i}} \prod_{j\neq i} \frac{1}{1-T^{\epsilon_{i}-\epsilon_{j}}}.
\]
We can expand each term using Taylor series. If $j < i$, then $\epsilon_{i}-\epsilon_{j} <0$, and we have 
\[
	\frac{1}{1-T^{\epsilon_{i}-\epsilon_{j}}} = -\frac{T^{\epsilon_{j}-\epsilon_{i}}}{1-T^{\epsilon_{j}-\epsilon_{i}}} = -T^{\epsilon_{j}-\epsilon_{i}}\left(1 + T^{\epsilon_{j}-\epsilon_{i}} + T^{2(\epsilon_{j}-\epsilon_{i})}+\cdots \right).
\]
Otherwise, $\frac{1}{1-T^{\epsilon_i - \epsilon_j}}$ is of valuation $0$.
Therefore,
\[
	z_{1} = (-1)^{k}T^{\epsilon_{i}}T^{\sum_{j<i}\epsilon_{j}-\epsilon_{i}} + T^{\text{h.o.t.}}
\]
and we thus have
\[
	\val (\alpha_{i}) = \left(\epsilon_{i}+\sum_{j=1}^{i-1}(\epsilon_{j}-\epsilon_{i}), \,\, \epsilon_{i}\right).
\]
Due to our choice of $\lambda_{1}=\lambda_{2}=0$, the moment polytope $\Delta_\Sigma$ is located in the 1st quadrant of $\mathbb{R}^2$, and our local calculation is performed near the origin. In particular, $\alpha_i$ sits in the interior of $\Delta_\Sigma$. Moreover, by choosing the locations of the blowup points sufficiently close to $D_{\Sigma,1} \cap D_{\Sigma,2}$ (which are irrelevant to $\epsilon_i$), we can ensure that all $\alpha_{1},\cdots, \alpha_k$ lie within $R_0$ where the local expansion $W = \overline{W} + T^{h.o.t.}$ is valid (see Figure \ref{fig:relposofblowup}). Therefore we obtain nontrivial objects $L(\alpha_{1}),\cdots, L(\alpha_k)$ of $\mathcal{F}_0 (X)$ supported over SYZ fibers at $\val(\alpha_{1}),\cdots,\val(\alpha_k)$.

Now, let's proceed to calculate the valuation of the critical value of $W$ at $\alpha_{i}$:
\begin{align*}
\val ( W\rvert_{\alpha_{i}}) &= \val \left( T^{\epsilon_{i}} + \left(T^{\epsilon_{i}+\sum_{j=1}^{i-1}(\epsilon_{j}-\epsilon_{i})}\right)\prod_{i=1}^{k}{\left(1 + \frac{1}{T^{\epsilon_{i}}}T^{\epsilon_{i}}\right)}\right) \\
 &= \mathrm{min}\{ \epsilon_{i}, \, \epsilon_{i}+\sum_{j=1}^{i-1}(\epsilon_{j}-\epsilon_{i})\} = \epsilon_{i}.
\end{align*}
Therefore, as long as $\epsilon_{i}$'s are generic, the critical values (of non-toric critical points) are all distinct. 

Moreover, Lemma \ref{lem:unimodal1} tells us that generically, toric critical points also have distinct eigenvalues. For the case where the central cell is non-convenient, one can check by direct computation that the critical values are indeed distinct for each critical point \eqref{nonconvcrit} found in the proof of \ref{prop:gnc}. The following is a direct consequence of the discussion so far combined with Proposition \ref{prop:hffibers}.

\begin{prop}\label{prop:nontrivfibers}
Let $X$ be obtained after a sequence of toric/non-toric blowups on $X_{\Sigma}$, where we assume that the non-toric blowup centers are close enough to the corner of the moment polytope.\footnote{i.e. blowup points in the interior of the toric divisor $D_{j}$ are close enough to $D_{\Sigma,j} \cap D_{\Sigma,j-1}$ for each $j$ according to our earlier convention.} Then each toric/non-toric blowup gives rise to a Floer-nontrivial SYZ fiber in $X$, and hence a nontrivial object in $\mathcal{F}_0 (X)$. Its endomorphism algebra is isomorphic to the Clifford algebra with respect to the Hessian of $W$ at its corresponding critical point.

For generic parameter, all such Floer-nontrivial SYZ fibers live in mutually distinct component of $\mathcal{F}_0 (X)$ when decomposed into potential values, and hence they do not interact with each other.
\end{prop}

When $k=1$, then the unique critical point created by this blowup is located at $(\epsilon,\epsilon)$ (for $\epsilon=\epsilon_{1}$ in the notation above). This precisely coincides with the effect of the toric blowup at $D_{\Sigma,1} \cap D_{\Sigma,2}$ with the size of the exceptional divisor $\epsilon$. Indeed, the two (toric and non-toric ones) are equivalent in view of symplectomorphism constructed in the proof of Lemma \ref{lem:torusweakunobs}. Apart from technical details (in achieving transversality for Floer theory), the associated Lagrangian should coincide with what is called the exceptional brane in \cite{VXW} (which was constructed in the local model of the toric blowup).

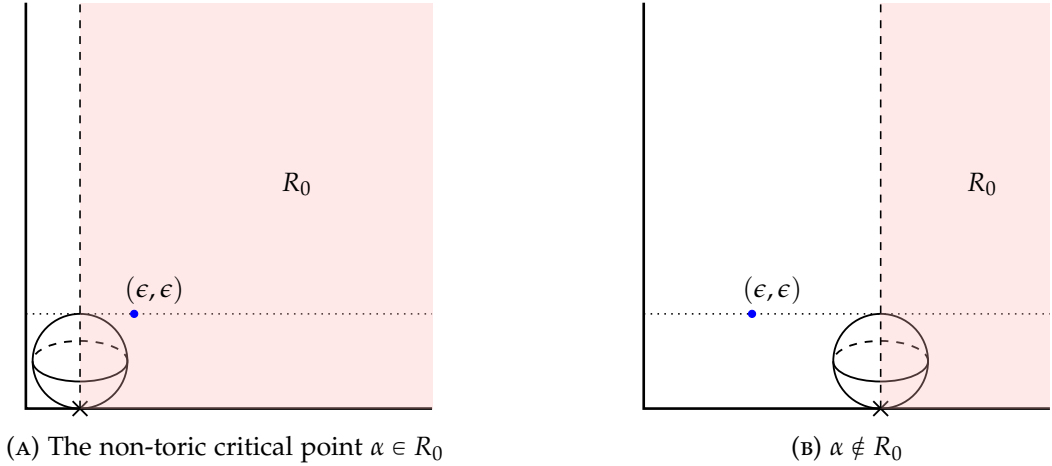
\begin{figure}[h]
\centering
\subcaptionbox{%
	The non-toric critical point $\alpha \in R_0$
	\label{fig:R_0in}%
	}
	{\makebox[0.45\linewidth][c]
		{\begin{tikzpicture}[scale=0.9]

\draw[line width=1pt] (0,6)--(0,0)--(6,0);
\draw[line width=0.6pt, dotted] (0,1.4)--(6,1.4);

\draw[line width=0.7pt] (0.8,0.7) circle (0.7);
\draw[line width=0.7pt] (0.8,0.7) +(0:0.7 and 0.3) arc (0:-180:0.7 and 0.3);
\draw[line width=0.7pt, dashed] (0.8,0.7) +(0:0.7 and 0.3) arc (0:180:0.7 and 0.3);
\draw[line width =0.7pt] (0.8,0) node[cross] {};

\fill[red!30, draw=none, opacity=0.3] (0.8,6)--(0.8,0)--(6,0)--(6,6);
\draw (4,3) node[above] {\small{$R_{0}$}};
\filldraw[blue] (1.6,1.4) circle (1.5pt);
\draw (1.9,1.4) node[above] {\small{$(\epsilon, \epsilon)$}};
\draw[line width=0.6pt, dashed] (0.8,0)--(0.8,6);

		\end{tikzpicture}%
	}
	}\hfill%
\subcaptionbox{%
	$\alpha \notin R_0$
	\label{fig:R_0out}%
	}
	{\makebox[0.45\linewidth][c]	
		{\begin{tikzpicture}[scale=0.9]

\draw[line width=1pt] (0,6)--(0,0)--(6,0);
\draw[line width=0.6pt, dotted] (0,1.4)--(6,1.4);

\draw[line width=0.7pt] (3.5,0.7) circle (0.7);
\draw[line width=0.7pt] (3.5,0.7) +(0:0.7 and 0.3) arc (0:-180:0.7 and 0.3);
\draw[line width=0.7pt, dashed] (3.5,0.7) +(0:0.7 and 0.3) arc (0:180:0.7 and 0.3);
\draw[line width =0.7pt] (3.5,0) node[cross] {};

\fill[red!30, draw=none, opacity=0.3] (3.5,6)--(3.5,0)--(6,0)--(6,6);
\draw (5,3) node[above] {\small{$R_{0}$}};
\filldraw[blue] (1.6,1.4) circle (1.5pt);
\draw (1.9,1.4) node[above] {\small{$(\epsilon, \epsilon)$}};
\draw[line width=0.6pt, dashed] (3.5,0)--(3.5,6);

		\end{tikzpicture}
		}
	}%
\caption{Local picture of the relative position of the blowup center with respect to $\epsilon$ and corresponding locations of the non-toric critical point.}
\label{fig:relposofblowup}
\end{figure}

\subsection{Geometric meaning of critical points lying outside $R_0$}\label{subsec:geomnon-toric}
Finally, we can remove the assumption on the location of the blowup center in Proposition \ref{prop:nontrivfibers} by including nodal fibers under the semi-Fano assumption on $X$. 
It is natural to expect that the singularity information of its LG mirror is independent of positions of the blowup center in any case, but the only reason we impose this assumption is to ensure the weakly unobstructedness of nodal fibers (the argument in the proof of Lemma \ref{lem:torusweakunobs} is obviously not valid for these fibers). We speculate that nodal fibers are weakly unobstrcuted for all bounding cochains formed by immersed generators beyond the semi-Fano case.

For the sake of simplicity, we assume $k=1$, and the argument easily generalizes for $k>1$. We work with the same setting as above, but now we allow the blowup points to be projected onto $(\delta,0)$ in the moment polytope, where $\delta$ is greater than the blowup size $\epsilon$. The situation is described in Figure \ref{fig:R_0out}, whereas previously we worked with the setting in Figure \ref{fig:R_0in} (cf. Proposition \ref{prop:nontrivfibers}).

Notice that for this choice of the blowup point, the critical point $(\epsilon,\epsilon)$ falls in the chamber  $R_{1}$ on the left side of $R_0$ sharing a single vertical wall (with respect to our choice of coordinates). On the other hand, the leading terms for $W_{R_{1}}$ near the origin is given as 
\[
	\overline{W} = z_1+ z_2,
\]
which does not admit any critical points (while being convenient). This seemingly contradictory phenomenon occurs because analytic-continuation of $W_{R_0}$ over $R_{1}$ cannot cover a certain region in $R_{1}$ due to singularity of the coordinate change (a cluster transformation) between $R_0$ and $R_{1}$ (see Remark \ref{rmk:failac}). 

More concretely, let $(z_{1}',z_{2}')$ and $(z_{1},z_{2})$ denote mirror affine coordinates over $R_{1}$ and $R_0$, respectively, induced by the disk classes $\nu_{1}$ and $\nu_{2}$. 
Recall from \eqref{eqn:wcformualepsilon} that they are related by
\begin{equation*}
\begin{array}{l}
 z_{1} = z_{1}' (1+ T^{-\epsilon} z_{2}') \\
 z_{2} = z_{2}'
 \end{array}
 \end{equation*}
in an infinitesimal neighborhood of the wall between $R_0$ and $R_{1}$. The map gives a diffeomorphism away from $\{z_{2}' = -T^{\epsilon}\}$. Besides, if $z_{2}' = -T^{\epsilon} + T^{h.o.t.}$, then the transition can change the valuations of points dramatically, and hence the continuation of $W_{R_0}$ in \ref{subsubsec:globalizecoord} cannot cover the region over $R_{1}$ given by $\{z_{2}' \in T^{\epsilon} ( -1 + \Lambda_+)\}$.
Notice that the critical point $\alpha_0=(-T^{\epsilon}, -T^{\epsilon})$ of $\overline{W}$ and hence $\alpha = (-T^{\epsilon} + T^{h.o.t.} , -T^{\epsilon} + T^{h.o.t.})$ of $W$ precisely lie on this region. For this reason, it is impossible to directly identify the geometric object associated with $\alpha$.

To remedy, we bring in an additional mirror chart induced by the (Floer) deformation space of the nodal fiber in $X$ responsible for the wall between $R_0$ and $R_1$. Such a technique of extending of the mirror space has first appeared in \cite{HKL}. See \cite[7.1]{BCHL} for its usage in a similar context.
Recall that near the exceptional divisor, the SYZ fibration on $X$ is locally modeled on the one given in \cite{auroux09}. 
The nodal fiber $\mathbb{L}$ in our situation is an immersion of a sphere that has a single transversal self-intersection. It produces two degree 1 immersed generators, say  $U$ and $V$. One can deform the Fukaya algebra $CF(\mathbb{L},\mathbb{L})$ by the weak bounding cochain $b= u U + v V$ for $(u,v) \in (\Lambda_+)^2$. 

The coordinate transitions between this new $(u,v)$-chart and $R_i$ are given as follows.
Consider regular fibers $L_{0}$ and $L_{1}$ over $R_0$ and $R_1$, respectively and assume that they are close enough to $\mathbb{L}$. Their projections to the divisor $D_2$ is depicted in Figure \ref{fig:coordchloc}, where the origin represents $D_1 \cap D_2$.
It will be convenient to introduce a local model for a neighborhood of the fiber $\mathbb{L}$, which is $\mathbb{C}^2 \setminus \{xy = 1\}$ equipped with a Lagrangian torus fibration given by $(|x|^2 - |y|^2, |xy - 1|)$. It models a neighborhood of $D_2 \setminus D_1 \cap D_2$ where (a part of) $D_1$ is given as $\{xy = 1\}$. Notice that the conic fibration $(x,y) \mapsto xy-1=z\in \mathbb{C}^\times$ can be interpreted as the reduction of $\mathbb{C}^2 \setminus \{xy = 1\}$ by the circle action $\rho \mapsto (\rho x, \rho^{-1} y)$. It is nothing but the pull-back of the toric action generated by $\nu_2 \in \Sigma$.

The projections of $L_i$ and $\mathbb{L}$ to the divisor $D_2$ are depicted in Figure \ref{fig:coordchloc}. See \cite[4.2]{BCHL} for more details (Figure 7 therein will be helpful to understand the picture).
In fact, all fibers are sitting over the concentric circles in the $z:=xy-1$-plane centered at $0$, and $\mathbb{L}$ corresponds to the fiber $\{|x|^2 - |y|^2 = 0, |xy - 1| = 1\}$ projecting to $|z|=1$. We may choose $L_i$ to be $L_{0}:=\{|x|^2 - |y|^2 = 0, |z|=1+\epsilon \}$ and $L_{1}:=\{|x|^2 - |y|^2 = 0, |z|=1-\epsilon \}$ for some small $\epsilon$.

\begin{figure}[h]
	\begin{center}
		\includegraphics[scale=0.42]{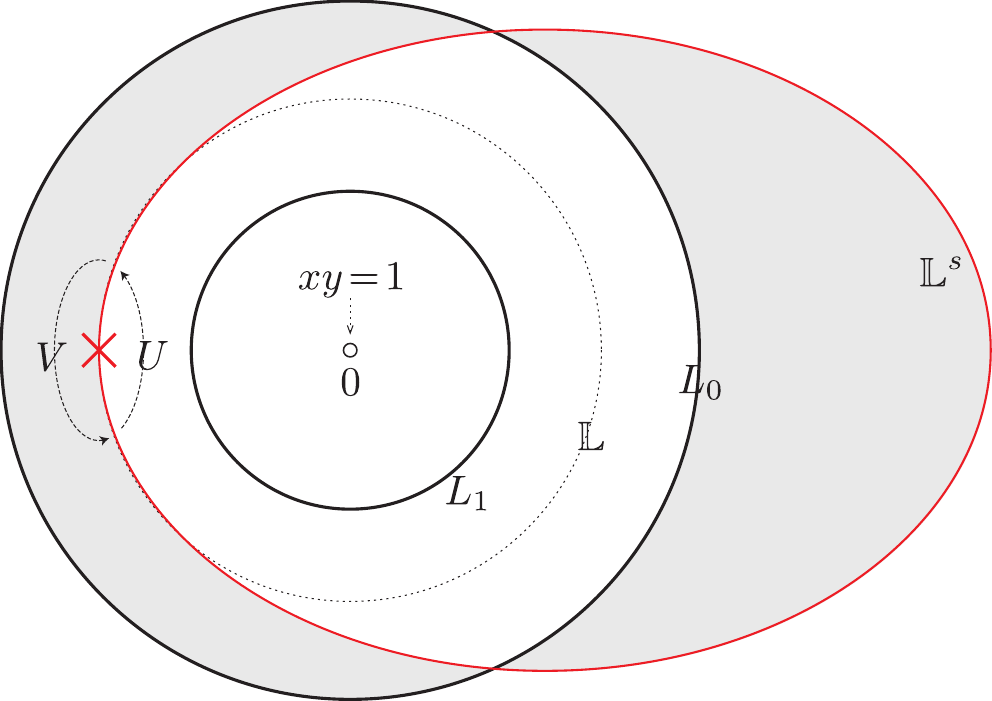}
		\caption{The projections of SYZ fibers near a nodal fiber}
		\label{fig:coordchloc}
	\end{center}
\end{figure}

Let $\mathbb{L}^{s}$ be a perturbation of $\mathbb{L}$ obtained by slightly modifying the base circle $|z| = 1$ 
as in Figure \ref{fig:coordchloc}. We want to compare two Lagrangians $(\mathbb{L}, uU + vV)$ and $(L_0, \nabla^{\uz_{1},\uz_{2}})$ each coupled with a suitable bounding cochain. More precisely, the latter is equipped with the flat connection $\nabla^{(\uz_{1},\uz_{2})}$ for $(\uz_{1},\uz_{2}) \in \Lambda_U^2$. In our picture, the holonomy along the direction of the circle action is set to be $\uz_{2}$ and that around a fixed choice of a horizontal lift of $\{z=1-\epsilon \}$ is $\uz_{1}$.
It was shown in \cite{HKL}  that $(\mathbb{L}, uU + vV)$ and $(L_{1},\nabla^{(\uz_{1},\uz_{2})})$ are isomorphic to each other (as objects of the Fukaya category) if and only if
\begin{equation}\label{eqn:blcoord}
T^{\delta(s)} \uz_{1} = v, \quad \uz_{2} = uv-1
\end{equation}
where $\delta(s) >0$ is the area difference between two shaded strips in Figure \ref{fig:coordchloc}. It is positive due to our particular choice of perturbation $\mathbb{L}^s$. We have some flexibility on $\delta$ by adjusting the size of the base circle for $\mathbb{L}^s $.

%
%
 
For the critical point $\alpha = (z_{1},z_{2})$, we have $\uz_{2} = -1+ T^{h.o.t.}$, and \eqref{eqn:blcoord} can be solved to get $(u,v) \in \Lambda_+^2$. Since $\alpha_0$ is certainly not a critical point of the global potential of $W$, $\alpha$ carries nontrivial higher order terms and, $u,v \neq 0$ for such $\alpha$. Thus $(u,v)$ lies in the region where the coordinate change \eqref{eqn:blcoord} is valid. 
The Floer potential for $\mathbb{L}$ can be obtained by applying this coordinate change to $W$ over $R_0$, and from the discussion so far, it admits a critical point $(u,v)$ which transfers to $\alpha$ under \eqref{eqn:blcoord}. This gives us a nontrivial object in $\mathcal{F}_0 (X)$ whose underlying Lagrangian in $\mathbb{L}$, although $\alpha$ itself does not directly represent a geometric object. 

Finally we remark that in the earlier situation where $R_0$ contains all the critical points (Figure \ref{fig:R_0in}), the above process does not create any new critical points. In fact, one can check by direct order-by-order calculation similar to the proof of Proposition \ref{prop:gnc} that the Floer potential (for $\mathbb{L}$) does not admit a critical point over its associated chart $(\Lambda_+)^2$ if the location of $\mathbb{L}$ is too close to the corner.

\subsection{Equivalence of categories}

From the discussion so far, $\mathcal{F}_0 (X)$ is generated by SYZ fibers $L(p)$ equipped with suitable bounding cochains corresponding to the critical points $p$ of $W$. For a regular point $p$ of $W$, $L(p)$ gives a trivial object since the unit class in $\hom(L(p),L(p))$ is an image of the differential, and hence is zero in the cohomology.
Thus the triangulated category $D^b \mathcal{F}_0 (X)$ is essentially the derived category of modules over the direct sum of as many Clifford algebras as the number of critical points of $W$ (associated with the Hessian of $W$ at those critical points).
%

\begin{remark}
By \cite{auroux07}, the computation of critical values of $W$ (for non-toric critical points) tells us that the $\epsilon_i$-eigenspace $QH(X)_{\epsilon_i}$ of $QH(X)$ with respect $c_{1}(X) \star -$ is $1$-dimensional. Assuming the extension of \cite[Corollary 1.12]{She} to non-monotone setting, this would imply that the the nontrivial object $L (\alpha_i)$ generates the corresponding component of the genuine Fukaya category. 
\end{remark}

On the mirror side, we have a decomposition of the singularity category of $W$ into critical values of $W$, that is $\oplus_\lambda D^{b}_{sing} (W^{-1} (\lambda))$. Equivalently, one may consider the matrix factorizations of $W-\lambda$. Clearly each factor $D^{b}_{sing} (W^{-1} (\lambda))$ is generated by skyscraper sheaves at critical points of $W$ whose corresponding critical values are $\lambda$.
Let us consider one of such skyscraper sheaves $\Bbbk(p)$ in $\mathcal{D}^{b}_{sing} (W^{-1} (\lambda))$. (More precisely, it is the image of $\mathcal{O}_p$ in the quotient category $\mathcal{D}^{b}_{sing} (W^{-1}(\lambda_p))$ by perfect complexes.) It is well-known that the endomorphism of $\Bbbk(p)$ is quasi-isomorphic to the Clifford algebra associated with the Hessian of $W$ at the corresponding point. See for e.g., \cite{Dy} or \cite{Tel}. Thus we proved:

\begin{thm}\label{thm:HMS}
There is an equivalence between $D^b \mathcal{F}_0 (X)$ and $\oplus_\lambda D^{b}_{sing} (W^{-1} (\lambda)) (\cong MF(W))$. Both categories admit orthogonal decomposition with respect to critical (potential) values $\lambda$, and each summand in the decomposition is generated by the skyscraper sheaf at the unique critical point whose value is $\lambda$ or its corresponding SYZ fiber.
\end{thm}

%
%
%
%
%
%
%
%
%
%
%
%
%
%

\appendix

\section{Estimates for the second-order expansion of W}\label{APPENDIX}

We prove that $(\partial_{x_+} W, \partial_{y_+} W) =0$ has a unique solution by considering in a more abstract setup as follows. $F$ below generalizes $(\partial_{x_+} W, \partial_{y_+} W) $ minus the constant terms. Readers are warned that $a,b,c,d$ in the statement represent some general coefficients, and are irrelevant to the symplectic sizes of toric divisors in $\mathbb{F}_k$ in Proposition \ref{prop:gnc}.

\begin{lemma}\label{lemma:appendix}
Let $F: (\Lambda_+)^2 \to (\Lambda_+)^2$ is given in the form of
\[
	F(z_1,z_2) = \left(a z_{1} + b z_{2} +  \sum_{v \neq e_{1},e_{2}} \lambda_v z^v, c z_{1} + d z_{2} +\sum_{v \neq e_{1},e_{2}} \eta_v z^v \right)
\]
for generic coefficients $(a,b,c,d) \in \Lambda_+^4$, where the terms $\lambda_v z^v$, $\eta_v z^v$ are higher order terms. We assume the following :
\begin{enumerate}
 \item $ad - bc \neq 0$, 
  \item $\val(a) \leq \val(b), \val(c)$ and $\val(d) \leq \val(b), \val(c)$ ($a$ and $d$ are  not assumed to be comparable), 
  \item $\val(a) + \val(d) < \val(b) + \val(c)$, 
  \item \begin{equation}\label{eqn:lambdaetamixed}
\val(\lambda_{(i,j)}) \geq \val(d) \,\,\mbox{for}\,\, j \geq1,\quad \mbox{and}\,\, \val(\eta_{(i,j)}) \geq \val(a)\,\, \mbox{for}\,\, i \geq 1.
\end{equation}
\end{enumerate}
Then $F = (C_{1},C_{2})$ has a unique solution for any $C_{1}$ and $C_{2}$ satisfying $\val(C_{1}) > \val(a) + \epsilon$ and $\val(C_{2}) > \val(d) +\epsilon$.
\end{lemma}

\begin{proof}

Since $\lambda_v z^v$, $\eta_v z^v$ are higher order terms, we have 
\[
	\val(\lambda_v z^v) > \min \{\val(az_{1}),\val(b z_{2})\}, \quad  \val(\eta_v z^v) > \min \{\val(cz_{1}),\val(d z_{2})\},
\]
which implies
\begin{equation}\label{eqn:lambdaetageneral}
\val(\lambda_v)   \geq \min \{ \val(a), \val(b) \} = \val(a) >0, \quad \val(\eta_v) \geq \min \{  \val(c), \val(d) \} =\val(d) >0.
\end{equation}
If we consider a restriction of $F$ to $(T^\epsilon \Lambda_0)^2$, then under the given conditions, this restriction maps to $T^{a+\epsilon} \Lambda_0 \times T^{d+\epsilon} \Lambda_0$.

We claim that $F(z_{1},z_{2}) = (C_{1},C_{2})$ has a unique solution for any $C_{1}$ and $C_{2}$ satisfying $\val(C_{1}) > \val(a) + \epsilon$ and $\val(C_{2}) > \val(d) +\epsilon$. 
To see this, let us first make the following linear coordinate change
\[
	w_{1} := az_{1} + bz_{2},  \quad w_{2} := cz_{1} + dz_{2}
\]
so that
\[
	\val(w_{1}) \geq \min \{\val(z_{1}) + \val(a), \val(z_{2}) + \val( b)\}  \geq \val(a) + \epsilon
\]
and likewise, 
\[
	\val(w_{2}) \geq  \val(d) + \epsilon.
\]
Therefore $(w_{1} (z ), w_{2}(z ))$ defines an isomorphism between 
$(T^\epsilon \Lambda_0)^2$ and $T^{a+\epsilon} \Lambda_0 \times T^{d+\epsilon} \Lambda_0$.
In fact, the inverse coordinate change is obviously given as 
\[
	z_{1} = \frac{d w_{1} -b w_{2}}{ ad-bc},\quad z_{2} = \frac{-c w_{1} + a w_{2}}{ ad-bc}
\]
and one can easily check that 
\begin{equation}\label{eqn:invertw}
\val\left( \frac{dw_{1} }{ ad-bc} \right), \,\, \val\left( \frac{bw_{2} }{ad-bc}\right), \,\, \val\left( \frac{cw_{1}}{ ad -bc}\right), \,\, \val\left( \frac{aw_{2}}{ ad-bc}\right) \geq \epsilon,
\end{equation}
and hence $F(w_{1},w_{2})$ (in new coordinates) gives a map from $T^{a+\epsilon} \Lambda_0 \times T^{d+\epsilon} \Lambda_0$ to itself. 
Notice that non-degeneracy of the leading order of $F$ is crucially used here in this coordinate change.

%

We want to find a solution for $F(w_{1},w_{2}) = (C_{1},C_{2}) \in T^{a+\epsilon} \Lambda_0 \times T^{d+\epsilon} \Lambda_0$. To this end, we introduce another function $G : T^{a+\epsilon} \Lambda_0 \times T^{d+\epsilon} \Lambda_0 \to T^{a+\epsilon} \Lambda_0 \times T^{d+\epsilon} \Lambda_0$ defined by
\begin{align*}
G(w_{1},w_{2})&:= F(-w_{1},-w_{2}) - (C_{1},C_{2}) + (w_{1},w_{2})\\
&= -(C_{1},C_{2}) + \left(\sum_{v \neq (1,0),(0,1)} \tilde{\lambda}_v w^v, \sum_{v \neq (1,0),(0,1)} \tilde{\eta}_v w^v \right).
\end{align*}
We show that $G$ is a contraction. Since the constant terms play no role, we will assume $(C_{1},C_{2})=(0,0)$ from now on. We will only deal with the first component; the second component can be dealt with similarly.
A typical term appearing in the expansion $\lambda_v z^v$ in $(w_{1},w_{2})$ is of the form
\[
	m(w_{1},w_{2}):=
	\lambda_v 
	\left( \frac{dw_{1}}{ad-bc}\right)^i 
	\left( \frac{bw_{2}}{ad-bc}\right)^j 
	\left( \frac{cw_{1}}{ad-bc}\right)^k 
	\left( \frac{aw_{2}}{ad-bc}\right)^l  
\]
with $i+j+k+l \geq 2$, omitting the binomial coefficients in $\mathbb{Z}$. 

For simplicity, we put $\tilde{a}:=\frac{a}{ad-bc}$, $\tilde{b}:=\frac{b}{ad-bc}$, $\tilde{c}:=\frac{c}{ad-bc}$ and $\tilde{d}:=\frac{d}{ad-bc}$. Then \eqref{eqn:invertw} translates to
\begin{equation}\label{eqn:winvertnew}
\val(\tilde{c} w_{1}), \,\, \val(\tilde{d} w_{1}), \,\, \val(\tilde{a} w_{2}),  \,\, \val(\tilde{b} w_{2}) \geq \epsilon >0.
\end{equation}

Now, for two elements $(w_{1},w_{2})$ and $(w_{1}',w_{2}')$ in $T^{a+\epsilon} \Lambda_0 \times T^{d+\epsilon} \Lambda_0$, we have
\begin{align*}
 m(w_{1},w_{2}) - m(w_{1}',w_{2}') =& \lambda_v (\tilde{d} w_{1})^i (\tilde{c}^k w_{1})^k (\tilde{b} w_{2})^j (\tilde{a} w_{2})^l - \lambda_v (\tilde{d} w_{1}')^i (\tilde{c}^k w_{1}')^k (\tilde{b} w_{2}')^j (\tilde{a} w_{2}')^l \\
 =& (w_{2}-w_{2}') \sum_\beta \lambda_v (\tilde{d} w_{1})^i (\tilde{c} w_{1})^{k} \cdot  \tilde{b}^j \tilde{a}^l w_{2}^\beta (w_{2}')^{j+l-1-\beta} \\
 & + (w_{1}-w_{1}') \sum_\beta \lambda_v \tilde{d}^i \tilde{c}^k w_{1}^\alpha  (w_{1}')^{i+k-1-\alpha} \cdot  (\tilde{b} w_{2}')^j (\tilde{a} w_{2}')^l
 \end{align*}
Let us estimate the valuation of each of summands in the last equation. We first find a lower bound for the valuation $\nu_{1}$ of the first summand.  
If $j=l=0$, then it is obvious from \eqref{eqn:winvertnew} that 
\begin{equation}\label{eqn:nu1est}
\nu_{1} \geq \val(w_{2} - w_{2}') + \epsilon.
\end{equation}
Suppose now that $j \geq 1$, while $l$ can be possibly zero.
Observe that $\val(\tilde{b}^{i_{1}} \tilde{a}^{i_{2}} w_{2}^{j_{1}} (w_{2}')^{j_{2}}) \geq \epsilon$ as long as $i_{1} + i_{2} \leq j_{1} + j_{2}$ due to \eqref{eqn:winvertnew}. On the other hand, using the given condition on $b$ an \eqref{eqn:lambdaetageneral}, we have
\[
	\val (\lambda_v  \tilde{b} )  \geq \val (a) + \val (\tilde{b}) \geq \val(a) + \val(b) - \val(a) -\val(d) \geq 0.
\]
Hence rewriting the first summand as
\[
	(w_{2} - w_{2}') \left(\lambda_v \tilde{b}\right) 
	(\tilde{d} w_{1})^i (\tilde{c} w_{1})^{k} \cdot  \tilde{b}^{j-1} \tilde{a}^l w_{2}^\beta (w_{2}')^{j+l-1-\beta},
\]
we see that the estimate \eqref{eqn:nu1est} holds for $j \geq 1$.

Finally, if $j =0$ but $l \geq 1$, then $m (w_{1},w_{2})$ should come from $\lambda_v z^v$ with $z^v$ divisible by $z_{2}$. Thus \eqref{eqn:lambdaetamixed} tells us that $\val(\lambda_v) \geq \val(d)$ in this case. A similar argument as in the previous paragraph using $\val(\lambda_v \tilde{a}) \geq 0$ leads to the same inequality \eqref{eqn:nu1est} for $\nu_{2}$. 
The valuation $\nu_{2}$ of the second summand can be estimated in the same way (using $\val(\lambda_v \tilde{d}) \geq 0$ in this case), resulting in:
\[
	\nu_{2} \geq \val(w_{1} - w_{1}') + \epsilon.
\]

In summary, we have
\[
	\val (m(w_{1},w_{2}) - m(w_{1}',w_{2}') ) \geq \min(\nu_{1},\nu_{2}) \geq \min(\val(w_{1}-w_{1}'),\val(w_{2}-w_{2}')) + \epsilon,
\]
and we conclude that
\begin{align*}
||G(w_{1},w_{2}) - G(w_{1}',w_{2}')|| &= e^{- \val( G(w_{1},w_{2}) - G(w_{1}',w_{2}') )} \\
&\leq e^{- \min_m \{\val (m(w_{1},w_{2}) - m(w_{1}',w_{2}') )\}} \\
&\leq e^{-\min(\val(w_{1}-w_{1}'),\val(w_{2}-w_{2}')) - \epsilon} \\
&= e^{-\epsilon} || (w_{1},w_{2}) - (w_{1}',w_{2}')||
 \end{align*}
which is as desired.

\end{proof}

\bibliographystyle{amsalpha}
\bibliography{geometry}

\providecommand{\bysame}{\leavevmode\hbox to3em{\hrulefill}\thinspace}
\providecommand{\MR}{\relax\ifhmode\unskip\space\fi MR }
\providecommand{\MRhref}[2]{%
  \href{http://www.ams.org/mathscinet-getitem?mr=#1}{#2}
}
\providecommand{\href}[2]{#2}
\begin{thebibliography}{BECHL21}

\bibitem[Abo14]{Ab-famFl1}
M.~Abouzaid, \emph{Family {F}loer cohomology and mirror symmetry}, Pro. Int.
  Congr. Math. Seoul \textbf{II} (2014), 813--836.

\bibitem[Aur07]{auroux07}
D.~Auroux, \emph{Mirror symmetry and {$T$}-duality in the complement of an
  anticanonical divisor}, J. G\"okova Geom. Topol. GGT \textbf{1} (2007),
  51--91.

\bibitem[Aur09]{auroux09}
\bysame, \emph{Special {L}agrangian fibrations, wall-crossing, and mirror
  symmetry}, Surveys in differential geometry. {V}ol. {XIII}. {G}eometry,
  analysis, and algebraic geometry: forty years of the {J}ournal of
  {D}ifferential {G}eometry, Surv. Differ. Geom., vol.~13, Int. Press,
  Somerville, MA, 2009, pp.~1--47.

\bibitem[Bay04]{Bayer}
Arend Bayer, \emph{Semisimple quantum cohomology and blowups}, Int. Math. Res.
  Not. (2004), no.~40, 2069--2083. \MR{2064316}

\bibitem[BECHL21]{BCHL}
Sam Bardwell-Evans, Man-Wai~Mandy Cheung, Hansol Hong, and Yu-Shen Lin,
  \emph{Scattering diagrams from holomorphic discs in log calabi-yau surfaces},
  2021.

\bibitem[BIMS15]{BIMS}
Erwan Brugall\'{e}, Ilia Itenberg, Grigory Mikhalkin, and Kristin Shaw,
  \emph{Brief introduction to tropical geometry}, Proceedings of the
  {G}\"{o}kova {G}eometry-{T}opology {C}onference 2014, G\"{o}kova
  Geometry/Topology Conference (GGT), G\"{o}kova, 2015, pp.~1--75. \MR{3381439}

\bibitem[Cho05]{cho05}
C.-H. Cho, \emph{Products of {F}loer cohomology of torus fibers in toric {F}ano
  manifolds}, Comm. Math. Phys. \textbf{260} (2005), no.~3, 613--640.
  \MR{2183959 (2006h:53094)}

\bibitem[CL13]{chan-lau}
K.~Chan and S.-C. Lau, \emph{Open {G}romov-{W}itten invariants and
  superpotentials for semi-{F}ano toric surfaces}, IMRN (2013),
  doi:10.1093/imrn/rnt050.

\bibitem[CO06]{CO}
C.-H. Cho and Y.-G. Oh, \emph{Floer cohomology and disc instantons of
  {L}agrangian torus fibers in {F}ano toric manifolds}, Asian J. Math.
  \textbf{10} (2006), no.~4, 773--814.

\bibitem[CPS22]{CPS}
Michael Carl, Max Pumperla, and Bernd Siebert, \emph{A tropical view on
  landau-ginzburg models}.

\bibitem[CW22]{FW}
Fran\c{c}ois Charest and Chris~T. Woodward, \emph{Floer cohomology and flips},
  Mem. Amer. Math. Soc. \textbf{279} (2022), no.~1372, v+166. \MR{4464438}

\bibitem[Dyc11]{Dy}
T.~Dyckerhoff, \emph{Compact generators in categories of matrix
  factorizations}, Duke Math. J. \textbf{159} (2011), no.~2, 223--274.

\bibitem[FOOO09]{FOOO}
K.~Fukaya, Y.-G. Oh, H.~Ohta, and K.~Ono, \emph{Lagrangian intersection {F}loer
  theory: anomaly and obstruction. {P}art {I} and {II}}, AMS/IP Studies in
  Advanced Mathematics, vol.~46, American Mathematical Society, Providence, RI,
  2009.

\bibitem[FOOO10]{FOOO-T}
\bysame, \emph{Lagrangian {F}loer theory on compact toric manifolds. {I}}, Duke
  Math. J. \textbf{151} (2010), no.~1, 23--174.

\bibitem[FOOO11]{FOOO-T2}
\bysame, \emph{Lagrangian {F}loer theory on compact toric manifolds {II}: bulk
  deformations}, Selecta Math. (N.S.) \textbf{17} (2011), no.~3, 609--711.

\bibitem[FOOO16]{FOOO10b}
\bysame, \emph{Lagrangian {F}loer theory and mirror symmetry on compact toric
  manifolds}, Asterisque (2016), no.~376.

\bibitem[GHK15]{GHK}
Mark Gross, Paul Hacking, and Sean Keel, \emph{Mirror symmetry for log
  {C}alabi-{Y}au surfaces {I}}, Publ. Math. Inst. Hautes \'{E}tudes Sci.
  \textbf{122} (2015), 65--168. \MR{3415066}

\bibitem[GHKK18]{GHKK}
Mark Gross, Paul Hacking, Sean Keel, and Maxim Kontsevich, \emph{Canonical
  bases for cluster algebras}, J. Amer. Math. Soc. \textbf{31} (2018), no.~2,
  497--608. \MR{3758151}

\bibitem[GPS10]{GrPS}
Mark Gross, Rahul Pandharipande, and Bernd Siebert, \emph{The tropical vertex},
  Duke Math. J. \textbf{153} (2010), no.~2, 297--362. \MR{2667135}

\bibitem[GW19]{GW}
Eduardo Gonz\'{a}lez and Chris~T. Woodward, \emph{Quantum cohomology and toric
  minimal model programs}, Adv. Math. \textbf{353} (2019), 591--646.
  \MR{3986375}

\bibitem[HK22]{HK}
Paul Hacking and Ailsa Keating, \emph{Homological mirror symmetry for log
  calabi-yau surfaces}, Geom. Topol. \textbf{26} (2022), 3747--3833.

\bibitem[HKL18]{HKL}
H.~Hong, Y.~Kim, and S.-C. Lau, \emph{Immersed two-spheres and syz with
  application to grassmannians}, arXiv preprint (2018).

\bibitem[KLS19]{KLS}
Yoosik Kim, Jaeho Lee, and Fumihiko Sanda, \emph{Detecting non-displaceable
  toric fibers on compact toric manifolds via tropicalizations}, Internat. J.
  Math. \textbf{30} (2019), no.~1, 1950003, 40. \MR{3916270}

\bibitem[Kou76]{Kush}
A.~G. Kouchnirenko, \emph{Poly\`edres de {N}ewton et nombres de {M}ilnor},
  Invent. Math. \textbf{32} (1976), no.~1, 1--31. \MR{419433}

\bibitem[Lin21]{Lin}
Yu-Shen Lin, \emph{Correspondence theorem between holomorphic discs and
  tropical discs on {K}3 surfaces}, J. Differential Geom. \textbf{117} (2021),
  no.~1, 41--92. \MR{4195752}

\bibitem[MS12]{McS}
Dusa McDuff and Dietmar Salamon, \emph{{$J$}-holomorphic curves and symplectic
  topology}, second ed., American Mathematical Society Colloquium Publications,
  vol.~52, American Mathematical Society, Providence, RI, 2012. \MR{2954391}

\bibitem[Orl09]{Orlov}
D.~Orlov, \emph{Derived categories of coherent sheaves and triangulated
  categories of singularities}, Algebra, arithmetic, and geometry: in honor of
  {Y}u. {I}. {M}anin. {V}ol. {II}, Progr. Math., vol. 270, Birkh\"auser Boston
  Inc., Boston, MA, 2009, pp.~503--531.

\bibitem[She16]{She}
Nick Sheridan, \emph{On the {F}ukaya category of a {F}ano hypersurface in
  projective space}, Publ. Math. Inst. Hautes \'{E}tudes Sci. \textbf{124}
  (2016), 165--317. \MR{3578916}

\bibitem[SYZ96]{SYZ96}
A.~Strominger, S.-T. Yau, and E.~Zaslow, \emph{Mirror symmetry is
  {$T$}-duality}, Nuclear Phys. B \textbf{479} (1996), no.~1-2, 243--259.

\bibitem[Tel20]{Tel}
Constantin Teleman, \emph{Matrix factorization of {M}orse-{B}ott functions},
  Duke Math. J. \textbf{169} (2020), no.~3, 533--549. \MR{4065148}

\bibitem[VWX20]{VXW}
Sushmita Venugopalan, Chris~T. Woodward, and Guangbo Xu, \emph{Fukaya
  categories of blowups}.

\bibitem[Yua20]{Yuan}
Hang Yuan, \emph{Family floer program and non-archimedean syz mirror
  construction}.

\end{thebibliography}
\end{document}